\documentclass[11pt,reqno]{amsart}
\usepackage[margin=1in]{geometry}
\usepackage{comment}
\usepackage{amssymb}
\usepackage{amsthm}
\usepackage{amsmath}
\usepackage{mathrsfs}
\usepackage{amsbsy}
\usepackage[all]{xy}
\usepackage{bm}
\usepackage{hyperref}
\usepackage{tikz}
\usetikzlibrary{patterns,math}
\usepackage{array}
\usepackage{float}
\usepackage{enumitem}
\usepackage{xcolor}
\usepackage{hhline}
\allowdisplaybreaks
\usepackage[noadjust]{cite}

\usepackage{caption}
\usepackage{tabu}
\usepackage{diagbox}

\DeclareMathOperator{\ct}{ct}
\DeclareMathOperator{\Grid}{Grid}
\DeclareMathOperator{\len}{len}
\DeclareMathOperator{\supp}{supp}
\newcommand{\TV}[1]{\|#1\|_{\operatorname{TV}}}

\usepackage[noabbrev,nosort,capitalise]{cleveref}

\crefname{enumi}{Theorem}{Theorems}
\Crefname{enumi}{Theorem}{Theorems}

\makeatletter
\renewcommand{\p@enumi}{\thetheorem}
\makeatother

\usepackage{thmtools}

\newenvironment{enumerate*}%
  {\begin{enumerate}[(I)]%
    \setlength{\itemsep}{10pt}%
    \setlength{\parskip}{0pt}}%
  {\end{enumerate}}

\newtheorem{theorem}{Theorem}[section]
\newtheorem{proposition}[theorem]{Proposition}
\newtheorem{corollary}[theorem]{Corollary}

\newtheorem{question}[theorem]{Question}

\newtheorem{lemma}[theorem]{Lemma}

\theoremstyle{definition}

\newtheorem{remark}[theorem]{Remark}
\newtheorem{example}[theorem]{Example}
\newtheorem{construction}{Construction}

\numberwithin{equation}{section}

\newcommand{\A}{\mathsf{A}}
\newcommand{\B}{\mathsf{B}}
\newcommand{\C}{\mathsf{C}}

\newcommand{\CC}{\mathbb C}
\newcommand{\RR}{\mathbb R}
\newcommand{\ZZ}{\mathbb Z}

\newcommand{\eps}{\varepsilon}

\newcommand{\be}{\mathbf e}
\newcommand{\bs}{\mathbf s}
\newcommand{\bt}{\mathbf t}
\newcommand{\bu}{\mathbf u}
\newcommand{\bv}{\mathbf v}
\newcommand{\bx}{\mathbf x}
\newcommand{\zero}{\mathbf 0}

\newcommand*{\EE}{
  \mathop{
    \mathchoice{\vcenter{\hbox{\larger[4]$\mathbb{E}$}}}
               {\kern0pt\mathbb{E}}
               {\kern0pt\mathbb{E}}
               {\kern0pt\mathbb{E}}
  }\displaylimits
}

\newcommand{\Zmod}[1]{\ZZ/{#1}\ZZ}
\begin{document}
\newcounter{segct}

\newcommand{\localDiagram}[1]{


    \addtocounter{segct}{1};

    \def\segmentSet{}
    \foreach \x/\y/\hx/\hy/\l/\loc in {#1} {
        \tikzmath{\leftx = 2*\x-1;}
        \tikzmath{\rightx = 2*\x+1;}
        \tikzmath{\bottomy = 2*\y-1;}
        \tikzmath{\topy = 2*\y+1;}
        \foreach \xl/\yl/\xr/\yr in 
        {\leftx/\bottomy/\rightx/\bottomy,
         \rightx/\bottomy/\rightx/\topy,
         \leftx/\topy/\rightx/\topy,
         \leftx/\bottomy/\leftx/\topy} {
            \ifcsname segment@\thesegct-\xl-\yl-\xr-\yr\endcsname
            \else
                \expandafter\gdef\csname segment@\thesegct-\xl-\yl-\xr-\yr\endcsname{} 
                \draw (\xl, \yl) -- (\xr, \yr);
            \fi
        }
    }     
    
    \foreach \x/\y/\hx/\hy/\l/\loc in {#1} {
    

        \node at (2*\x, 2*\y) [circle,fill,inner sep=1.5pt]{};

        \tikzmath{\nm = ((\hx)^2+(\hy)^2)^(1/2);}
        \tikzmath{\hxs = 0.8*\hx/\nm;}
        \tikzmath{\hys = 0.8*\hy/\nm;}
        \draw[->] (2*\x, 2*\y) -- (2*\x+\hxs,2*\y+\hys);

        \tikzmath{\ldist = 1.5*\nm^2;}
        \node[label={[label distance=-\ldist mm]\loc:{\tiny \l}}] at (2*\x+\hxs,2*\y+\hys) {};
    }
}

\title[]{Words with Repeated Letters in a Grid}

\author{Zachary Halberstam}
\address{Department of Pure Mathematics and Mathematical Statistics, University of Cambridge}
\email{zbh23@cam.ac.uk}
\author{Carl Schildkraut}
\address{Department of Mathematics, Stanford University}
\email{\texttt {carlsch@stanford.edu}}

\begin{abstract}
    Given a word $w$, what is the maximum possible number of appearances of $w$ reading contiguously along any of the directions in $\{-1, 0, 1\}^d \setminus \{\mathbf{0}\}$ in a large $d$-dimensional grid (as in a word search)? Patchell and Spiro first posed a version of this question, which Alon and Kravitz completely answered for a large class of ``well-behaved" words, including those with no repeated letters. We study the general case, which exhibits greater variety and is often more complicated (even for $d=1$). We also discuss some connections to other problems in combinatorics, including the storied $n$-queens problem.
\end{abstract}

\maketitle

\section{Introduction}\label{sec:intro}
In their article \cite{patchellspiro}, Patchell and Spiro pose and partially answer the following question: Suppose we have an $n$ by $n$ grid and a word $w$ of length $k$. We would like to place a letter in each square of the grid to maximize the number of copies of $w$ reading in the directions of a standard word search; that is, vertically, horizontally, or diagonally. What is the maximum number of copies of $w$ we can achieve? 

In \cite{alonkravitz}, Alon and Kravitz answer this question asymptotically in the case of words with distinct letters. They address the problem for the case of a grid with modular shape $\mathbb{Z}/n_1\mathbb{Z}\times\ZZ/n_2\ZZ$, noting that this makes no difference in the asymptotic as the $n_i$ grow large. They show that for such a word $w = w_1 \cdots w_\ell$ the (not necessarily unique) optimal construction is to let each row read
\begin{equation}\label{dim1distinct}
\cdots  w_2  w_1  w_2  \cdots   w_{\ell-1} w_\ell  w_{\ell-1} \cdots  w_2  w_1 w_2 \cdots
\end{equation}
stacked without any offset. They provide a similar result in arbitrary dimension, as well as addressing a narrow class of words with repeated letters. 

In this article, we further investigate this question for words with repeated letters. We will be especially interested in characterizing when, as in \cite[Theorem~1]{alonkravitz}, there exist optimal constructions in $d$ dimensions which are constant in all but one coordinate.

We first quickly introduce some terminology and notation; see \cref{sec:definition} for more precise definitions. The \emph{concentration} of a word $w$ in a $d$-dimensional grid $\Gamma$ of letters is the number of copies of $w$ in $\Gamma$, reading along one of the directions of a standard word search (that is, directions in which the increment in each coordinate is chosen from $\{-1,0,1\}$), divided by the size of $\Gamma$. Given a word $w$ and a dimension $d$, let $C_d(w)$ be the supremum concentration of $w$ among all $d$-dimensional grids. We restrict ourselves to words which contain some two distinct letters.

\begin{question}\label{qn:main} Given a word $w$ and a dimension $d$, what is $C_d(w)$?    
\end{question}

We are able to answer \cref{qn:main} in many regimes. 

\subsection{Introductory results}\label{subsec:intro-results}

We begin our discussion with the $d=1$ case, in which we are able to give a complete answer to \cref{qn:main}.

\begin{theorem}[Resolving the one-dimensional case]\label{thm:1d} Let $w$ be any word. There exists a closed form\footnote{We state this closed form in \cref{prop:1d} and the classification of extremal grids in \cref{prop:1d-structure}.} for $C_1(w)$ and a simple classification of all extremal grids. 
\end{theorem}

\noindent Given an understanding of the one-dimensional case, we provide simple bounds on $C_d(w)$ for all $d$ in terms of $C_1(w)$.

\begin{proposition}[General bounds on $C_d(w)$]\label{prop:simple-d-bound} For any dimension $d$ and any word $w$,
\[3^{d-1}C_1(w)\leq C_d(w)\leq\frac{3^d-1}2C_1(w).\]
\end{proposition}

This proposition motivates the following definition. For a positive integer $d$, we call a word \emph{$d$-stackable} if the lower bound of \cref{prop:simple-d-bound} is tight, and \emph{$d$-slantable} if the upper bound is tight.\footnote{This terminology is inspired by the proof of \cref{prop:simple-d-bound}: the lower bound is tight if and only if the densest constructions in $d$ dimensions come from ``stacking'' the $1$-dimensional optimum, while the upper bound is tight if and only if, in every direction (``slant'') in which we count copies of $w$, the concentration of $w$ is the same as $C_1(w)$.} In our language, the main result of \cite{alonkravitz} describes a family of words which are $d$-stackable for every $d$, arising from the $d$-stackability of the single word $\A\B$. 

\subsection{Main results}\label{subsec:results} 

Our first class of main results concerns $d$-stackability.

\begin{theorem}\label{thm:stack} The notion of $d$-stackability enjoys the following properties:
\begin{enumerate}[label={(\alph*)}]
    \item\label{thm:odd-even} (Odd-even words) If $w$ is a word in which no letter appears at both an odd and even index, then $w$ is $d$-stackable for all $d$.
    
    \item\label{thm:AkBk} (Letter-repeated words) If $w$ is $d$-stackable and $k$ is a positive integer, then the word $w'$ formed by repeating each letter in $w$ $k$ times is $d$-stackable.
    
    \item\label{thm:short-words} (Short words) If $w$ is a word with at most four characters, then $w$ is $2$-stackable if and only if $w$ is not of the form $ABBB$ or its reverse for some letters $A$ and $B$.
\end{enumerate}
In fact, the non-$2$-stackability of $\mathsf{ABBB}$ is part of a more general phenomenon:
\begin{enumerate}[label={(\alph*)},resume]
    \item\label{prop:ABell-1} (The word $\A\B\cdots\B$) For $\ell>3$, the word $\A\B^{\ell-1}$ is not $2$-stackable.
\end{enumerate}
\end{theorem}

\cref{thm:odd-even} generalizes the main result of \cite{alonkravitz}, and uses the $d$-stackability of $\A\B$ (as proven in that work) as input. \cref{thm:short-words} uses \cref{thm:odd-even,thm:AkBk} as input, as well as a couple of other ``base cases,'' and is included to demonstrate the wide-ranging nature of our techniques in resolving \cref{qn:main} in the $d=2$ setting. \cref{prop:ABell-1} presents an infinite class of words which are not 2-stackable; in this sense, 2-stackability is not strictly the ``generic" behavior. 

Our second class of main results concerns $d$-slantability. In contrast with $d$-stackability, which can be achieved for all $d$, no word (excepting the trivial case of words with only one distinct letter) can be $d$-slantable for every $d$. 

\begin{theorem}\label{thm:slant} The notion of $d$-slantability enjoys the following properties:
\begin{enumerate}[label={(\alph*)}]
    \item\label{thm:no-slant} A word $w$ of length $\ell$ is not $d$-slantable for any $d\geq 8\log_2\ell+47$.

    \item\label{prop:queen-words} If $\ell>1$ and all prime factors of $\ell$ are at least $2^d$, then the word $w=\A\B^{\ell-1}$ is $d$-slantable.
\end{enumerate}
\end{theorem}

\noindent By Bertrand's postulate, there exists for each $d$ some prime $\ell$ with $2^d<\ell<2^{d+1}$. Therefore, \cref{prop:queen-words} implies that \cref{thm:no-slant} is tight up to a constant factor.

\cref{prop:ABell-1,prop:queen-words} are closely related to the (modular) \emph{$n$-queens problem}; see \cref{sec:not-stack} for some discussion of the connection.

\section{Definitions, techniques, and organization}\label{sec:definition}

\subsection{Definitions and notation}\label{subsec:definition}

We begin by giving formal definitions to the concepts presented in the introduction, as well as some other definitions and notation we will need to fix for the rest of the paper.

\begin{itemize}
    \item An \emph{alphabet} is a set of letters. 
    \item A \emph{word} is a finite string of letters from some alphabet $\mathsf\Sigma$. The \emph{length} of $w$, denoted $\len(w)$, is the number of letters in $w$.
    \begin{itemize}
        \item Two words are \emph{isomorphic} if one can be obtained from the other by applying a bijection between the alphabets of the two words (only considering the letters of each alphabet which appear in the respective words) and then (possibly) reversing the words. For example, $\mathsf{AAB}$, $\mathsf{BEE}$, and $\mathsf{EEL}$ are all isomorphic to each other, while $\mathsf{ABA}$ and $\mathsf{AAB}$ are not.
        
        \item A \emph{subword} of $w$ is any string of consecutive letters in $w$. We call a subword of $w$ \emph{proper} if it is not $w$ itself, and we call a subword a \emph{prefix} (resp.\ \emph{suffix}) of $w$ if it starts (resp.\ ends) $w$.
        
        \item Given words $w$ and $v$ and a nonnegative integer $m$, we write $wv$ for the concatenation of $w$ and $v$, $w^n$ for the concatenation of $n$ copies of $w$ together, and $w^{\mathrm{rev}}$ for the reverse of $w$.

        \item We generally exclude the trivial case of a word with only one distinct letter.
    \end{itemize}

    \item A \emph{grid} $\Gamma$ is a function $G\to\mathsf\Sigma$ where $\mathsf\Sigma$ is some alphabet and $G=\prod_{i=1}^d\Zmod{n_i}$ for some sequence $(n_1,\ldots,n_d)$ of positive integers.\footnote{$G$ includes the data of the values of $n_i$ in order; it is not an abstract set or group.} We call $G$ the \emph{shape} of $\Gamma$, $|G|$ the \emph{size} of $\Gamma$ (which we sometimes denote $|\Gamma|$), and $d$ the \emph{dimension} of $\Gamma$. We think of a grid $\Gamma$ as a labeling of the cells of a toroidal $d$-dimensional ``box'' $[0,n_1)\times[0,n_2)\times\cdots\times[0,n_d)$ by elements of $\mathsf\Sigma$.
    \begin{itemize}
        \item Given a word $w$, $\Grid(w)$ denotes the one-dimensional grid of shape $\ZZ/\len(w)\ZZ$ defined by sending $0\leq i<\len(w)$ to the $i$th letter of $w$.
        
        \item An \emph{appearance} of $w$ in a grid $\Gamma$ of shape $G=\prod_{i=1}^d\Zmod{n_i}$ is a pair 
        \[(p,\bv)\in G\times(\{-1,0,1\}^d\setminus\{\zero\})\]
        such that the letters in $\Gamma$ at positions $(p+i\bv)_{i=0}^{\len(w)-1}$ are the letters of $w$, in order. That is, an appearance of $w$ is a copy of $w$ in $\Gamma$ reading in one of the standard ``word-search'' directions, extended to $d$ dimensions in a natural way. We call the number of appearances of $w$ in $\Gamma$ the \emph{count} of $w$ in $\Gamma$ and denote it by $\ct(w, \Gamma)$. We denote the set of appearances of $w$ in $\Gamma$ by $\mathcal A(w, \Gamma)$. 

        \item The \emph{concentration} of $w$ in a grid $\Gamma$ is the number of appearances of $w$ in $\Gamma$ divided by the size of $\Gamma$. We denote this by $c_d(w,\Gamma)$, where $d$ is the dimension. Note that $c_d(w,\Gamma)$ may exceed $1$.
        
        \item We may think of a grid $\Gamma$ of shape $\prod_{i=1}^d\Zmod{n_i}$ as arising from an ``infinite grid,'' the map $\ZZ^d\to\mathsf\Sigma$ given by precomposing $\Gamma$ with the quotient map $\ZZ^d\to\prod_{i=1}^d\Zmod{n_i}$. This is the \emph{infinite grid corresponding to} $\Gamma$. We call two grids $\Gamma$ and $\Gamma'$ \emph{equivalent} if they have the same dimension, and the infinite grids corresponding to $\Gamma$ and $\Gamma'$ can be taken to each other by precomposing with some isometry of $\ZZ^d$. Equivalently, two grids are equivalent if one can be transformed into the other by a sequence of the following operations:
        \begin{itemize}
            \item \emph{Enlarging} the grid in some coordinate $i$: precomposing the map $\Gamma\colon\prod_{i=1}^d\Zmod{n_i}\to\mathsf\Sigma$ with the quotient $\Zmod{kn_i}\to\Zmod{n_i}$ on the $i$th coordinate.
            \item \emph{Contracting} the grid in some coordinate $i$: the inverse of an enlarging map, if applicable.            \item \emph{Reversing} some coordinate $i$: precomposing $\Gamma$ by the negation map in some coordinate.            \item \emph{Translating} the grid: precomposing $\Gamma$ by the map $x\mapsto x+t$ for some $t\in\prod_{i=1}^d \Zmod{n_i}$.
            \item \emph{Swapping} two coordinates $i$ and $j$: precomposing on the $i$th and $j$th coordinates by the map $\Zmod{n_j}\times\Zmod{n_i}\to\Zmod{n_i}\times\Zmod{n_j}$ which interchanges the two factors.
        \end{itemize}
        If two grids $\Gamma$ and $\Gamma'$ of dimension $d$ are equivalent, then $c_d(w,\Gamma)=c_d(w,\Gamma')$ for every word $w$.

        \item Given a grid $\Gamma$ on alphabet $\mathsf\Sigma$ (or, more generally, any function $\Gamma$ to $\mathsf\Sigma$ with finite domain), the \emph{letter distribution} of $\Gamma$ is the function $h_\Gamma\colon\mathsf\Sigma\to[0,1]$ defined such that $h_\Gamma(A)$ is $\frac1{|\Gamma|}$ times the number of times $\Gamma$ takes the value $A$. We equip the space of functions on $\mathsf\Sigma$ with the \emph{total variation distance}, defined by
\[\TV{h_1-h_2}=\frac12\sum_{A\in\mathsf\Sigma}|h_1(A)-h_2(A)|.\]
        
        \item For each grid $\Gamma$ of shape $G$, with respect to each vector $\bv\in\{-1, 0, 1\}\setminus \{\zero\}$, $G$ can be partitioned into sets of the form $S = \{p+ n\bv: n\in \mathbb{N}\}$. We define the \emph{search lines} of $\Gamma$ in direction $\bv$ to be the one-dimensional grids $\Delta$ with shape $\Zmod{(|S|)}$ and $\Delta(t) = \Gamma(p+t\bv)$. Up to equivalence of one-dimensional grids, the ``starting point'' $p$ of a search line does not matter, nor does the choice of direction between $\bv$ and $-\bv$. The \emph{set of search lines} of $\Gamma$ consists of an arbitrary search line in each equivalence class. For any word $w$, the set of appearances of $w$ in $\Gamma$ is in bijection with the set of appearances of $w$ across the set of search lines in $\Gamma$.
        
        For example, in the grid $\Gamma$ of shape $\ZZ/2\ZZ\times\ZZ/3\ZZ$ given by
        \[\begin{matrix}
                \A & \B & \mathsf{C} \\
                \mathsf{D} & \mathsf{E} & \mathsf{F}
        \end{matrix}\]
        the set of search lines of $\Gamma$ is
        \begin{align*}
        \big\{\!&\Grid(\mathsf{AD}),\Grid(\mathsf{BE}),\Grid(\mathsf{CF}),\\
        &\Grid(\mathsf{ABC}),\Grid(\mathsf{DEF}),\\
        &\Grid(\mathsf{AECDBF}),\\
        &\Grid(\mathsf{AFBDCE})\big\}.
        \end{align*}
    \end{itemize}
\end{itemize}

The central object of study, $C_d(w)$, is defined by
\begin{equation}\label{eq:C-def}
C_d(w)=\sup_{\substack{\Gamma\text{ a grid}\\\dim\Gamma=d}}c_d(w, \Gamma).
\end{equation}
A grid $\Gamma$ is called $w$-\emph{extremal} (or simply \emph{extremal} when $w$ is clear) if $c_d(w, \Gamma)=C_d(w)$. Note that, if $w$ and $w'$ are isomorphic, then $C_d(w) = C_d(w')$ for all $d$.

In our notation, the main result of \cite{alonkravitz} can be stated as follows:

\begin{theorem}[{\cite[Theorem 3]{alonkravitz}}]\label{thm:AK-restated}
Let $w$ be a word of length $k$ with no letter appearing at both even and odd indices and no contiguous two-letter block appearing more than once in the word, either forward or backwards, and let $d$ be a positive integer. Then $w$ is $d$-stackable, and
\[C_d(w) = \frac{3^{d-1}}{k-1}.\]
\end{theorem}

\subsection{Higher dimensions from one dimension}\label{subsec:bounds-from-1d}

As stated in \cref{prop:simple-d-bound}, bounds on $C_d(w)$ can be obtained from $C_1(w)$. Now that we have defined all of our objects, we can provide a proof of this simple proposition.

First, we prove a slightly stronger lower bound result about density growth in dimension:

\begin{lemma}\label{lem:concentration-growth-lower-bound}
    For any dimension $d$ and any word $w$, $C_{d+1}(w)\ge 3C_d(w)$.
\end{lemma}

\begin{proof}
    Let $\Gamma$ be any $d$-dimensional grid, of shape $\prod_{i=1}^d \Zmod{n_i}$. Then, consider the $d+1$-dimensional grid $\Gamma'$ of shape $(\prod_{i=1}^d \Zmod{n_i}) \times \Zmod{1}$, defined such that $\Gamma'(t_1, \ldots, t_d, 0) = \Gamma(t_1, \ldots, t_d)$. That is, the letters in $\Gamma'$ are exactly the same as the letters in $\Gamma$, but $\Gamma'$ has an extra coordinate dimension. We are ``stacking" $\Gamma$ to get a grid constant in the $(d+1)$th coordinate.

    Now, suppose there is an appearance of $w$ at position $p$ in $\Gamma$, reading in the direction $\bv$. Then, there are three corresponding appearances of $w$ at position $p\times 0$ in $\Gamma'$, reading in the directions $\bv \times \{-1, 0, 1\}$. In fact every appearance of $w$ in $\Gamma'$ is of this form, since $w$ has more than one unique letter (recall that we exclude the trivial case where $w$ has only one distinct letter) and $\Gamma'$ is constant in the $(d+1)$th coordinate.

    Therefore $\ct(w, \Gamma') = 3\ct(w, \Gamma)$. Moreover $|\Gamma'| = |\Gamma|$, so $c_{d+1}(w, \Gamma') = 3c_d(w, \Gamma)$. Taking the supremum over $\Gamma$, this implies that $C_{d+1}(w)  \ge 3C_d(w)$ as desired.
\end{proof}

\begin{proof}[Proof of \cref{prop:simple-d-bound}]
The lower bound follows from applying \cref{lem:concentration-growth-lower-bound} $d-1$ times to get
\[C_d(w) \ge 3C_{d-1}(w)\geq 3^2C_{d-2}(w)\geq\cdots\geq3^{d-1} C_1(w).\]
To prove the upper bound, we look within each search line. Let $\Gamma$ be any grid in $d$ dimensions. The total size of all search lines reading in each unsigned direction is $|\Gamma|$. Therefore, the total number of appearances of $w$ across all search lines in each unsigned direction is at most $C_1(w)|\Gamma|$. Since there are $\frac{3^d-1}{2}$ unsigned directions, the total number of appearances of $w$ across the set of search lines is then at most $C_1(w)\frac{3^d-1}{2}|\Gamma|$. Recalling that the appearances of $w$ across the set of search lines of $\Gamma$ are in bijection with the appearances of $w$ in $\Gamma$ shows that $c_d(w, \Gamma) \le \frac{3^d-1}{2}C_1(w)$, which gives the desired upper bound.
\end{proof}

We will state many results in terms of how $C_d(w)$ and $C_1(w)$ relate. An important narrowing of \cref{qn:main} is the following:

\begin{question}\label{qn:d-stackable}
For which $w$ and $d$ does it hold that $w$ is $d$-stackable?
\end{question}

Note in particular that the answer to this question is affirmative in every dimension for every word for which \cref{thm:AK-restated} of \cite{alonkravitz} applies.

\vspace{2mm}

In proving our results, we introduce and develop a number of techniques relevant to \cref{qn:main}. It is in terms of these techniques that this paper is organized. We give a short description of each below.

\subsection{Additive formulation}\label{subsec:intro-additive} Our first technique is to formulate \cref{qn:main} as a question of set addition.

Let $\Gamma$ be a $d$-dimensional grid of shape $G$, and let $w$ be a word. For each $\bv\in\{-1,0,1\}^d\setminus\{\zero\}$, let $S_\bv\subset G$ denote the set of $p\in G$ for which $(p,\bv)$ is an appearance of $w$, i.e.~such that $\Gamma(p+i\bv)$ is the $i$th letter $w_i$ of $w$ for all $0\leq i<\len(w)$. For each pair $(\bu,\bv)$ of directions, define the sets
\[S_\bu-S_\bv:=\{\bs-\bt : \bs\in S_\bu,\bt\in S_\bv\};\qquad I_{\bu,\bv}:=\big\{i\bv-j\bu : 0\leq i,j<\len(w),w_i\neq w_j\big\}.\]
We have
\begin{equation}\label{eq:additive-avoidance}
(S_\bu-S_\bv)\cap I_{\bu,\bv}=\emptyset\text{ for all }\bu,\bv\in\{-1,0,1\}^d\setminus\{\zero\}.
\end{equation}
Indeed, for any solution to $\bs-\bt=i\bv - j\bu$ with $\bs\in S_\bu$, $\bt\in S_\bv$, and $0\leq i,j<\len(w)$, the letter at $\bs+j\bu=\bt+i\bv$ in $\Gamma$ must be both $w_j$ and $w_i$, and so $w_i=w_j$. Conversely, consider some sets $S_\bv\subset G$ for $\bv\in\{-1,0,1\}^d\setminus\{\zero\}\}$. As long as \eqref{eq:additive-avoidance} holds, there exists a grid $\Gamma$ for which, for all $\bv$, $(p,\bv)$ is an appearance of $w$ for all $p\in S_\bv$. To find such a grid, simply place letter $w_i$ at $p+i\bv$ for all $p\in S_\bv$, and once this is done fill in the remaining cells arbitrarily. Condition \eqref{eq:additive-avoidance} ensures that no two distinct letters are placed in the same cell. Finally, we have
\[c_d(w,\Gamma)=\sum_\bv\frac{|S_\bv|}{|G|}.\]
So, we can reformulate \cref{qn:main} as follows:
\begin{question}[{\cref{qn:main}}, additive formulation]\label{qn:main-additive} Given a word $w$ and a dimension $d$, define for each $\bu,\bv\in\{-1,0,1\}^d\setminus\{\zero\}$
\[I_{\bu,\bv}:=\big\{i\bv-j\bu : 0\leq i,j<\len(w),w_i\neq w_j\big\}.\]
Let $\mathcal A_d$ denote the set of finite rank-$d$ abelian groups. Determine
\[\sup_{G\in\mathcal A_d}\max_{\substack{S_\bv\subset G\\(S_\bu-S_\bv)\cap I_{\bu,\bv}=\emptyset}}\sum_\bv\frac{|S_\bv|}{|G|},\]
where the conditions in the maximum are over all $\bu,\bv\in\{-1,0,1\}^d\setminus\{\zero\}$.
\end{question}

\noindent In the $d=1$ case, this reformulation is particularly simple. We will use this in \cref{sec:1d} to prove \cref{thm:1d}.

We remark that, in the additive formulation, our main question is reminiscent of the \emph{periodic tiling conjecture}: given a finite set $T\subset\ZZ^d$ which tiles $\ZZ^d$ by translations, must it tile $\ZZ^d$ periodically? For a tile $T$, the set $S\subset\ZZ^d$ of translations must satisfy $S+T=\ZZ^d$; in particular, $S$ has ``density'' $1/|T|$, and $S-S$ avoids $(T-T)\setminus\{\zero\}$. This is of a similar shape to \cref{qn:main-additive}; the analogous question is the following:

\begin{question}\label{qn:extremal-exist} Given a word $w$ and a dimension $d$, does there exist a $w$-extremal grid? In other words, is the supremum in \eqref{eq:C-def} defining $C_d(w)$ attained?
\end{question}

We are able to answer \cref{qn:extremal-exist} in the affirmative when $d=1$; see \cref{cor:1d-extremal-exist}. Given the recent disproof of the periodic tiling conjecture by Greenfeld and Tao in \cite{greenfeldtao}, however, it is not unreasonable to propose that the answer may be negative for large $d$.

\subsection{Combinatorial reductions}\label{subsec:intro-reductions} In many cases, it is possible to extend results about the supremum density for one word to results for many more words. For instance, the crux of the argument in \cite{alonkravitz} is the proof (using Fourier analysis) that $C_d(\mathsf{AB}) = 3^{d-1}$, and that there exist $w$-extremal grids for all $d$ which are constant in all but one coordinate and alternate in this coordinate between $\A$ and $\B$. The proof of \cref{thm:AK-restated} then follows, essentially, from mapping longer words onto the word $\mathsf{AB}$.

The argument is sketched as follows. First, the upper bound on the concentration of the word $\mathsf{AB}$ yields the same upper bound on the concentration of consecutive odd letter-even letter pairs from $w$ in any grid. 
Next, a grid constant in all but one coordinate reading as in \eqref{dim1distinct} achieves this upper bound on the concentration of consecutive odd letter-even letter pairs. Moreover, since each odd-even pair can only be used in one appearance of $w$ (by the restrictions on $w$ specified in \cref{thm:AK-restated}), each appearance of $w$ requires $\len(w) - 1$ odd-even pairs, and this labeling achieves one appearance of $w$ per every $\len(w)-1$ odd-even pairs, Alon and Kravitz conclude \cref{thm:AK-restated}.

Alon and Kravitz suggest related results may be possible for all words with no letters appearing at both even and odd indices, removing the restriction on odd letter-even letter pairs that appear in more than one word. One example of such a word is $\mathsf{ELEPHANT}$, where the pair $\mathsf{EL}$ appears both forwards and backwards.

We will develop a combinatorial framework for reducing cases of \cref{qn:main} for one word $w$ to other words for which the answer is already known. This framework will allow us to show \cref{thm:odd-even}, that words with no letters appearing at both even and odd indices are $d$-stackable for every $d$, as well as many other cases of \cref{thm:short-words}.

\subsection{Local analysis and linear programming}\label{subsec:intro-local}

If one restricts oneself to grids of a fixed shape, maximizing the concentration of $w$ is a finite problem, since there are finitely many grids of a given shape on a given alphabet. In principle, this fact can be used to approximate $C_d(w)$ for any word $w$ and any $d$: one can simply compute the optimal concentration of $w$ in a grid of shape $(\ZZ/n\ZZ)^d$ for successively larger $n$.

This procedure can be generalized to provide upper bounds on $c_d(w,\Gamma)$ independent of the shape of $\Gamma$. We cannot directly use bounds on $c_d(w,\Gamma)$ for small $\Gamma$ to bound $C_d(\Gamma)$ (at least in a tight way), since the geometry of small grids contains the ``wraparound'' of $\Zmod{n}$. So, small modular grids do not neatly fit into larger grids. However, we can restrict focus to a set $S$ of cells in a grid without wraparound (in essence, a set of cells within the infinite grid $\ZZ^d$), compute the maximum number of appearances of $w$ within these cells, and then average this maximum number over all isometric images of $S$ in a grid $\Gamma$. For technical reasons (and to obtain better bounds), we will give each of the possible appearances of $w$ within $S$ a different weight. See \cref{subsec:local-AB} for a motivating example.

Unlike the methods described in previous sections, this ``local'' method has many parameters to fine-tune, and choosing optimal values for these parameters is a computationally time-consuming procedure. If the set $S$ chosen is too small, the bounds we attain on $C_d(w)$ will often be weak or trivial. However, if the set $S$ is too large, carrying out the computation (even with fixed parameter values) to bound $C_d(w)$ is too costly. We explain some optimizations we have applied in \cref{subsec:local-philosophy}. There are only a few pairs $(w,d)$ not covered by our previous techniques for which we have been able to attain tight upper bounds on $C_d(w)$ using the local method. However, among these pairs is $(\mathsf{ABB},2)$, which is a crucial case of \cref{thm:short-words} and one we do not know how to treat using other methods.

Although our linear programming approach is not sourced from any particular reference, we remark that the use of linear (or, more generally, semidefinite) programming to study extremal problems is not new. The most developed example of this is the study of flag algebras \cite{Razborov}, a technique based in semidefinite programming which has been instrumental in understanding the relationships between subgraph densities in graphs. Linear programming and the study of finite substructures have also been used fruitfully in understanding geometric problems such as the density of sets without unit distances \cite{Szekely,ACMVZ} and of spherical sets without pairs of orthogonal points \cite{decortepikhurko,bkov}, as well as in additive combinatorics \cite{lmpv}. Our technique thus considered follows in a long line of work, and from this perspective the utility of such an approach should not be surprising. What may be surprising, however, is that we are able to obtain \emph{tight} bounds in some cases with our linear programming approach. In many of the referenced examples, such bounds seem out of reach with purely local techniques.

\subsection{Fourier analysis}\label{subsec:intro-fourier}

Given that \cref{qn:main} can be phrased purely in terms of additive combinatorics, it is perhaps not surprising that Fourier analysis is a useful tool.

In the language of spectral graph theory, Alon and Kravitz applied Fourier analysis to prove \cref{thm:AK-restated}, in particular to prove that $\mathsf{AB}$ is $d$-stackable for all $d$. In this paper, we use Fourier analysis when the question we aim to answer has some nice algebraic structure, or when a nice algebraic structure can be isolated within it. 

One result we prove along these lines, which is essentially the only tight result related to \cref{qn:main} not treated in \cite{alonkravitz} which we know how to obtain using Fourier analysis, is the following:

\begin{proposition}\label{thm:ABB-Zmod3} In any grid $\Gamma$ of shape $(\ZZ/3\ZZ)^d$, $\mathsf{ABB}$ has concentration at most $C_1(\mathsf{ABB})3^{d-1}=2\cdot 3^{d-2}$.    
\end{proposition}

Our interest in this statement, despite the unpleasant restriction on grid shape, is motivated by the difficulty of understanding $C_d(\mathsf{ABB})$ using other methods. \cref{thm:ABB-Zmod3} is amenable to Fourier analysis because the directions in which we count copies of $w$ are \emph{every} direction in $(\ZZ/3\ZZ)^d$, so the function $c_d(w, \Gamma)$ can be expressed cleanly as a convolution involving the indicator function of copies of $\A$. The restrictive setting of \cref{thm:ABB-Zmod3} underscores the difficulty of using Fourier analysis in general cases.

We also use Fourier analysis, combined with the understanding of the one-dimensional case provided by \cref{thm:1d}, to prove \cref{thm:no-slant}. Here, we isolate a single letter and use Fourier analysis to show that, in high dimension, its frequency within a random line in which we search for copies of $w$ must have large variance. From here, results in the one-dimensional setting are enough.

It seems that Fourier analysis has severe limitations towards attacking \cref{qn:main} in more general settings. For words $w$ of length $\ell$, controlling $C_d(w)$ is reminiscent of understanding the occurrence of $\ell$-term arithmetic progressions in subsets of $\ZZ^d$ with differences restricted to $\{-1,0,1\}^d$. This is a well-studied and very difficult problem, and one for which Fourier-analytic techniques have only recently become available (in the $\ell=3$ case) \cite{BKM}. For $\ell>3$, the situation is even worse. Even without restricted differences, $\ell$-term arithmetic progressions for $\ell>3$ are not well-controlled by Fourier analysis, instead only by higher-order Fourier analysis (see \cite[Section~1.3]{Tao}, specifically Exercises~1.3.4 and 1.3.5). We would be very curious to see whether either of these techniques can be applied to studying $C_d(w)$.

\subsection{Organization}\label{subsec:organization} We organize this document based mostly on the technique used rather than the result proven. We begin in \cref{sec:1d} by solving the one-dimensional case of \cref{qn:main}, a problem for which the additive formulation is most useful. \cref{sec:reductions} focuses on combinatorial reductions, which we use to prove \cref{thm:odd-even,thm:AkBk}. In \cref{sec:local}, we develop the local technique described in \cref{subsec:intro-local} and give some relevant applications. In \cref{sec:short-words}, we combine previous techniques to prove \cref{thm:short-words}, as well as discussing words of length five in two dimensions. In \cref{sec:fourier}, we use Fourier analysis to prove \cref{thm:no-slant,thm:ABB-Zmod3}. In \cref{sec:not-stack}, we investigate some words which are not $d$-stackable, and make a connection to the $n$-queens problem, proving \cref{prop:ABell-1,prop:queen-words}. \cref{sec:conclusion} is the conclusion and includes many directions for further research.

\section{One dimension}\label{sec:1d}

The case of one dimension is the simplest possible nontrivial setting for our problem, and is the foundation for results about the growth of supremum concentration in dimension. In this section, we will determine $C_1(w)$ for every word $w$, classify all $w$-extremal grids, and prove a stability result, describing a way in which grids with $c_1(w,\Gamma)$ near $C_1(w)$ must not be ``too far" from some extremal grid. We begin with the determination of $C_1(w)$, as the notation and case breakdown will guide the remainder of the section.

\begin{proposition}\label{prop:1d} Let $w$ be any word. Let $\ell=\len(w)$, and define
\begin{align*}
c_{\mathrm{left}}&=\text{maximum length of palindromic prefix of }w;\\
c_{\mathrm{right}}&=\text{maximum length of palindromic suffix of }w;\\
c_{\mathrm{repeat}}&=\text{maximum length of substring of $w$, both a proper prefix and a proper suffix}.
\end{align*}
Then
\begin{enumerate}[label={(\alph*)}]
    \item If $w$ is not a palindrome,
    \[C_1(w)=\max\left(\frac1{\ell-c_{\mathrm{repeat}}},\frac1{\ell-\frac{c_{\mathrm{left}}+c_{\mathrm{right}}}2}\right).\]
    \item If $w$ is a palindrome, then $c_{\mathrm{left}} = c_{\mathrm{right}} = \ell$ and $C_1(w)=\frac2{\ell-c_{\mathrm{repeat}}}$.
\end{enumerate}
\end{proposition}

We also state our stability result. We delay the statement of the structural result to the end of \cref{subsec:1d-ex}, at which point we will have defined the necessary constructions.

\begin{proposition}\label{prop:1d-stable} For any word $w$, there exists a letter distribution $h_w$ satisfying the following: if $\delta\geq 0$ and $\Gamma$ is a $1$-dimensional grid satisfying $c_1(w,\Gamma)=C_1(w)(1-\delta)$, then
\[\TV{h_\Gamma-h_w}\leq\delta.\]
\end{proposition}

We state this result in terms of letter distribution for two reasons: firstly, it is a natural and easy-to-describe property of a grid, and secondly, letter distribution of near-extremal grids will be relevant in our proof of \cref{thm:no-slant}. More detailed structural information about near-extremal grids can be read off from our proofs, if desired.

\subsection{Examples and constructions}\label{subsec:1d-ex} If $w$ is a word with distinct letters, like the word $\mathsf{CAT}$ studied in \cite{alonkravitz}, the maximum concentration in one dimension is always attained by spelling the word alternatingly forward and backwards, reusing the first and the last letters as in \eqref{dim1distinct}. For many words with repeated letters, it is possible to improve on this concentration. 

For a word with a palindromic prefix or suffix of length longer than one, we may obtain a grid with higher concentration by still spelling the words forwards and backwards alternatingly, but now using the entire palindromic prefix or suffix in both words, rather than just the first and last letters. For instance, for the word $\mathsf{ELEPHANT}$, we can achieve a higher concentration in one dimension with the grid
    \begin{equation}\label{eq:elephant} \cdots\mathsf{HP\underline{ELE}PHAN\underline{T}NAHP\underline{ELE}PH}\cdots,
    \end{equation}
    where the underlined portions indicate letters that are used in two copies of the word. 
    
    For a word which have a string of letters appearing as both a proper prefix and a suffix, we can instead ``overlap'' this string of letters between copies of the word reading in the same direction. For instance, for the word $\mathsf{SALSA}$, we can produce the grid 
    \begin{equation}\label{eq:SALSA1dim}\cdots\mathsf{\underline{SA}L\underline{SA}L\underline{SA}L\underline{SA}L}\cdots
    \end{equation}
    where, as before, the underlined letters are used in more than one copy of the word, but now both copies are reading forward. This is more efficient than the optimal ``alternating forwards--backwards'' labeling, which is
    \begin{equation}\label{SALSAsuboptimal}\cdots\mathsf{\underline{A}SLA\underline{S}ALS\underline{A}SLA\underline{S}}\cdots.\end{equation}

    Note that for some words with a high degree of regularity, it is possible for letters or sequences of letters to be used in more than two words. For instance, for the word $\mathsf{ABABAB}$, in the grid given by \begin{equation}\label{ABABAB}\cdots\mathsf{ABABABABAB}\cdots\end{equation} each letter is used in six appearances of the word, three reading in either direction.

There exist modifications of the constructions we used above for the words $\mathsf{SALSA}$ and $\mathsf{ELEPHANT}$ which can be applied to any word $w$. They are described in general as follows:

\begin{construction}\label{constr:1d-rep} Write $w=vs$ where $\len(v)>0$ and $s$ is a suffix of length $c_\mathrm{repeat}$ that also is a prefix of $w$. Then, construct the grid $\Gamma_1=\Grid(v)$ with shape $\Zmod{(\ell-c_{\mathrm{repeat}})}$.
\end{construction}

\begin{construction}\label{constr:1d-pal} Write $w=p_{\mathrm{left}}u_1=u_2p_{\mathrm{right}}$, where $p_{\mathrm{left}}$ is a palindromic prefix of $w$ of length $c_{\mathrm{left}}$ and $p_{\mathrm{right}}$ is a palindromic suffix of $w$ of length $c_{\mathrm{right}}$. Then, construct the grid $\Gamma_2=\Grid(u_2u_1^{\mathrm{rev}})$ with shape $\Zmod{(2\ell - c_\mathrm{left} - c_\mathrm{right})}$.
\end{construction}

For each word $w$, one of these two constructions will attain the bound in \cref{prop:1d}. In particular, we have the following structural result, which describes all of the extremal constructions given any word $w$.

\begin{proposition}\label{prop:1d-structure} Let $w$, $\ell$, $c_{\mathrm{left}}$, $c_{\mathrm{right}}$, and $c_{\mathrm{repeat}}$ be as in \cref{prop:1d}. Additionally, suppose without loss of generality that $c_{\mathrm{left}} \le c_{\mathrm{right}}$. Recall that we call a grid $\Gamma$ $w$-\emph{extremal} if $c_1(w,\Gamma)=C_1(w)$.
\begin{enumerate}[label={(\alph*)}]
    \item If $c_{\mathrm{left}}+c_{\mathrm{right}}<2c_{\mathrm{repeat}}$, or if $w$ is a palindrome, then a grid is $w$-extremal if and only if it is equivalent to the grid given by \cref{constr:1d-rep}.

    \item If $c_{\mathrm{left}}+c_{\mathrm{right}}>2c_{\mathrm{repeat}}$, then a grid is $w$-extremal if and only if it is equivalent to the grid given by \cref{constr:1d-pal}.

    \item\label{prop:1d-structure c}If $c_{\mathrm{left}}+c_{\mathrm{right}}=2c_{\mathrm{repeat}}$, then the grids given by \cref{constr:1d-rep} and \cref{constr:1d-pal} are inequivalent and both $w$-extremal. Define $v$ as in \cref{constr:1d-rep}, and define $v'$ to be the word formed by the last $\ell-c_{\mathrm{repeat}}$ letters of $u_2u_1^{\mathrm{rev}}$, where $u_1$ and $u_2$ are as in \cref{constr:1d-pal}. Then a grid is $w$-extremal if and only if it is of the form $\Grid(v_1\cdots v_m)$ for some positive integer $m$ and some $v_1,\ldots,v_m\in\{v,v'\}$.\footnote{In this notation, the grid given by \cref{constr:1d-rep} is $\Grid(v)$, while the grid given by \cref{constr:1d-pal} is $\Grid(vv')$.} Moreover, there exist words $x$, $y$, and $z$ so that $v=xyzy^{\mathrm{rev}}$ and $v'=xyz^{\mathrm{rev}}y^{\mathrm{rev}}$.
\end{enumerate}
\end{proposition}

The last condition in (c) is an intermediate step in proving the proposition, and will be useful in proving \cref{prop:1d-stable}.

\subsection{Lower bound}\label{subsec:lower}

We begin by proving the lower bound of \cref{prop:1d}. We analyze \cref{constr:1d-rep} and \cref{constr:1d-pal} separately. Our arguments work in full generality, and their full strength is only needed for words like $\mathsf{ABABAB}$, where some of $\{c_{\mathrm{left}},c_{\mathrm{right}},c_{\mathrm{repeat}}\}$ are quite large. Thus, the arguments may be somewhat opaque. Simpler arguments would be possible if we imposed restrictions on the sizes of $c_{\mathrm{left}}, c_{\mathrm{right}}$, and $c_{\mathrm{repeat}}$. The reader may wish to keep in mind the model cases given above in \eqref{eq:elephant} and \eqref{eq:SALSA1dim}.

\begin{lemma}\label{lem:constr-1d-rep} The grid $\Gamma_1$ given in \cref{constr:1d-rep} satisfies
\[c_1(w,\Gamma_1)\geq\begin{cases}\frac1{\ell-c_{\mathrm{repeat}}}&\text{if }w\text{ is not a palindrome}\\\frac2{\ell-c_{\mathrm{repeat}}}&\text{if }w\text{ is a palindrome}.\end{cases}\]
\end{lemma}
\begin{proof} We use the notation in the statement of \cref{constr:1d-rep}. Let $n\geq 0$ be maximal so that $s$ contains $v^n$ as a prefix. Then $w$ contains both $v^{n+1}$ and $s$ as prefixes, so one is a prefix of the other. By the maximality of $n$, $v^{n+1}$ cannot be a prefix of $s$, so $s$ is a prefix of $v^{n+1}$. In particular, $w$ is contained within $v^{n+2}$, so $w$ appears in $\Gamma_1$. 

Since $\Gamma_1$ has size $\ell-c_{\mathrm{repeat}}$, this immediately implies that
\[c_1(w,\Gamma_1)\geq\frac1{\ell-c_{\mathrm{repeat}}}.\]
If $w$ is a palindrome, then $w$ appears in $\Gamma_1$ both forwards and backwards, so we have twice the above bound in this case.
\end{proof}

\begin{lemma}\label{lem:constr-1d-pal} The grid $\Gamma_2$ given in \cref{constr:1d-pal} satisfies
\[c_1(w,\Gamma_2)\geq\frac1{\ell-\frac{c_{\mathrm{left}}+c_{\mathrm{right}}}2}.\]
\end{lemma}
\begin{proof} We use the notation in the statement of \cref{constr:1d-pal}. Since $\Gamma_2$ has size $2\ell-c_{\mathrm{left}}-c_{\mathrm{right}}$, it is enough to show that $w$ appears once forwards and once backwards in $\Gamma_2$. By the symmetry of the construction --- if we feed $w^{\mathrm{rev}}$ into \cref{constr:1d-pal}, we obtain a grid of the same size filled with the letters of $u_1^{\mathrm{rev}}u_2$, which is equivalent without reflection to $\Gamma_2$ ---  it suffices to show that $w$ appears in $\Gamma_2$ forwards.

For $n\ge 0$, define words $a_n$ and $b_n$ recursively as follows: let $a_0$ and $b_0$ be the empty word, and for $n\geq 1$ let
\[
a_n = 
\begin{cases} 
      (u_2u_1^{\mathrm{rev}})^{n/2} & n \text{ even} \\
      a_{n-1}u_2 & n \text{ odd}
   \end{cases}
\quad\text{and}\quad b_n = 
\begin{cases} 
      (u_2^{\mathrm{rev}}u_1)^{n/2} & n \text{ even} \\
      u_1b_{n-1} & n \text{ odd}.
   \end{cases}
\]
Note that $a_n^{\mathrm{rev}} u_1 = b_{n+1}$ and $u_2b_n^{\mathrm{rev}} = a_{n+1}$. Additionally, $a_n$ appears in the grid $\Gamma_2$ reading left-to-right for all $n$, so it suffices to show that $w$ is a subword of $a_n$ for some $n$.

Let $n\geq 0$ be maximal so that $a_n$ is a prefix of $p_\mathrm{left}$ and $b_n$ is a suffix of $p_\mathrm{right}$. Then $p_\mathrm{left}$ has $a_n^\mathrm{rev}$ as a suffix and $p_\mathrm{right}$ has $b_n^{\mathrm{rev}}$ as a prefix. Therefore $a_{n+1}$ is a prefix of $w$ and $b_{n+1}$ is a suffix of $w$. Therefore one of $p_\mathrm{left}$ and $a_{n+1}$ is contained in the other, and one of $p_\mathrm{right}$ and $b_{n+1}$ is contained in the other. 

By the maximality of $n$, either $a_{n+1}$ contains $p_\mathrm{left}$ or $b_{n+1}$ contains $p_\mathrm{right}$. If $a_{n+1}$ contains $p_\mathrm{left}$, then $a_{n+1}^\mathrm{rev} u_1 = b_{n+2}$ has $w$ as a suffix, so $b_{n+2}^\mathrm{rev}$ has $p_\mathrm{right}$ as a prefix, so $w$ is a subword of $u_2 b_{n+2}^\mathrm{rev} = a_{n+3}$. If $b_{n+1}$ contains $p_\mathrm{right}$, then $u_2 b_{n+1}^\mathrm{rev} = a_{n+2}$ has $w$ as a prefix. Either way, $w$ appears left-to-right in $\Gamma_2$, as desired.     
\end{proof}

The lower bound of \cref{prop:1d} follows immediately from \cref{lem:constr-1d-rep,lem:constr-1d-pal}.

\subsection{Upper bound}\label{subsec:upper} We now prove the upper bound of \cref{prop:1d}. It is here that we introduce an additive framework, specialized from the general description given in \cref{subsec:intro-additive} to one dimension, which will be useful in proving \cref{prop:1d-stable,prop:1d-structure}.

For sets $A,B\subset\Zmod n$, denote by $A-B$ the difference set $\{a-b : a\in A,b\in B\}$. We begin by showing the following lemma. We will use the bound to prove the upper bound of \cref{prop:1d} and the structural properties to prove \cref{prop:1d-structure}.

\begin{lemma}\label{lem:difference-avoid-bound} Let $S,T\subset\Zmod n$ and suppose $a,b,c>0$ are integers so that the sets
\[(S-S)\cap\{1,\ldots,a-1\},\quad (T-T)\cap\{1,\ldots,a-1\},\quad (S-T)\cap\{-(c-1),\ldots,b-1\}\]
are all empty. Then
\begin{equation}\label{eq:additive-ub}
|S|+|T|\leq \max\left(\frac{n}a,\frac{2n}{b+c}\right).
\end{equation}
In addition, the bound \eqref{eq:additive-ub} enjoys the following structural properties:
\begin{enumerate}[label={(\alph*)}]
    \item If $a<\frac{b+c}2$ and equality holds in \eqref{eq:additive-ub}, then $a\mid n$, one of $S$ and $T$ is empty, and the other is a translate of $a\ZZ/n\ZZ$.

    \item If $a>\frac{b+c}2$ and equality holds in \eqref{eq:additive-ub}, then $(b+c)\mid n$ and $S=T-c$, and both $S$ and $T$ are translates of $(b+c)\ZZ/n\ZZ$.

    \item If $a = \frac{b+c}2$ and equality holds in \eqref{eq:additive-ub}, then $a \mid n$ and $S \sqcup (T+a-c)$ is a translate of $a\ZZ/n\ZZ$. 
\end{enumerate}
\end{lemma}

\begin{proof} We first verify \eqref{eq:additive-ub}. Set $T'=T-c$, so that $T'-T'$ avoids $1,\ldots,a-1$ and $S-T'$ avoids $1,\ldots,b+c-1$. Let $r,r'$ be positive integers so that $\max(r,r')\leq a$ and $r+r'\leq b+c$. We claim that the $r+r'$ sets
\[S-r,\ldots,S-1,T',\ldots,T'+(r'-1)\]
are disjoint. Indeed, the translates of $S$ are disjoint since $S-S$ avoids $1,\ldots,r-1$, the translates of $T'$ are disjoint since $T'-T'$ avoids $1,\ldots,r'-1$, and the translates of $S$ are disjoint from the translates of $T'$ since $S-T'$ avoids $1,\ldots,r+r'-1$. Since these sets are all contained in $\Zmod n$, we conclude
\[r|S|+r'|T|=r|S|+r'|T'|\leq n.\]
By swapping $r$ and $r'$, we also have $r'|S|+r|T|\leq n$, and so
\[|S|+|T|\leq \frac{2n}{r+r'}.\]
If $2a\leq b+c$, take $r=r'=a$ to get $|S|+|T|\leq \frac na$. If $2a>b+c$, take $r=\lfloor\frac{b+c}2\rfloor$ and $r'=\lceil\frac{b+c}2\rceil$ to get $|S|+|T|\leq \frac{2n}{b+c}$. This gives \eqref{eq:additive-ub}.

We now prove the structural properties.
\begin{enumerate}[label={(\roman*)}]
    \item \label{item:first} Suppose $a\leq\frac{b+c}2$. If equality holds in \eqref{eq:additive-ub}, then the sets
    \[S-a,\ldots,S-1,T',\ldots,T'+(a-1)\]
    partition $\ZZ/n\ZZ$. By translation, the sets
    \[S-(a-1),\ldots,S,T'+1,\ldots,T'+a\]
    do as well. We conclude that
    \[S\sqcup(T'+a)=T'\sqcup(S-a).\]
    Let $U=S\sqcup(T'+a)$; then $|U|=n/a$ and $U=U-a$. We conclude that $U$ is a translate of $a\ZZ/n\ZZ$. In the case where $a=\frac{b+c}2$, this is enough to show (c). If $a<\frac{b+c}2$, then $S$ and $T'+2a$ are disjoint. Both $T'+a$ and $T'+2a$ are subsets of $U\setminus S$ of size $|U\setminus S|$, so we in fact have $T'+a=T'+2a$, i.e.~$T'=T'+a$. Similarly, $S=S+a$. Since $S\sqcup(T'+a)$ consists of exactly one residue class modulo $a$, one of $S$ and $T'$ is empty in this case.

    \item Suppose $a>\frac{b+c}2$. Let $r=\lfloor\frac{b+c}2\rfloor$ and $r'=\lceil\frac{b+c}2\rceil$, so that $r<a$ and $r'\leq a$. If equality holds in \eqref{eq:additive-ub}, then we must have $r|S|+r'|T'|=r'|S|+r|T'|=n$, and the sets
    \[\{S-r,\ldots,S-1,T',\ldots,T'+(r'-1)\}\text{ and }\{S-(r-1),\ldots,S,T'+1,\ldots,T'+r\}\]
    partition $\ZZ/n\ZZ$. We conclude that $(T'+r')\sqcup S=T'\sqcup(S-r)$.

    Since $r<a$, we have that $S$ and $S-r$ are disjoint. If $r=r'$, then $r'<a$, and so $T'$ and $T'+r'$ are disjoint as well. If $r\neq r'$, then $r|S|+r'|T'|=r'|S|+r|T'|=n$ implies $|S|=|T'|$, and the fact that $T'$ and $T'+r'$ are disjoint follows from $(T'+r')\sqcup S=T'\sqcup(S-r)$. In either case, we conclude that $S=T'$ and $T'+r'=S-r$. In particular, both $S$ and $T'$ are of the same size $n/(b+c)$ and are preserved under adding $b+c$. So, both are translates of $(b+c)\ZZ/n\ZZ$, as desired. \qedhere
\end{enumerate}
\end{proof}

\begin{proof}[Proof of \cref{prop:1d}, upper bound]
Let $n$ be a positive integer, and consider any one-dimensional grid $\Gamma$ of shape $\Zmod{n}$. Define the sets
\begin{align*}
    S &= \{\text{leftmost indices of appearances of $w$ reading left-to-right}\}\\
    T &= \{\text{leftmost indices of appearances of $w$ reading right-to-left}\}.
\end{align*}
The number of appearances of $w$ in $\Gamma$ is $|S|+|T|$, and the concentration is $\frac{|S|+|T|}n$. So, we seek an upper bound on $|S|+|T|$. We record the following observations:
\begin{itemize}
    \item If $0<d<\ell$ and $d\in S-S$, then the last $\ell-d$ characters and the first $\ell-d$ characters of copies of $w$ in $\Gamma$ coincide, and so these characters form a proper prefix and a proper suffix of $w$. So, $\ell-d\leq c_{\mathrm{repeat}}$. This implies that
    \begin{equation}\label{eq:S-S}
        (S-S)\cap\big\{1,\ldots,\ell-c_{\text{repeat}}-1\big\}=\emptyset.
    \end{equation}

    \item By applying the argument above to $T$, we get that
    \begin{equation}\label{eq:T-T}
        (T-T)\cap\big\{1,\ldots,\ell-c_{\text{repeat}}-1\big\}=\emptyset.
    \end{equation}

    \item If $0\leq d<\ell$ and $d\in S-T$, then the first $\ell-d$ characters of a left-to-right copy of $w$ in $\Gamma$ and the last $\ell-d$ characters of a right-to-left copy of $w$ in $\Gamma$ coincide. This implies that the first $\ell-d$ characters of $w$ form a palindrome. So, $\ell-d\leq c_{\text{left}}$ (unless $w$ itself is a palindrome, in which case this inequality holds whenever $d\neq 0$). We conclude that
    \begin{equation}\label{eq:S-T}
        (S-T)\cap\big\{1,\ldots,\ell-c_{\text{left}}-1\big\}=\emptyset.
    \end{equation}

    \item By applying the argument above to $T-S$, we get that
    \begin{equation}\label{eq:T-S}
        (T-S)\cap\big\{1,\ldots,\ell-c_{\text{right}}-1\big\}=\emptyset.
    \end{equation}
\end{itemize}

If $w$ is not a palindrome, then $0\not\in S-T$. This, combined with \eqref{eq:S-T} and \eqref{eq:T-S}, give that
\[(S-T)\cap\big\{-(\ell-c_{\mathrm{right}}-1),\ldots,\ell-c_{\mathrm{left}}-1\big\}=\emptyset.\]
Applying \cref{lem:difference-avoid-bound}, using the above with \eqref{eq:S-S} and \eqref{eq:T-T}, directly gives the result. If $w$ is a palindrome, we have $S=T$ and $c_{\mathrm{left}}=c_{\mathrm{right}}$. We may apply \cref{lem:difference-avoid-bound} to the sets $S$ and $\emptyset$ with $a=b=c=c_{\mathrm{left}}$ to get that $|S|\leq \frac{n}{c_{\mathrm{left}}}$, which is enough to give the upper bound in this case.
\end{proof}

By combining this with the lower bound, we get that, for every word $w$, at least one of \cref{constr:1d-rep} and \cref{constr:1d-pal} produces a grid that achieves the maximum possible concentration of $w$ in one dimension. In particular, as promised in the discussion around \cref{qn:extremal-exist}, we have the following corollary:

\begin{corollary}\label{cor:1d-extremal-exist}
For every word $w$, there exists a grid $\Gamma$ such that $c_1(w, \Gamma) = C_1(w).$
\end{corollary}

\subsection{Characterization of extremal grids}\label{subsec:structure}

In this section, we prove \cref{prop:1d-structure}, characterizing all grids $\Gamma$ for which $c_1(w,\Gamma)=C_1(w)$. Much of the proof will mirror the proof of the upper bound in \cref{prop:1d}. In case (c) below, the argument is quite delicate, and the reader may wish to keep in mind an example, such as $\mathsf{CROC}$.

\begin{proof}[Proof of \cref{prop:1d-structure}] Let $\Gamma$ be a grid of shape $\ZZ/n\ZZ$ satisfying $c_1(w,\Gamma)=C_1(w)$. As before, define
\begin{align*}
    S &= \{\text{leftmost indices of appearances of $w$ reading left-to-right}\}\\
    T &= \{\text{leftmost indices of appearances of $w$ reading right-to-left}\},
\end{align*}
so that $nC_1(w)=|S|+|T|$, and 
\begin{align*}
(S-S)\cap\big\{1,\ldots,\ell-c_{\text{repeat}}-1\big\}&=\emptyset;\\
(T-T)\cap\big\{1,\ldots,\ell-c_{\text{repeat}}-1\big\}&=\emptyset;\\
(S-T)\cap\big\{-(\ell-c_{\mathrm{right}}-1),\ldots,\ell-c_{\mathrm{left}}-1\big\}&=\emptyset\ (\text{if $w$ is not a palindrome}).
\end{align*}
We first treat the case where $w$ is a palindrome. In this case, $S=T$ and $|S|=\frac n{\ell-c_{\mathrm{repeat}}}$. Applying \cref{lem:difference-avoid-bound}(a) to the sets $S$ and $\emptyset$ with $a=\ell-c_{\mathrm{repeat}}$ and $b=c=\ell$ gives that $S$ is a translate of $a\ZZ/n\ZZ$. For each $s\in S$, the $a$ characters in $\Gamma$ at positions $\{s,s+1,\ldots,s+a-1\}$ are the first $a$ characters of $w$ in order. Letting $v$ be the word formed by these $a$ characters, we conclude from the fact that $S$ is a translate of $a\ZZ/n\ZZ$ that $\Gamma\cong\Grid(v^m)$ for some $m$. So, $\Gamma$ is equivalent to the grid given by \cref{constr:1d-rep}, as desired.

We now assume $w$ is not a palindrome. Let $a=\ell-c_{\mathrm{repeat}}$, $b=\ell-c_{\mathrm{left}}$, and $c=\ell-c_{\mathrm{right}}$, so that we may apply \cref{lem:difference-avoid-bound} to the sets $S$ and $T$ with $a$, $b$, and $c$.

\begin{enumerate}[label={(\alph*)}]
    \item If $c_{\mathrm{left}}+c_{\mathrm{right}}<2c_{\mathrm{repeat}}$, then $a<\frac{b+c}2$, and \cref{lem:difference-avoid-bound}(a) gives that one of $\{S,T\}$ is empty and the other is a translate of $a\ZZ/n\ZZ$. We may assume without loss of generality that $T$ is empty; otherwise, replace $\Gamma$ with the grid obtained by reversing the letters in $\Gamma$. Now, as in the palindromic case, $\Gamma$ is formed by repeating the first $a$ letters of $w$, i.e.~$\Gamma\cong\Grid(v^m)$ for some $m$, where $v$ is as in \cref{constr:1d-rep}. So, $\Gamma$ is equivalent to the grid given by \cref{constr:1d-rep}.

    \item If $c_{\mathrm{left}}+c_{\mathrm{right}}>2c_{\mathrm{repeat}}$, then $a>\frac{b+c}2$, and \cref{lem:difference-avoid-bound}(b) gives that $S=T-c$ and both $S$ and $T$ are translates of $(b+c)\ZZ/n\ZZ$. For each $s\in S$, the $c$ characters in $\Gamma$ at positions $\{s,s+1,\ldots,s+c-1\}$ are the first $c$ characters of $w$ in order; since $s+c\in T$, the next $b$ characters are the last $b$ characters of $w$ reversed. In the notation of \cref{constr:1d-pal}, these strings of characters of $w$ are $u_2$ and $u_1^{\mathrm{rev}}$, respectively, and so $\Gamma=\Grid((u_2u_1^{\mathrm{rev}})^m)$ for some positive integer $m$. Thus $\Gamma$ is equivalent to the grid given by \cref{constr:1d-pal}.

    \item If $c_{\mathrm{left}}+c_{\mathrm{right}}=2c_{\mathrm{repeat}}$, then $a=\frac{b+c}2$. We may assume without loss of generality that $c_{\mathrm{left}}\leq c_{\mathrm{repeat}}\leq c_{\mathrm{right}}$, so that $c\leq a\leq b$.     The two grids given by \cref{constr:1d-rep} and \cref{constr:1d-pal} are both extremal by \cref{lem:constr-1d-rep} and \cref{lem:constr-1d-pal}, and they are inequivalent since \cref{constr:1d-rep} contains copies of $w$ in only one direction while \cref{constr:1d-pal} contains them in both directions. Let $v$, $u_1$, and $u_2$ be as in those constructions. 
    
    By \cref{lem:constr-1d-rep}, one of $\{w,vv\}$ is a prefix of the other, and by \cref{lem:constr-1d-pal}, the same holds of $\{w,u_2u_1^{\mathrm{rev}}\}$. If $\ell\geq 2c_{\mathrm{repeat}}$, then $w$ is longer than $u_2u_1^{\mathrm{rev}}$ and $vv$, so $u_2u_1^{\mathrm{rev}}$ and $vv$ are both prefixes of $w$ of the same length. This implies that \cref{constr:1d-rep} and \cref{constr:1d-pal} are equivalent, which cannot be. Therefore, $2c_{\mathrm{repeat}}<\ell$.

    Let $p_{\mathrm{left}}$, $p_{\mathrm{right}}$ and $s$ be as in \cref{constr:1d-rep} and \cref{constr:1d-pal}. Since $c_{\mathrm{left}}\leq c_{\mathrm{repeat}}\leq c_{\mathrm{right}}$, $p_{\mathrm{left}}$ is a prefix of $s$ and $s$ is a suffix of $p_{\mathrm{right}}$. Let $x=p_{\mathrm{left}}$ (a palindrome), and define $y$ so that $s=xy$. We have $\len(p_{\mathrm{right}})=\len(s)+\len(y)$, so $p_{\mathrm{right}}=y^{\mathrm{rev}}xy$. Now, $w$ has $s=xy$ as a prefix and $p_{\mathrm{right}}=y^{\mathrm{rev}}xy$ as a suffix. If $\len(s)+\len(p_{\mathrm{right}})>\ell$, then the substring $y$ from the prefix and the substring $y^{\mathrm{rev}}$ from the suffix overlap, and the first $\ell-\len(y)$ letters of $w$ form a palindrome (since $x$ is a palindrome). This contradicts the definition of $c_{\mathrm{left}}$. Therefore the prefix $s$ and the suffix $p_{\mathrm{right}}$ do not overlap. Therefore there exists a string $z$ for which $w=xyzy^{\mathrm{rev}}xy$. We have
    \[v=xyzy^{\mathrm{rev}},\qquad u_2=xyz,\qquad u_1=yzy^{\mathrm{rev}}xy,\qquad v'=xyz^{\mathrm{rev}}y^{\mathrm{rev}}.\]
    We now have enough structural understanding of these words to conclude the proof using arguments similar to (a) and (b). By \cref{lem:difference-avoid-bound}(c), $S\sqcup(T+a-c)$ is a translate of $a\ZZ/n\ZZ$; translate the grid so $S\sqcup(T+a-c)=a\ZZ/n\ZZ$. For each $s\in S$, the $a$ characters in $\Gamma$ at positions $\{s,s+1,\ldots,s+a-1\}$ are the first $a$ characters of $w$, i.e.~$v$. For each $u\in T+a-c=T+b-a$, the next $a$ characters are the $a$ characters at indices $\{b-a+1,\ldots,b\}$ in $w$ when read in reverse. These are the last $a$ characters of $u_1^{\mathrm{rev}}$, i.e.~$v'$. Therefore, each block of length $a$ in $\Gamma$ starting at a multiple of $a$ is either $v$ or $v'$, and $\Gamma=\Grid(v_1\cdots v_m)$ for some $v_1,\ldots,v_m\in\{v,v'\}$. Moreover, any grid of this form contains one forwards copy of $w$ beginning at the location of each copy of $v$, since both $v$ and $v'$ begin with $xy$ (note that $w=vxy$). Such a grid also contains one backwards copy of $w$ containing each copy of $v'$ as a subword, since both $v'$ and $v$ begin with $x$ and end with $y^{\mathrm{rev}}$ (note that $w^{\mathrm{rev}}=y^{\mathrm{rev}}v'x$). So, $\Grid(v_1\cdots v_m)$ has size $ma$ and contains at least $m$ copies of $w$, meaning that it is a $w$-extremal grid. This concludes the proof. \qedhere
\end{enumerate}
\end{proof}

\begin{remark}
The complexity of the proof of Proposition \ref{prop:1d-structure c} reflects a complexity in the structure of the extremal grids for such words. For instance, the number of inequivalent extremal grids of size $3n$ the word $\mathsf{CROC}$ admits is exponential in $n$, as we may freely choose the relative ordering of each consecutive $\mathsf{R}$ and $\mathsf{O}$. To illustrate, the grid $\Grid(\mathsf{CROCROCOR)}$ is $\mathsf{CROC}$-extremal and not equivalent to the grid given either by \cref{constr:1d-rep} or \cref{constr:1d-pal}. 
\end{remark}

\subsection{Letter distribution in near-extremal grids: a stability result}\label{subsec:stability} 
We are now ready to prove \cref{prop:1d-stable}. We begin by proving the following weaker statement about letter distributions of extremal grids (as opposed to near-extremal grids).

\begin{lemma}\label{lem:same-dist} For any word $w$, there exists a fixed letter distribution $h_w$ such that, for any one-dimensional $w$-extremal grid $\Gamma$, we have $h_\Gamma=h_w$.
\end{lemma}
\begin{proof} This is essentially a corollary of \cref{prop:1d-structure}. Note that equivalent grids have the same letter distribution. If $c_{\mathrm{left}}+c_{\mathrm{right}}\neq 2c_{\mathrm{repeat}}$, or if $w$ is a palindrome, then all grids $\Gamma$ with $c_1(w,\Gamma)=C_1(w)$ are equivalent by parts (a) and (b) of the proposition, and so $h_w$ can be chosen to be the letter distribution of any extremal grid. If $c_{\mathrm{left}}+c_{\mathrm{right}}= 2c_{\mathrm{repeat}}<\ell$, then define $v$ and $v'$ as in case (c) of the proposition. Since $v=xyzy^{\mathrm{rev}}$ and $v'=xyz^{\mathrm{rev}}y^{\mathrm{rev}}$ for strings $x$, $y$, and $z$, both $v$ and $v'$ have the same letter distribution. Setting $h_w$ to be this distribution, \cref{prop:1d-structure}(c) gives the result.
\end{proof} 

\begin{proof}[Proof of \cref{prop:1d-stable}] The proof is conceptually similar to the proof of the upper bound of \cref{prop:1d}, and uses some of the same definitions and notation. Say $\Gamma$ has shape $\ZZ/n\ZZ$. Let $\ell=\len(w)$ and $a=\ell-c_{\mathrm{repeat}}$ and $b=\ell-c_{\mathrm{left}}$ and $c=\ell-c_{\mathrm{right}}$. Write $w=w_1w_2\cdots w_\ell$, where $w_1,\ldots,w_\ell$ are letters. Define the sets
\begin{align*}
    S &= \{\text{leftmost indices of appearances of $w$ reading left-to-right}\}\\
    T &= \{\text{leftmost indices of appearances of $w$ reading right-to-left}\}.
\end{align*}
First, we dispatch with the case where $w$ is a palindrome. In this case, $S=T$, $C_1(w)=2/a$, $|S|=n(1-\delta)/a$, and the sets
\[S,S+1,\ldots,S+(a-1)\]
are disjoint. Let $V$ be the union of these sets, so that $|V|=a|S|=n(1-\delta)$. The letters in $\Gamma$ at positions in $V$ (the \emph{letters in} $V$, for brevity) comprise $|S|$ copies of the first $a$ letters of $w$. These $a$ letters have distribution $h_w$. Now, let the distribution of the remaining letters be $h'$. Then $h_\Gamma=(1-\delta)h_w+\delta h'$, so
\[\TV{h_\Gamma-h_w}=\TV{\delta h'-\delta h_w}=\delta\TV{h'-h_w}\leq\delta,\]
as desired.

We may now assume $w$ is not a palindrome. As in the proof of \cref{lem:difference-avoid-bound}, whenever $\max(r,r')\leq a$ and $r+r'\leq b+c$, the $r+r'$ sets
\begin{equation}\label{eq:disjoint}
S-r,\ldots,S-1,T-c,\ldots,T+(r'-c-1)
\end{equation}
are disjoint. We now split into cases.
\begin{enumerate}[label={(\alph*)}]
    \item If $2a\leq b+c$, we may without loss of generality assume $c\geq a$. Here, we have $C_1(w)=1/a$, so
    \[|S|+|T|=nc_1(w,\Gamma)=\frac{n(1-\delta)}a.\]
    Select $r=r'=a$ in \eqref{eq:disjoint}, and add $a$ to each set to get that
    \[S,\ldots,S+(a-1),T+(a-c),\ldots,T+(2a-c-1)\]
    are disjoint. Let $V$ be the union of these sets. The letters in $V$ comprise $|S|$ copies of the first $a$ letters of $w$ and $|T|$ copies of the $a$ letters in $w$ beginning at $w_{\ell-(2a-c)+1}$. (These $a$ letters are all in $w$, since $\ell-(2a-c)\geq \ell-b>0$.) These collections of $a$ letters are both strings of $a$ consecutive letters in $w$. So, by \cref{lem:constr-1d-rep}, they define equivalent grids, and so are cyclic shifts of one another. By \cref{lem:constr-1d-rep} and \cref{lem:same-dist}, each of these strings thus has letter distribution $h_w$. So, the letters in $V$, a subset of $\ZZ/n\ZZ$ with size $|V|=a(|S|+|T|)=(1-\delta)n$, have the distribution $h_w$. We finish as in the palindromic case.

    \item If $2a>b+c$, we can without loss of generality assume $b\geq c$. Here, we have $C_1(w)=2/(b+c)$, so
    \[|S|+|T|=nc_1(w,\Gamma)=\frac{2n(1-\delta)}{b+c}.\]
    Let $r=\lfloor\frac{b+c}2\rfloor$ and $r'=\lceil\frac{b+c}2\rceil$, so that $c\leq r$ and $r'\leq a$. Plugging $(r,r')$ in to \eqref{eq:disjoint} and adding $r$ to each set gives that
    \[S,\ldots,S+(r-1),T+(r-c),\ldots,T+(b-1)\]
    are disjoint. Let $V_1$ be the union of these sets. In the notation of \cref{constr:1d-pal}, the letters in $V_1$ comprise $|S|$ copies of the first $r$ letters of $w$ and $|T|$ copies of the last $r'$ letters of $u_1^{\mathrm{rev}}$, since $u_1^{\mathrm{rev}}$ is formed from the first $b$ letters of $w^{\mathrm{rev}}$. In addition, plugging $(r',r)$ to \eqref{eq:disjoint} and adding $\ell-r+c$ to each set gives that
    \[S+(\ell-b),\ldots,S+(\ell-b+r'-1),T+(\ell-r),\ldots,T+(\ell-1)\]
    are disjoint. Let $V_2$ be the union of these sets; the letters in $V_2$ comprise $|S|$ copies of the first $r'$ letters of $u_1$, i.e.~the last $r'$ letters of $u_1^{\mathrm{rev}}$, and $|T|$ copies of the last $r$ letters of $w^{\mathrm{rev}}$, i.e.~the first $r$ letters of $w$. Note that the first $r$ letters of $w$ and the last $r'$ letters of $u_1^{\mathrm{rev}}$ jointly comprise all of the $r+r'=b+c$ letters in $u_2u_1^{\mathrm{rev}}$.
    
    Let $M$ be the union of the multiset of letters in $V_1$ and the letters in $V_2$. The multiset $M$ consists of $|S|+|T|$ copies of the first $r$ letters of $w$ and $|S|+|T|$ copies of the last $r'$ letters of $u_1^{\mathrm{rev}}$, and thus consists of $|S|+|T|$ copies of the letters in $u_2u_1^{\mathrm{rev}}$. Therefore, the distribution of the letters in $M$ is $h_w$. Moreover, $|M|=(b+c)(|S|+|T|)=2n(1-\delta)$, and each letter in $\Gamma$ corresponds to two letters in $M$. Let the distribution of the remaining letters in $\Gamma\sqcup\Gamma$ be $h'$, so that $h_\Gamma=(1-\delta)h_w+\delta h'$. Then
    \[\TV{h_\Gamma-h_w}=\delta\TV{h'-h_w}\leq\delta.\qedhere\]
\end{enumerate}
\end{proof}

\section{Combinatorial reductions}\label{sec:reductions}

In this section, we describe the combinatorial procedures we use to reduce the question of determining $C_d(w)$ to the question of determining $C_d(w')$ for various pairs $(w,w')$, as described in \cref{subsec:intro-reductions}. We begin with what we call \emph{projection reductions}, our most powerful and general procedure.

\subsection{Projection reductions}\label{subsec:proj-reductions}
Our projection reductions arise from extending the arguments in \cite{alonkravitz}. Before discussing our more general result, we give an example (modified from \cite{alonkravitz}) which illustrates the premise.

\begin{example}\label{ex:cat}
Suppose we know the word $\mathsf{AB}$ is $d$-stackable, and we want to prove the word $\mathsf{CAT}$ is $d$-stackable. For any grid, if we consider ``projecting" the grid by mapping both $\C$ and $\mathsf{T}$ to $\B$ while fixing $\A$, the number of appearances of $\mathsf{CAT}$ in the original grid is at most one-half the number of appearances of $\mathsf{AB}$ in the projected grid, because each appearance of $\mathsf{CAT}$ corresponds to an appearance of each of $\mathsf{AC}$ and $\mathsf{AT}$, both of which project to $\mathsf{AB}$. Moreover, projecting the grid $\Theta = \Grid(\mathsf{CATA}) \times (\Zmod{1})^{d-1}$ in this way gives us a $\mathsf{AB}$-extremal grid, since $\mathsf{AB}$ is known to be $d$-stackable. The ratio of appearances of $\mathsf{CAT}$ in $\Theta$ to appearances of $\mathsf{AB}$ in the projected grid is exactly $1/2$, and so $\Theta$ must be $\mathsf{CAT}$-extremal. Thus $\mathsf{CAT}$ is $d$-stackable. 
\end{example}

The general proposition underlying projection reductions is the following.

\begin{proposition}\label{prop:proj-reductions} Let $w$ be a word with letters from the alphabet $\mathsf\Sigma$. Suppose that there exists a word $w'$ with letters from the alphabet $\mathsf\Sigma'$, a map $\pi\colon\mathsf\Sigma\to\mathsf\Sigma'$ and a one-dimensional grid $\Gamma_0$ with letters in $\mathsf\Sigma$ such that
\begin{enumerate}[label={(\alph*)}]
    \item $\Gamma_0$ is $w$-extremal,
    \item $\pi(\Gamma_0)$ is $w'$-extremal,\footnote{By $\pi(\Gamma_0)$ we mean the grid given by the map $\pi\circ\Gamma_0$. That is, if some letter $A$ occurs at position $p$ in $\Gamma_0$, the letter $\pi(A)$ occurs at position $p$ in $\pi(\Gamma_0)$.} and
    \item for every one-dimensional grid $\Gamma$ with letters from $\mathsf\Sigma$,
    \[\frac{\ct(w,\Gamma)}{\ct(w',\pi(\Gamma))}\leq\frac{\ct(w,\Gamma_0)}{\ct(w',\pi(\Gamma_0))}.\]
\end{enumerate}
Then, for any positive integer $d$,
\[\frac{C_d(w)}{C_1(w)}\leq\frac{C_d(w')}{C_1(w')}.\]
\end{proposition}
\begin{proof} Let $\Theta$ be any $d$-dimensional grid with letters from $\mathsf\Sigma$; we must upper-bound $c_d(w,\Theta)$. Let $\mathcal{S}$ be the set of search lines of $\Theta$. Recall
\[\ct(w,\Theta)=\sum_{\Gamma\in\mathcal S}\ct(w,\Gamma).\]
Write $r=\ct(w,\Gamma_0)/\ct(w',\pi(\Gamma_0))$. Using condition (c), we compute
\begin{align*}
c_d(w,\Theta)
&=\frac1{|\Theta|}\ct(w,\Theta)\\
&=\frac1{|\Theta|}\sum_{\Gamma\in\mathcal S}\ct(w,\Gamma)\\
&\leq\frac1{|\Theta|}\sum_{\Gamma\in\mathcal S}r\ct(w',\pi(\Gamma))\\
&=r\frac{\ct(w',\pi(\Theta))}{|\Theta|}\\
&=r\cdot c_d(w',\pi(\Theta))\leq r\cdot C_d(w').
\end{align*}
By (a) and (b), $r=\frac{C_1(w)}{C_1(w')}$. Rearranging finishes the proof.
\end{proof}

\cref{prop:proj-reductions} provides a way to upper-bound $C_d(w)$ given $C_d(w')$. A corresponding lower bound can often be found by reversing the procedure: finding some $d$-dimensional grid $\Theta'$, extremal for $w'$, which arises as $\pi(\Theta)$ for some grid $\Theta$ satisfying
\[c_d(w,\Theta)=\frac{C_1(w)C_d(w')}{C_1(w')}.\]
If such a pair $(\Theta,\Theta')$ exists, then equality follows in the conclusion of \cref{prop:proj-reductions}, and this is enough to determine $C_d(w)$. For example, in \cref{ex:cat}, equality is demonstrated by taking $\Theta'=\Grid(\mathsf{BABA})\times(\ZZ/1\ZZ)^{d-1}$ and $\Theta=\Grid(\mathsf{CATA})\times(\ZZ/1\ZZ)^{d-1}$.

\subsection{Projection to \texorpdfstring{$\mathsf{ABB}$}{ABB}}\label{subsec:proj-ABB} Our first application of \cref{prop:proj-reductions} will be towards proving \cref{thm:short-words}. We will reduce four of the isomorphism classes of four-letter words to the problem of determining $C_2(\mathsf{ABB})$ (which we solve by other means in \cref{sec:local}). This is encapsulated in the following lemma.

\begin{lemma}\label{lem:ABB-reduction} Let $d$ be a positive integer for which $\mathsf{ABB}$ is $d$-stackable. Then the words $\mathsf{ABCA}$, $\mathsf{ABBC}$, $\mathsf{ABBA}$, and $\mathsf{BABB}$ are $d$-stackable.    
\end{lemma}

In proving this lemma, we will apply \cref{prop:proj-reductions}. The full strength of the proposition is not especially necessary, as it will turn out in these cases that the inequality in condition (c) is just as easy to verify for $d$-dimensional grids $\Gamma$ as for $1$-dimensional grids. However, we phrase our arguments in this way to highlight the versatility of our projection machinery.

\begin{proof} For each of our four four-letter words $w$, it suffices to show that 
\[\frac{C_d(w)}{C_1(w)}\leq \frac{C_d(\mathsf{ABB})}{C_1(\mathsf{ABB})};\]
if $\mathsf{ABB}$ is $d$-stackable, then the fraction on the right equals $3^{d-1}$, and the result follows from the lower bound of \cref{prop:simple-d-bound}. As a result, it suffices to verify the preconditions of \cref{prop:proj-reductions}.
\begin{enumerate} 
    \item Consider $w=\mathsf{ABCA}$ and $w'=\mathsf{ABB}$, and let $\pi\colon\{\A,\B,\C\}\to\{\A,\B\}$ denote the projection fixing $\A$ and $\B$ and sending $\C$ to $\B$. Let $\Gamma_0$ be the one-dimensional grid of size $3$ given by $\mathsf{ABC}$. We compute using \cref{prop:1d} that
    \[c_1(w,\Gamma_0)=\frac13=C_1(w);\qquad c_1(w',\pi(\Gamma_0))=\frac23=C_1(w').\]
    So, conditions (a) and (b) of \cref{prop:proj-reductions} are satisfied. Moreover, in any grid $\Gamma$,
    \[2\ct(w,\Gamma)\leq\ct(\mathsf{ABC},\Gamma)+\ct(\mathsf{ACB},\Gamma)\leq \ct(w',\pi(\Gamma));\]
    the first inequality is because each appearance of $w$ in $\Gamma$ can be associated to an appearance of $\mathsf{ABC}$ (appearing in the same direction, starting at the starting cell of the appearance of $w$) and a appearance of $\mathsf{ACB}$ (appearing in the opposite direction, starting at the ending cell of the appearance of $w$). This implies that
    \[\frac{\ct(w,\Gamma)}{\ct(w',\pi(\Gamma))}\leq \frac12=\frac{\ct(w,\Gamma_0)}{\ct(w',\pi(\Gamma_0))}.\]
    We have thus verified all of the preconditions of \cref{prop:proj-reductions}.

    \item Consider $w=\mathsf{ABBC}$ and $w'=\mathsf{ABB}$, and let $\pi\colon\{\A,\B,\C\}\to\{\A,\B\}$ denote the projection fixing $\A$ and $\B$ and sending $\C$ to $\A$. Let $\Gamma_0$ be the one-dimensional grid of size $6$ given by $\mathsf{ABBCBB}$. As in (1), we may compute that
    \[c_1(w,\Gamma_0)=\frac26=C_1(w);\qquad c_1(w',\pi(\Gamma_0))=\frac46=C_1(w').\]
    Moreover, with the same reasoning as in (1), in any grid $\Gamma$, we have
    \[2\ct(w,\Gamma)\leq\ct(\mathsf{ABB},\Gamma)+\ct(\mathsf{CBB},\Gamma)=\ct(w',\pi(\Gamma)).\]
    This gives condition (c) of \cref{prop:proj-reductions}.

    \item Consider $w=\mathsf{ABBA}$ and $w'=\mathsf{ABB}$, and let $\pi\colon\{\A,\B\}\to\{\A,\B\}$ be the identity. Let $\Gamma_0$ be the one-dimensional grid of size $3$ given by $\mathsf{ABB}$. We may compute that
    \[c_1(w,\Gamma_0)=\frac23=C_1(w);\qquad c_1(w',\Gamma_0)=\frac23=C_1(w').\]
    In any grid $\Gamma$, we have $\ct(w,\Gamma)\leq\ct(w',\Gamma)$, by associating to each appearance of $\mathsf{ABBA}$ the appearance of $\mathsf{ABB}$ in the same direction with the same starting point. This gives condition (c).

    \item For $w=\mathsf{BABB}$, the proof is the same as in (3) with the same choices of $\Gamma_0$ and $\pi$: the inequality $\ct(w,\Gamma)\leq\ct(w',\Gamma)$ follows from associating to an appearance of $\mathsf{BABB}$ the appearance of $\mathsf{ABB}$ in the same direction beginning at the second letter. \qedhere
\end{enumerate}
\end{proof}

\subsection{\texorpdfstring{$\mathsf{ELEPHANT}$}{ELEPHANT}-type words: projection to \texorpdfstring{$\mathsf{AB}$}{AB}}\label{subsec:proj-AB}
We now return to the question of Alon and Kravitz of words like $\mathsf{ELEPHANT}$, which have no letters that appear at both even and odd indices but no other restrictions. We shall prove \cref{thm:odd-even}, that all such words are $d$-stackable for all $d$.

Before we prove \cref{thm:odd-even}, we first introduce some notation. Fix some word $w$ with each letter appearing only at even indices or only at odd indices. Let $\mathcal{O}$ be the set of letters of $w$ appearing only at odd indices of the word, and $\mathcal{E}$ be the set of letters that appear only at even indices. Let $\mathsf\Sigma=\mathcal O\cup\mathcal E$ be the alphabet from which the letters in $w$ are chosen. Say a grid $\Gamma$ is \emph{parity-respecting} if all of the two-letter words that appear in $\Gamma$ are of the form $OE$ or $EO$ for some $O\in \mathcal{O}$ and some $E\in \mathcal{E}$. (Note that only one-dimensional grids may be parity-respecting.) Let $\pi$ denote the projection map $\mathsf\Sigma\to\{\mathsf O,\mathsf E\}$ mapping every element of $\mathcal O$ (resp.\ $\mathcal E$) to $\mathsf O$ (resp.\ $\mathsf E$). The map $\pi$ induces a map from grids with letters in $\mathsf\Sigma$ to grids with letters in $\{\mathsf O,\mathsf E\}$ given by replacing a letter $x$ with $\pi(x)$; we also denote this map by $\pi$.

The main technical portion of the proof is encapsulated in the following lemma.

\begin{lemma}\label{lem:OE-ratio} Let $w$ be as in the hypothesis of \cref{thm:odd-even}. There exists a one-dimensional grid $\Gamma_0$ such that
\begin{enumerate}[label={(\alph*)}]
    \item $\Gamma_0$ is $w$-extremal,
    \item $\Gamma_0$ is parity-respecting, and
    \item for any one-dimensional grid $\Gamma$, we have
    \begin{equation}
    \label{eq:1d-gamma0-bound}\frac{\ct(w,\Gamma)}{\ct(\mathsf{OE},\pi(\Gamma))}\leq\frac{\ct(w,\Gamma_0)}{\ct(\mathsf{OE},\pi(\Gamma_0))}.
    \end{equation}
\end{enumerate}
\end{lemma}

We first prove \cref{thm:odd-even} assuming \cref{lem:OE-ratio}.

\begin{proof}[Proof of \cref{thm:odd-even} assuming \cref{lem:OE-ratio}]
Let $\Gamma_0$ be as in the conclusion of \cref{lem:OE-ratio}. Since $\Gamma_0$ is parity-respecting, $\pi(\Gamma_0)$ is $\mathsf{OE}$-extremal. Letting $w'=\mathsf{OE}$, conditions (a) and (c) of \cref{prop:proj-reductions} follow exactly from conclusions (a) and (c) of \cref{lem:OE-ratio}. Therefore, for any dimension $d$, we have \[\frac{C_d(w)}{C_1(w)} \le \frac{C_d(\mathsf{OE})}{C_1(\mathsf{OE})}.\]
Since $\mathsf{OE}$ is $d$-stackable by \cref{thm:AK-restated}, we conclude that $C_d(w) \le 3^{d-1} C_1(w)$ and $w$ is also $d$-stackable.
\end{proof}

What remains is to prove \cref{lem:OE-ratio}.

\begin{proof}[Proof of \cref{lem:OE-ratio}] We let $\Gamma_0$ be one of the grids constructed in \cref{constr:1d-rep} and \cref{constr:1d-pal}; specifically, we choose \cref{constr:1d-rep} if $\frac{c_{\mathrm{left}}+c_{\mathrm{right}}}2<c_{\mathrm{repeat}}$ and \cref{constr:1d-pal} otherwise. It is clear from this choice, using \cref{prop:1d}, that (a) is satisfied. 

We next show that (b) is satisfied. We split into cases based on from which construction $\Gamma_0$ originated:
\begin{itemize}
    \item If $\Gamma_0$ is given by \cref{constr:1d-rep}, we must have $c_{\mathrm{repeat}}>0$. Recall that $w=vs$, where $s$ is both a prefix and suffix of $w$ of length $c_{\mathrm{repeat}}$, and that $\Gamma_0$ is given by the letters in $v$. If $\len(v)$ is odd, then the first letter of $s$ occurs at both an even and odd index in $w$. The grid $\Gamma_0$ must have even size, since otherwise the first letter of $w$ would be an element of $\mathcal{O} \cap \mathcal{E}$. Therefore, $\len(v)$ is even. Since every letter of $v$ appears either only at even indices or only at odd indices, the letters of this grid alternate between $\mathcal{O}$ and $\mathcal{E}$ and so $\Gamma_0$ is parity-respecting.

    \item If $\Gamma_0$ is given by \cref{constr:1d-pal}, recall that $w=p_{\mathrm{left}}u_1=u_2p_{\mathrm{right}}$, where $p_{\mathrm{left}}$ and $p_{\mathrm{right}}$ are palindromes, and $\Gamma_0$ is given by the sequence of letters $u_2u_1^{\mathrm{rev}}$. We first observe that the palindromes $p_\mathrm{left}$ and $p_\mathrm{right}$ must both have odd length, as otherwise they would contain letters at both even and odd indices. (They are nonempty since a single letter is a palindrome.)
    
    Now, if $\len(w)$ is even, both $u_1$ and $u_2$ are of odd length. The subword $u_2$ then begins and ends at an odd index in $w$, and $u_1$ begins and ends at an even index. Therefore $\pi(u_2u_1^{\mathrm{rev}})=\mathsf{OE}^k$ for $k=\frac{\len(u_1)+\len(u_2)}2$, and so $\Gamma_0$ is parity-respecting. If, on the other hand, $\len(w)$ is odd, both $u_1$ and $u_2$ are of even length. The subword $u_2$ of $w$ begin at an odd index and ends at an even index, while the subword $u_1$ begins at an even index and ends at an odd index. The conclusion is the same as before.
\end{itemize}

We now show (c). Let $\Gamma$ be any one-dimensional grid. If $\Gamma$ is parity-respecting, then $\ct(\mathsf{OE}, \pi(\Gamma)) = |\Gamma|$, and so
\[\frac{\ct(w,\Gamma)}{\ct(\mathsf{OE},\pi(\Gamma))}=\frac{\ct(w,\Gamma)}{|\Gamma|}=c_1(w,\Gamma)\leq c_1(w,\Gamma_0)=\frac{\ct(w,\Gamma_0)}{\ct(\mathsf{OE},\pi(\Gamma_0))}.\]
It thus suffices to show that, for any one-dimensional grid $\Gamma$, there exists a parity-respecting grid $\Gamma'$ satisfying
\[\frac{\ct(w,\Gamma)}{\ct(\mathsf{OE},\pi(\Gamma))}\leq\frac{\ct(w,\Gamma')}{\ct(\mathsf{OE},\pi(\Gamma'))}.\]
We will construct such a $\Gamma'$ from $\Gamma$. At each step, we will take care that $f(\Gamma):=\frac{\ct(w,\Gamma)}{\ct(\mathsf{OE},\pi(\Gamma))}$ does not decrease. Write $w=w_1\cdots w_\ell$.
\begin{enumerate}
    \item First, eliminate from $\Gamma$ all of the letters which are part of no appearance of $w$. Note that eliminating a letter can never increase the number of appearances of words of the form $\mathsf{OE}$ in the grid. So, the numerator of $f(\Gamma)$ does not change, while the denominator does not increase. Call the resulting grid $\Gamma_1$.

    \item Second, consider the appearances of $\mathsf{EE}$ or $\mathsf{OO}$ in $\pi(\Gamma_1)$. Since every letter of $\Gamma_1$ is part of an appearance of $w$, and $\pi(w)$ contains neither $\mathsf{OO}$ nor $\mathsf{EE}$, both of these letters must be the ends or starts of appearances of $w$. Therefore every two-letter word appearing in the grid not of the form $\mathsf{OE}$ or $\mathsf{EO}$ is of the form $w_1w_\ell$, $w_\ell w_1$, $w_1w_1$, or $w_\ell w_\ell$. From every occurrence of $w_1w_1$ or $w_\ell w_\ell$, delete one of the two letters. Call the resulting grid $\Gamma_2$. We have $f(\Gamma_1)=f(\Gamma_2)$, and the only remaining two-letter words not of the form $\mathsf{OE}$ or $\mathsf{EO}$ in $\pi(\Gamma_2)$ arise from two-letter words of the form $w_1w_\ell$ or $w_\ell w_1$.

    \item If $\ell$ is even, then $\Gamma_2$ is parity-respecting, and we are done. Otherwise, we ``insert dividers" into $\Gamma_2$ between each consecutive occurrence of $w_1$ and $w_\ell$, splitting the letters in $\Gamma_2$ into words $s_1,\ldots,s_k$, each beginning and ending with an element of $\{w_1,w_\ell\}$. Now, for each $s_i$, let $s'_i$ be the word of length $2\len(s_i)-2$ constructed by overlapping $s_i$ and $s_i^{\mathrm{rev}}$ at their first and last letters, and let $\Gamma_2^{(i)} = \Grid(s'_i)$.\footnote{For example, if $s_i=\mathsf{ABABC}$, then $\Gamma_2^{(i)} =  \Grid(\mathsf{ABABCBAB})$.} We compute
    
    \begin{align*}
    \sum_{i=1}^k\ct(w,\Gamma_2^{(i)})&=2\sum_{i=1}^k\ct(w,s_i)=2\ct(w,\Gamma_2);\\
    \sum_{i=1}^k\ct\big(\mathsf{OE},\pi(\Gamma_2^{(i)})\big)&=2\sum_{i=1}^k\ct(\mathsf{OE},\pi(s_i))=2\ct(\mathsf{OE},\pi(\Gamma_2)).
    \end{align*}
    Using that
    \[\frac{\beta_1+\cdots+\beta_k}{\gamma_1+\cdots+\gamma_k}\leq\max\left(\frac{\beta_1}{\gamma_1},\ldots,\frac{\beta_k}{\gamma_k}\right)\]
    whenever $\beta_1,\ldots,\beta_k\geq 0$ and $\gamma_1,\ldots,\gamma_k>0$, we conclude that
    \[f(\Gamma_2)\leq\max_{1\leq i\leq k}f\big(\Gamma_2^{(i)}\big).\]
    Set $\Gamma':=\Gamma_2^{(i)}$ for some $i$ which maximizes $f\big(\Gamma_2^{(i)}\big)$.
\end{enumerate}

We have $f(\Gamma)\leq f(\Gamma')$, and that the letters of $\pi(\Gamma')$ alternate $\mathsf O$ and $\mathsf E$. So, $\Gamma'$ is parity-respecting, and thus $f(\Gamma')\leq f(\Gamma_0)$. This finishes the verification of (c).
\end{proof}

\subsection{Words like \texorpdfstring{$\A^k\B^k$}{AkBk}}\label{subsec:proj-k-rep}
We can prove \cref{thm:AkBk} by using methods similar to our proof of \cref{prop:proj-reductions}. Given a word $w$, write $w^{(k)}$ for the word formed by repeating each letter of $w$ $k$ times.\footnote{For example, if $w=\mathsf{AABC}$, then $w^{(2)}=\mathsf{AAAABBCC}$.}

\begin{proof}[Proof of \cref{thm:AkBk}]
    Recall our assumption that $w$ is $d$-stackable; we must show that $w^{(k)}$ is also $d$-stackable.

    Let $\Gamma$ be any $d$-dimensional grid of shape $\prod_{i=1}^d \Zmod{n_i}$. Replace $\Gamma$ by an equivalent larger grid if necessary so that $k \mid n_i$ for all $i$. For every vector $\bv = (v_1, \dots, v_k)$ in $\{0, \ldots, k-1\} ^d$, let $\Gamma_{\bv}$ be the grid of size $\frac{1}{k^d} |\Gamma|$ and shape $\prod_{i=1}^d \Zmod{(n_i/k)}$ given by $\Gamma_\mathbf{v}(z_1, \dots, z_d) = ((z_1-v_1)/k, \dots, (z_d-v_d)/k)$.

    Now, each appearance of $w^{(k)}$ with at least one cell equivalent to $\bv\bmod k$ in $\Gamma$ corresponds to exactly one appearance of $w$ in $\Gamma_{\bv}$, and this correspondence is injective. Moreover, each appearance of $w^{(k)}$ in $\Gamma$ has cells equivalent to $k$ elements of $\{0, \ldots, k-1\}^d$, as each appearance lies in a single line. Therefore we have 
    \[\ct(w^{(k)}, \Gamma) \le \frac{1}{k} \sum_{\bv} \ct(w, \Gamma_{\bv})\le \frac{1}{k} \sum_{\bv} C_d(w) |\Gamma_{\bv}| = \frac{1}{k} C_d(w) |\Gamma|.\]
    Since $\frac1kC_d(w)$ is the concentration of $w^{(k)}$ in some grid constant in all but one coordinate (by the $d$-stackability of $w$), we have that some such grid is $w^{(k)}$-extremal. From this, the $d$-stackability of $w^{(k)}$ follows.
    \end{proof}

    Though this proof doesn't follow \cref{prop:proj-reductions} exactly, it is closely related. Instead of bounding $\ct(w^{(k)}, d, \Gamma)$ in terms of the count of some word $w'$ in a grid of the form $\pi(\Gamma)$, we bounded $\ct(w^{(k)}, d, \Gamma)$ in terms of the count of $w$ across many grids $\Gamma_{\bv}$.

    More broadly, this is not the only possible application of the general idea of \cref{prop:proj-reductions}. We have not attempted to determine the broadest possible generalization, nor have we attempted to classify all the cases for which we can answer \cref{qn:main} using \cref{prop:proj-reductions} or a close variant together with our other results.

\section{Local analysis}\label{sec:local}

In this section, we will provide a ``local'' method to give upper-bound $C_d(w)$. We begin by giving a simple example (a very special case of \cref{thm:AK-restated}) to demonstrate the method. Next, we present the general form of the method, and finally we provide specific applications, computing $C_d(w)$ for some pairs $(w,d)$ which our other methods are not able to cover.

\subsection{Toy application: the word \texorpdfstring{$\mathsf{AB}$}{AB}}\label{subsec:local-AB}

To exemplify our technique, we give an elementary proof of the result of \cite{alonkravitz} that $\mathsf{AB}$ is $2$-stackable. (The existence of such an elementary proof was noted in \cite{alonkravitz}.)

\begin{lemma}\label{lem:AB-elementary} It holds that $C_2(\mathsf{AB})=3$.
\end{lemma}
\begin{proof} Since $C_1(\mathsf{AB}) = 1$, it suffices to show that $C_2(\mathsf{AB})\le 3$. Let $\Gamma$ be a grid on the alphabet $\{\A,\B\}$ of shape $\Zmod{m}\times\Zmod{n}$; we will show that $\ct(\mathsf{AB},\Gamma)=|\mathcal A(\mathsf{AB},\Gamma)|\leq 3mn$. Write $\be_1=(1,0)$ and $\be_2=(0,1)$.

We now produce a ``local'' bound on appearances of $\mathsf{AB}$. Fix some $\bu\in\Zmod{m}\times\Zmod{n}$, and consider the three letters $\Gamma(\bu)$, $\Gamma(\bu+\be_1)$, and $\Gamma(\bu+\be_2)$. Any pair of these letters which is distinct gives rise to one appearance of $\mathsf{AB}$ in $\Gamma$. Since the three cannot all be distinct, we have
\begin{align}\notag\big|\mathcal A(\mathsf{AB},\Gamma)&\cap\{(\bu,\be_1),(\bu,\be_2),(\bu+\be_1,-\be_1),(\bu+\be_2,-\be_2)\}\big|\\
&+2\big|\mathcal A(\mathsf{AB},\Gamma)\cap\{(\bu+\be_1,\be_2-\be_1),(\bu+\be_1,\be_1-\be_2)\}\big|\leq 3.\label{eq:AB-local-bound}
\end{align}
(The left side is $0$ if all three of the letters are the same, $2$ if the two letters diagonally adjacent are the same and the third, $\Gamma(\bu)$, is different, and $3$ otherwise.) For each $\bv\in\{-1,0,1\}^2\setminus\{\zero\}$, let $N_\bv$ be the number of appearances of $\mathsf{AB}$ in $\Gamma$ reading in the direction of $\bv$. Summing \eqref{eq:AB-local-bound} over all $\bu\in\Zmod{m}\times\Zmod{n}$ gives
\begin{equation}\label{eq:AB-local-bound-total-1}
N_{\be_1}+N_{\be_2}+N_{-\be_1}+N_{-\be_2}+2N_{\be_2-\be_1}+2N_{\be_1-\be_2}\leq 3mn.
\end{equation}
The same argument works identically if $\be_1$ is replaced by $-\be_1$:
\begin{equation}\label{eq:AB-local-bound-total-2}
N_{\be_1}+N_{\be_2}+N_{-\be_1}+N_{-\be_2}+2N_{\be_1+\be_2}+2N_{-\be_1-\be_2}\leq 3mn.
\end{equation}
Averaging \eqref{eq:AB-local-bound-total-1} and \eqref{eq:AB-local-bound-total-2} gives
\[\ct(\mathsf{AB},\Gamma)=\sum_{\bv\in\{-1,0,1\}^2\setminus\{\zero\}}N_{\bv}\leq 3mn.\qedhere\]
\end{proof}

\subsection{General method}

The proof of \cref{lem:AB-elementary} can be generalized to the following.

\begin{proposition}\label{prop:local} For each $1\leq j\leq d$, let $V_j$ be the set of $\bv\in\{-1,0,1\}^d\setminus\{\zero\}$ satisfying $\|\bv\|_2=\sqrt j$.

Let $w$ be a word of length $\ell$, let $d$ be a positive integer, let $K,M>0$ be real numbers, and let $F\colon\ZZ^d\times(\{-1,0,1\}^d\setminus\{\zero\})\to\RR$ be a weight function with finite support such that
\begin{enumerate}[label={(\roman*)}]
    \item for each $1\leq j\leq d$, we have
    \[\frac1{|V_j|}\sum_{\bv\in V_j}\sum_{p\in\ZZ^d}F(p,\bv)=K;\]
    \item for every infinite grid $\mathcal{G}$ of shape $\ZZ^d$, we have
    \[\sum_{(p, \bv) \in \mathcal A(w, \mathcal{G})}F(p,\bv)\leq M.\]
\end{enumerate}
Then $C_d(w)\leq M/K$.    
\end{proposition}

Before proving \cref{prop:local}, we make some remarks about its content. The rather peculiar form of (i) is due to the fact that the different ``types'' of directions in $\{-1,0,1\}^d\setminus\{\zero\}$ (i.e.~the equivalence classes modulo symmetries of $\ZZ^d$) are weighted by $|V_j|$ in $\ct(w,\Gamma)$ (so that each individual direction has equal weight). In the proof of \cref{lem:AB-elementary}, this corresponds to the doubling of the contribution of the ``diagonally adjacent'' appearances in \eqref{eq:AB-local-bound}, so that when averaging \eqref{eq:AB-local-bound-total-1} and \eqref{eq:AB-local-bound-total-2} each direction comes with the same weight. Condition (ii) is simply a generalization of \eqref{eq:AB-local-bound}.

When applying this proposition, it is useful to note that the expression in condition (ii) depends only on the letters assigned to cells in
\[\bigcup_{(p,\bv)\in\supp F}\{p,p+\bv,\ldots,p+(\ell-1)\bv\},\]
where $\ell=\len(w)$. In particular, since $F$ has finite support, verifying (ii) is a finite problem.

\begin{proof}[Proof of \cref{prop:local}] Let $\Gamma$ be a grid in $d$ dimensions of shape $G=\prod_{i=1}^d\Zmod{n_i}$. Without loss of generality (by passing to an equivalent grid), we may assume that each $n_i$ is much larger than $\lvert\supp F\rvert$. We shall average (ii) over all isometric images (translates and rotates) of $F$ within $\Gamma$. Each appearance of $w$ in $\Gamma$ will be counted with weight exactly $K/|\Gamma|$; the fact that this average is at most $M$ will give that $C_d(w,\Gamma)\leq M/K$.

Let $\Pi\cong(\ZZ/2\ZZ)^d\rtimes S_d$ be the group of isometries of $\RR^d$ which fix $\zero$ and preserve $\ZZ^d$, i.e.~the symmetries of $\ZZ^d$ as a lattice. Define the quantity
\[T:=\EE_{\bu\in G}\EE_{\pi\in\Pi}\sum_{(p,\bv)\in\mathcal A(w,\Gamma)}F(\pi(p)-\bu,\pi(\bv)).\]
On one hand, by considering the infinite grid $\mathcal G$ given by $\mathcal G(x)=\Gamma(\pi(x)-\bu)$, we have
\[\sum_{(p,\bv)\in\mathcal A(w,\Gamma)}F(\pi(p)-\bu,\pi(\bv))=\sum_{\substack{(\pi(p)-\bu,\pi(\bv))\\\in\mathcal A(w,\mathcal G)}}F(\pi(p)-\bu,\pi(\bv))=\sum_{(p',\bv')\in\mathcal A(w,\mathcal G)}F(p',\bv')\leq M,\]
so $T\leq M$. On the other hand,
\begin{align*}
T
&=\sum_{(p,\bv)\in\mathcal A(w,\Gamma)}\EE_{\pi\in\Pi}\EE_{\bu\in G}F(\pi(p)-\bu,\pi(\bv))\\
&=\sum_{(p,\bv)\in\mathcal A(w,\Gamma)}\EE_{\pi\in\Pi}\EE_{\bu'\in G}F(\bu',\pi(\bv))&(\text{set $\bu'=\pi(p)-\bu$})\\
&=\sum_{(p,\bv)\in\mathcal A(w,\Gamma)}\EE_{\bu'\in G}\EE_{\bv'\in V_{\|\bv\|_1}}F(\bu',\bv')&\begin{array}{r}\text{(the distribution of $\pi(\bv)$ as $\pi$\ \ }\\\text{ varies is uniform in $V_{\|\bv\|_1}$)}\end{array}\\
&=\frac1{|\Gamma|}\sum_{(p,\bv)\in\mathcal A(w,\Gamma)}\left(\frac1{|V_{\|\bv\|_1}|}\sum_{\bv'\in V_{\|\bv\|_1}}\sum_{\bu'\in G}F(\bu',\bv')\right)\\
&=\frac1{|\Gamma|}\sum_{(p,\bv)\in\mathcal A(w,\Gamma)}K=\frac{K\ct(w,\Gamma)}{|\Gamma|}=K\cdot C_d(w,\Gamma).
\end{align*}
We conclude that $C_d(w,\Gamma)\leq M/K$; taking the supremum over $\Gamma$ gives the result.
\end{proof}

When $d=2$, we will represent and define these weight functions $F$ with diagrams. For each pair $(p,\bv)$ in the support of $F$, we draw an arrow beginning at the point $p$ in the direction of $\bv$, labeled by $F(p,\bv)$. See \cref{fig:AB-2d-wt} for an example of such a diagram, which defines the same function $F$ as used in \cref{lem:AB-elementary}.

\begin{figure}
\begin{tikzpicture}[scale=0.7]
    \localDiagram{0/0/1/0/1/below,
                0/0/0/1/1/left,
                1/0/-1/0/1/below,
                0/1/0/-1/1/left,
                1/0/-1/1/2/above right,
                0/1/1/-1/2/above right}
\end{tikzpicture}
\caption{The weight function used implicitly in the proof of \cref{lem:AB-elementary} to upper-bound $C_2(\mathsf{AB})$.}
\label{fig:AB-2d-wt}
\end{figure}
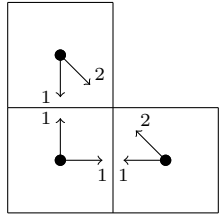

\begin{remark}[Elementary proofs in higher dimensions]
In \cite{alonkravitz}, Alon and Kravitz remark that the computation of $C_2(\mathsf{AB})$ can be done in an elementary way, but that their elementary approach does not generalize to higher dimensions. Using \cref{prop:local}, we can partially remedy this: for \emph{each} dimension $d$, we can reduce the proof that $C_d(\mathsf{AB})=3^{d-1}$ to a finite computation.

Consider the weight function given by
\[F(p,\bv)=\begin{cases}2^{\|\bv\|_1}&p,p+\bv\in\{0,1\}^d\\0&\text{otherwise}.\end{cases}\]
This function satisfies (i) with $K=2^d$. The condition
\begin{equation}\label{eq:finite-case-check}
F\text{ satisfies (ii) with }M=3^{d-1}2^d
\end{equation}
is exactly the statement that any grid of shape $(\ZZ/2\ZZ)^d$ has at most $3^{d-1}2^d$ copies of $\mathsf{AB}$. In particular, \eqref{eq:finite-case-check} is true by \cref{thm:AK-restated}, and thus can be verified by checking each of the $2^{2^d}$ cases. From \eqref{eq:finite-case-check} and \cref{prop:local}, the determination of $C_d(\mathsf{AB})$ is immediate. For each fixed dimension $d$, this finite procedure furnishes an elementary proof that $C_d(\mathsf{AB})=3^{d-1}$. 
\end{remark}

\subsection{Linear programming: how to use \texorpdfstring{\cref{prop:local}}{Proposition 6.2}}\label{subsec:local-philosophy}
Applying \cref{prop:local} has two central difficulties: choosing the weight function $F$ and verifying the bound (ii). Before presenting our applications of \cref{prop:local}, we give a bit of insight into how we performed these steps. Our procedure is as follows:

\begin{enumerate}
    \item Choose some set $S\subset\ZZ^d$ --- we think of this set as the set of cells whose letters are allowed to vary.
    
    \item Select some letter $A$ from $w$. Now, let
    \[\mathcal F=\left\{\begin{array}{l}(p,\bv)\in \ZZ^d\times\{-1,0,1\}^d\setminus\{\zero\}\colon\\\ \ \ \ \text{ for all }1\leq i\leq\len(w),\ p+(i-1)\bv\in S\text{ or }w_i=A\end{array}\right\}.\]
    The set $\mathcal F$ consists of all pairs $(p,\bv)$ for which the event $(p,\bv)\in\mathcal A(w,\mathcal G)$ depends only on the letters assigned to elements of $S$, as well as whether other cells are assigned $A$.

    \item By the definition of $\mathcal F$, for any weight function $F$ with $F\geq 0$ and $\supp F\subset\mathcal F$,
    \[\sum_{(p,\bv)\in\mathcal A(w,\mathcal G)}F(p,\bv)\]
    is maximized at some grid $\mathcal G$ with $\mathcal G(x)=A$ for all $x\not\in S$. Let $\mathscr G$ be the set of such grids.

    \item The problem
    \begin{align}
    \text{minimize\ }&M_1\label{eq:linprog}\\
    \notag\text{subject to\ }
    &\supp F\subset\mathcal F,\\
    \notag&F(p,\bv)\geq 0\text{ for all }(p,\bv)\in\mathcal F,\\
    \notag&\text{$F$ satisfies condition (i) of \cref{prop:local} with $K=1$, and}\\
    &\sum_{(p,\bv)\in\mathcal A(w,\mathcal G)}F(p,\bv)\leq M_1\text{ for all }\mathcal G\in\mathscr G\label{eq:linprog-constraint}
    \end{align}
    is a linear program with $|\mathcal F|+1$ variables and $|\mathscr G|=2^{|S|}$ conditions. Solving it gives a weight function $F$ which satisfies (i) with $K=1$ and $M=M_1$, and so gives a bound $C_d(w)\leq M_1$.
\end{enumerate}

The utility of this method is constrained by the fact that the time to produce and solve this linear program grows exponentially in $|S|$. To make the computations more feasible, we are able to rely on the symmetry of the problem to speed up the computation. If $S$ is chosen to have some symmetry group $\Pi$, then the linear program \eqref{eq:linprog} admits an optimum $F$ which is preserved under action by $\Pi$ --- such an optimum can be found by starting with any optimal $F$ and averaging over its $\Pi$-orbit. This allows us to reduce the number of variables in the linear program: each $\Pi$-orbit of $\mathcal F$ can be collapsed into a single variable. This reduction in the number of variables reduces the number of distinct constraints given by \eqref{eq:linprog-constraint}. In particular, grids $\mathcal G$ in the same $\Pi$-orbit give the same constraint on $F$, so we only need to enumerate the constraints for one grid in each $\Pi$-orbit. This substantially reduces the amount of time necessary to write down a linear program equivalent to \eqref{eq:linprog}.

Finally, we are able to save some time on solving the (sleeker) linear program. The number of constraints is vastly greater than the number of variables. So, we begin by sampling a small number of constraints and solving the linear program with only these constraints (ideally, we choose few enough constraints that this step can be done quickly), to obtain a variable vector $\bx_1$. Then, we add in all constraints which are violated by $\bx_1$, and solve again, obtaining a variable vector $\bx_2$. At some point, we will have a vector $\bx_t$ which optimizes our linear program with a subset of the constraints, and violates no constraints. Such a vector is necessarily an optimum to the original linear program. This iterative procedure reduces both the time and space required to solve the linear program. (Such ``sparsification'' of linear programs is an active research area, and there is a rather effective general-purpose algorithm due to Clarkson \cite{clarkson}. Our algorithm is rather more primitive, but has sufficed for all of our purposes in this work.)

\subsection{Applications of \texorpdfstring{\cref{prop:local}}{Proposition 6.2}}\label{subsec:local-applications}

Now that we have explained the method whereby we determine appropriate weight functions $F$, we are ready to apply this method to compute $C_2(\mathsf{ABB},2)$, $C_2(\mathsf{ABCC},2)$, and $C_2(\mathsf{BABBB},2)$. 

\begin{lemma}\label{lem:ABB} The word $\mathsf{ABB}$ is $2$-stackable.
\end{lemma}
\begin{proof} Since $C_1(\mathsf{ABB}) = 2/3$, it suffices to show that $C_2(\mathsf{ABB})\le 2$. Consider the function $F\colon\ZZ^2\times(\{-1,0,1\}^2\setminus\{\zero\})\to \RR$ described by the diagram in \cref{fig:ABB-2d-wt}.
\begin{figure}
\begin{tikzpicture}[scale=1]
    \localDiagram{0/0/1/1/2/below right,
                  0/0/-1/1/2/below left,
                  0/0/-1/-1/2/above left,
                  0/0/1/-1/2/above right,
                  1/0/-1/0/2/below,
                  1/0/-1/1/1/above right,
                  1/0/-1/-1/1/below right,
                  -1/0/1/0/2/below,
                  -1/0/1/1/1/above left,
                  -1/0/1/-1/1/below left,
                  0/1/0/-1/2/right,
                  0/1/-1/-1/1/above left,
                  0/1/1/-1/1/above right,
                  0/-1/0/1/2/right,
                  0/-1/-1/1/1/below left,
                  0/-1/1/1/1/below right,
                  1/1/-1/0/2/above,
                  1/1/0/-1/2/right,
                  1/1/-1/-1/2/above left,
                  -1/1/1/0/2/above,
                  -1/1/0/-1/2/left,
                  -1/1/1/-1/2/above right,
                  -1/-1/1/0/2/below,
                  -1/-1/0/1/2/left,
                  -1/-1/1/1/2/below right,
                  1/-1/-1/0/2/below,
                  1/-1/0/1/2/right,
                  1/-1/-1/1/2/below left}
\end{tikzpicture}
\caption{The weight function used to upper-bound $C_2(\mathsf{ABB})$.}
\label{fig:ABB-2d-wt}
\end{figure}
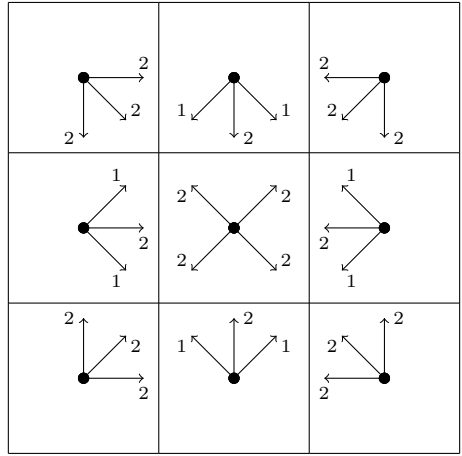
This function satisfies (i) of \cref{prop:local} with $K=6$, so we need to prove that for any grid $\mathcal G$ of dimension 2,
\begin{equation}\label{eq:ABB-local}
f(\mathcal G):=\sum_{(p,\bv)\in \mathcal A(\mathsf{ABB},\mathcal G)}F(p,\bv)\leq 12.
\end{equation}
As $F$ falls into the paradigm described by \cref{subsec:local-philosophy} with $S=\{-1,0,1\}^2$ and $A=\B$, it suffices to show \eqref{eq:ABB-local} for grids $\mathcal G$ where $\mathcal G(x)=\B$ for $x\not\in S$. There are $2^9=512$ such grids. We present a by-hand verification of \eqref{eq:ABB-local} in \cref{sec:appendix}, where it is presented as \cref{lem:ABB-by-hand}.
\end{proof}

\begin{lemma}\label{lem:ABCC-2d} The word $\mathsf{ABCC}$ is $2$-stackable.
\end{lemma}

\begin{proof} Since $C_1(\mathsf{ABCC})=2/5$, it suffices to show that $C_2(\mathsf{ABCC})\le 6/5$. Consider the function $F\colon\ZZ^2\times(\{-1,0,1\}^2\setminus\{\zero\})\to\RR$ described by the diagram in \cref{fig:ABCC-2d-wt}.
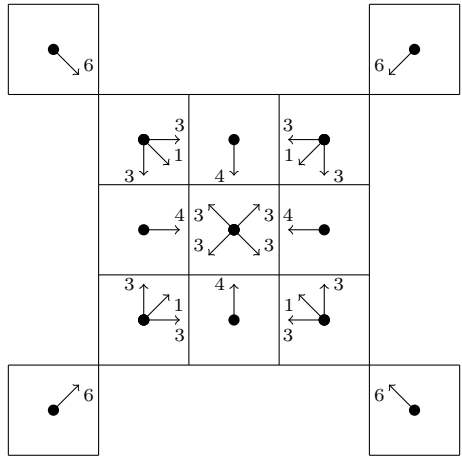
\begin{figure}
\begin{tikzpicture}[scale=0.6]
    \localDiagram{-2/-2/1/1/6/below right,
                  2/-2/-1/1/6/below left,
                  2/2/-1/-1/6/above left,
                  -2/2/1/-1/6/above right,
                  -1/-1/1/0/3/below,
                  -1/-1/0/1/3/left,
                  1/-1/-1/0/3/below,
                  1/-1/0/1/3/right,
                  1/1/-1/0/3/above,
                  1/1/0/-1/3/right,
                  -1/1/1/0/3/above,
                  -1/1/0/-1/3/left,
                  -1/-1/1/1/1/below right,
                  1/-1/-1/1/1/below left,
                  1/1/-1/-1/1/above left,
                  -1/1/1/-1/1/above right,
                  -1/0/1/0/4/above,
                  1/0/-1/0/4/above,
                  0/-1/0/1/4/left,
                  0/1/0/-1/4/left,
                  0/0/1/1/3/below right,
                  0/0/-1/1/3/below left,
                  0/0/-1/-1/3/above left,
                  0/0/1/-1/3/above right}
\end{tikzpicture}
\caption{The weight function used to upper-bound $C_2(\mathsf{ABCC})$.}
\label{fig:ABCC-2d-wt}
\end{figure}

This function satisfies (i) of \cref{prop:local} with $K=10$, so we need only show that it satisfies (ii) with $M=12$. Similarly to the case of $C_2(\mathsf{ABB})$, we defer this to \cref{sec:appendix}, where the verification is presented as \cref{lem:ABCC-by-hand}.
\end{proof}

\begin{lemma}\label{lem:BABBB} We have $C_2(\mathsf{BABBB})=8/5$. In particular, since $C_1(\mathsf{BABBB})=1/2$, the word $\mathsf{BABBB}$ is neither $2$-stackable nor $2$-slantable.    
\end{lemma}
\begin{proof} For the lower bound, consider the following grid $\Gamma$ of shape $\Zmod{5}\times\Zmod{5}$:
\[
\begin{array}{ccccc}
\A & \B & \B & \B & \B \\
\B & \B & \B & \A & \B \\
\B & \A & \B & \B & \B \\
\B & \B & \B & \B & \A\\
\B & \B & \A & \B & \B
\end{array}
\]
This grid has five occurrences of $\A$, and in each direction around these occurrences, there is an appearance of $\mathsf{BABBB}$. So, $\Gamma$ has $40$ appearances of $\mathsf{BABBB}$, and thus $c_2(\mathsf{BABBB},\Gamma)=8/5$.

For the upper bound, consider the function $F\colon\ZZ^2\times(\{-1,0,1\}^2\setminus\{\zero\})\to\RR$ described by the diagram in \cref{fig:BABBB-2d-wt}. To highlight the elements of $S$ (where $S$ is as in \cref{subsec:local-philosophy}) as the relevant cells, we draw arrows starting from the \emph{second} character of the word, i.e.~the only $\A$, rather than the first character $\B$. The function $F$ so described satisfies condition (i) of \cref{prop:local} with $K=40$. We have verified by computer check that it satisfies condition (ii) with $M=64$. (We are not aware of a way to make this finite case-check short enough to be reasonably performed by a human.)
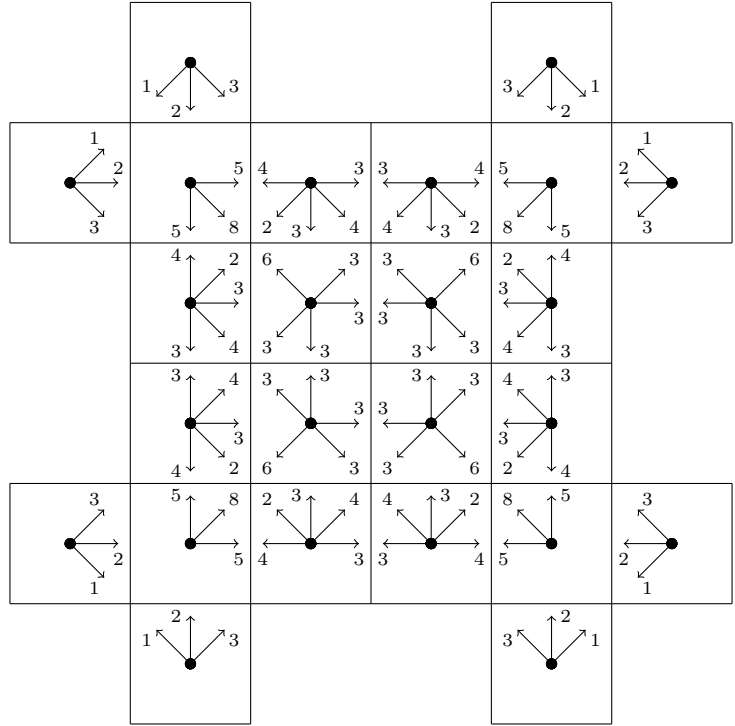
\begin{figure}
\begin{tikzpicture}[scale=0.8]
    \localDiagram{2/3/1/0/2/below,
            7/3/-1/0/2/below,
            2/6/1/0/2/above,
            7/6/-1/0/2/above,
            3/2/0/1/2/left,
            3/7/0/-1/2/left,
            6/2/0/1/2/right,
            6/7/0/-1/2/right,
		  3/3/1/0/5/below,
		  3/6/1/0/5/above,
		  6/3/-1/0/5/below,
		  6/6/-1/0/5/above,
		  3/3/0/1/5/left,
		  6/3/0/1/5/right,
		  3/6/0/-1/5/left,
		  6/6/0/-1/5/right,
		  3/4/1/0/3/below,
		  6/4/-1/0/3/below,
		  3/5/1/0/3/above,
		  6/5/-1/0/3/above,
		  4/3/0/1/3/left,
		  4/6/0/-1/3/left,
		  5/3/0/1/3/right,
		  5/6/0/-1/3/right,
		  3/4/0/1/3/left,
		  6/4/0/1/3/right,
		  3/5/0/-1/3/left,
		  6/5/0/-1/3/right,
		  4/3/1/0/3/below,
		  4/6/1/0/3/above,
		  5/3/-1/0/3/below,
		  5/6/-1/0/3/above,
		  3/5/0/1/4/left,
		  3/4/0/-1/4/left,
		  6/5/0/1/4/right,
		  6/4/0/-1/4/right,
		  5/3/1/0/4/below,
		  4/3/-1/0/4/below,
		  5/6/1/0/4/above,
		  4/6/-1/0/4/above,
		  4/4/1/0/3/above,
		  4/5/1/0/3/below,
		  5/4/-1/0/3/above,
		  5/5/-1/0/3/below,
		  4/4/0/1/3/right,
		  5/4/0/1/3/left,
		  4/5/0/-1/3/right,
		  5/5/0/-1/3/left,
		  2/3/1/1/3/above left,
		  2/6/1/-1/3/below left,
		  7/3/-1/1/3/above right,
		  7/6/-1/-1/3/below right,
		  3/2/1/1/3/below right,
		  6/2/-1/1/3/below left,
		  3/7/1/-1/3/above right,
		  6/7/-1/-1/3/above left,
            2/3/1/-1/1/below left,
            2/6/1/1/1/above left,
            7/3/-1/-1/1/below right,
            7/6/-1/1/1/above right,
            3/2/-1/1/1/below left,
            6/2/1/1/1/below right,
            3/7/-1/-1/1/above left,
            6/7/1/-1/1/above right,
		  3/3/1/1/8/above right,
		  3/6/1/-1/8/below right,
		  6/3/-1/1/8/above left,
		  6/6/-1/-1/8/below left,
		  3/4/1/1/4/above right,
		  3/5/1/-1/4/below right,
		  6/4/-1/1/4/above left,
		  6/5/-1/-1/4/below left,
		  4/3/1/1/4/above right,
		  5/3/-1/1/4/above left,
		  4/6/1/-1/4/below right,
		  5/6/-1/-1/4/below left,
            4/3/-1/1/2/above left,
            4/6/-1/-1/2/below left,
            5/3/1/1/2/above right,
            5/6/1/-1/2/below right,
            3/4/1/-1/2/below right,
            6/4/-1/-1/2/below left,
            3/5/1/1/2/above right,
            6/5/-1/1/2/above left,
		  4/5/1/1/3/above right,
		  4/4/1/-1/3/below right,
		  5/5/-1/1/3/above left,
		  5/4/-1/-1/3/below left,
		  5/4/1/1/3/above right,
		  4/4/-1/1/3/above left,
		  5/5/1/-1/3/below right,
		  4/5/-1/-1/3/below left,
		  5/5/1/1/6/above right,
		  5/4/1/-1/6/below right,
		  4/5/-1/1/6/above left,
		  4/4/-1/-1/6/below left}
\end{tikzpicture}
\caption{The weight function used to upper-bound $C_2(\mathsf{BABBB})$. Note that each arrow in this figure starts at the $\A$ and points towards the $\mathsf{BBB}$ (and away from the single $\B$) in the appearance of $w$ the arrow indicates.}
\label{fig:BABBB-2d-wt}
\end{figure}
\end{proof}

\section{Short words in two dimensions}\label{sec:short-words}

The techniques we have presented thus far are enough to compute $C_d(w)$ for many short words. In particular, we are able to determine which words of length at most four are $2$-stackable. We do the same for many words of length five.

\subsection{Words of length at most four}\label{subsec:short-up-to-4} Combining the results of \cref{sec:reductions,sec:local}, we may show \cref{thm:short-words}.
\begin{proof}[Proof of \cref{thm:short-words}]
    Recall that we exclude words with only one distinct letter.

    There is only one isomorphism class of words with two letters, the class containing $\mathsf{AB}$. This isomorphism class is $2$-stackable (in fact, $d$-stackable) by \cref{thm:AK-restated} (i.e., the main result of Alon and Kravitz \cite{alonkravitz}).

    There are three isomorphism classes of words with three letters, namely, the class containing $\mathsf{ABA}$, the class containing $\mathsf{ABB}$, and the class containing $\mathsf{ABC}$. The first is $2$-stackable (and $d$-stackable) by \cref{thm:odd-even}. The second is $2$-stackable by \cref{lem:ABB}. The third is $2$-stackable (and $d$-stackable) by \cref{thm:AK-restated}.

    There are 10 isomorphism classes of words with four letters, namely
    \begin{itemize}
        \item five isomorphism classes of words with two distinct letters: those containing each of
        \begin{itemize}
            \item $\mathsf{AABB}$, which is $d$-stackable by \cref{thm:AkBk}, since $\mathsf{AB}$ is $d$-stackable by \cref{thm:AK-restated}.
            \item $\mathsf{ABAB}$, which is $d$-stackable by \cref{thm:odd-even}.
            \item $\mathsf{ABBA}$, which is $2$-stackable by \cref{lem:ABB-reduction}, since $\mathsf{ABB}$ is $2$-stackable by \cref{lem:ABB}.
            \item $\mathsf{AABA}$, which is $2$-stackable by \cref{lem:ABB-reduction}, since $\mathsf{AABA}$ and $\mathsf{BABB}$ are isomorphic and $\mathsf{ABB}$ is $2$-stackable by \cref{lem:ABB}.
            \item $\mathsf{AAAB}$, which is not $2$-stackable. It is isomorphic to $\mathsf{ABBB}$; see \cref{fig:ABBB}. 
        \end{itemize}
        \item four isomorphism classes of words with three distinct letters: those containing
        \begin{itemize}
            \item $\mathsf{AABC}$, which is $2$-stackable by \cref{lem:ABCC-2d}, since $\mathsf{AABC}$ and $\mathsf{ABCC}$ are isomorphic.
            \item $\mathsf{ABAC}$, which is $d$-stackable by \cref{thm:odd-even}.
            \item $\mathsf{ABBC}$, which is $2$-stackable by \cref{lem:ABB-reduction}, since $\mathsf{ABB}$ is $2$-stackable by \cref{lem:ABB}.
            \item $\mathsf{ABCA}$, which is $2$-stackable by \cref{lem:ABB-reduction}, since $\mathsf{ABB}$ is $2$-stackable by \cref{lem:ABB}.
        \end{itemize}
        \item one isomorphism class of words with 4 distinct letters, containing \begin{itemize}
            \item $\mathsf{ABCD}$, which is $d$-stackable by \cref{thm:AK-restated}.
        \end{itemize}
    \end{itemize}

\begin{figure}
    \centering
    $\begin{matrix}
\A & \B & \B & \B & \B \\
\B & \B & \B & \A & \B \\
\B & \A & \B & \B & \B \\
\B & \B & \B & \B & \A \\
d \B & \B & \A & \B & \B 
    \end{matrix}$ \hspace{1.8cm}
    $\begin{matrix}
\A & \B & \B & \B \\
\A & \B & \B & \B \\
\A & \B & \B & \B \\
\A & \B & \B & \B 
    \end{matrix}$
    \caption{The grid $\Gamma$ on the left has $c_2(\mathsf{ABBB}, \Gamma) = 8/5.$ This is greater than the concentration of $\mathsf{ABBB}$ in the right grid $\Gamma'$, which is constant in all but one coordinate, and has the maximum concentration $c_2(\mathsf{ABBB}, \Gamma') = 3/2 = 3C_1(\mathsf{ABBB})$ among all grids constant in all but one coordinate.}
    \label{fig:ABBB}
\end{figure}
\end{proof}

\subsection{Words of length five}\label{subsec:short-5} More techniques are needed to determine which words of length 5 are and are not $2$-stackable. Still, we are able to show 2 of the 31 isomorphism classes are not $2$-stackable, and that 14 of the 31 are $2$-stackable. We do not know which of the remaining 15 cases are $2$-stackable or not. This is a natural direction for future work.

The 2 isomorphism classes which we can show are not $2$-stackable contain $\mathsf{ABBBB}$ and $\mathsf{BABBB}$. The grid $\Gamma$ on the left of \cref{fig:ABBB} satisfies $c_2(\mathsf{ABBBB}, \Gamma) = 8/5 > 6/5 = 3 C_1(\mathsf{ABBBB})$ and $c_2(\mathsf{BABBB}, \Gamma) = 8/5 > 6/4 = 3 C_1(\mathsf{BABBB})$, so neither of these words is $2$-stackable. In fact, since $c_2(\mathsf{ABBBB}, \Gamma) = 4\times C_1(\mathsf{ABBBB})$, we have that $\mathsf{ABBBB}$ is 2-slantable.

There are eight isomorphism classes which we can show are $d$-stackable by \cref{thm:odd-even}: the isomorphism classes containing
\[\mathsf{ABABA},\ \mathsf{ABABC},\ \mathsf{ABACA},\ \mathsf{ABACD},\ \mathsf{ABCBA},\ \mathsf{ABCBD},\ \mathsf{ABCDA},\ \mathsf{ABCDE}.\]
(The $d$-stackability of the last two words in the above list also follows from the earlier result of Alon and Kravitz.) The rest of our stackability claims are based on \cref{prop:proj-reductions}. 

We can show that one more isomorphism class is $d$-stackable for every $d$: namely, the class containing $\mathsf{AABBA}$. We reduce to $\mathsf{AABB}$, which we know from the previous proposition is $d$-stackable, using \cref{prop:proj-reductions} with $\pi$ equal to the identity. The result follows from the facts that, in any grid $\Gamma$, $\ct(\mathsf{AABBA},\Gamma)\leq\ct(\mathsf{AABB},\Gamma)$, and that the grid $\Grid(\mathsf{AABB})$ in one dimension is both $\mathsf{AABBA}$-extremal and $\mathsf{AABB}$-extremal.

Three isomorphism classes can be seen to be $2$-stackable by reduction to $\mathsf{ABB}$ or $\mathsf{AAB}$ using \cref{prop:proj-reductions}. Recall $\mathsf{ABB}$ is $2$-stackable by \cref{lem:ABB}.
\begin{itemize}
    \item $\mathsf{AABCC}$ is $2$-stackable: Let $\pi$ be the projection that fixes $\A$ and $\B$ and takes $\C$ to $\A$. Note that $C_1(\mathsf{AABCC})=\frac13=\frac12C_1(\mathsf{AAB})$. Now, (a) there exists an $\mathsf{AABCC}$-extremal grid $\Gamma_0:=\Grid(\mathsf{AABCCB})$ in one dimension such that (b) $\pi(\Gamma_0)$ is $\mathsf{AAB}$-extremal. Moreover, (c) for any grid $\Gamma$ it holds that $2\ct(\mathsf{AABCC}, \Gamma) \le \ct (\mathsf{AAB}, \pi(\Gamma))$. So, since $\mathsf{AAB}$ is $2$-stackable, $\mathsf{AABCC}$ is $2$-stackable as well.
    
    \item $\mathsf{AABAA}$ is $2$-stackable: Let $\pi$ be the identity. Note that $C_1(\mathsf{AABAA})=\frac23=C_1(\mathsf{AAB})$. Now, (a) there exists an $\mathsf{AABAA}$-extremal grid $\Gamma_0:=\Grid(\mathsf{AAB})$ in one dimension which (b) is also $\mathsf{AAB}$-extremal. Moreover, (c) for any grid $\Gamma$ it holds that $\ct(\mathsf{AABAA}, \Gamma) \le \ct(\mathsf{AAB}, \Gamma)$. Since $\mathsf{AAB}$ is $2$-stackable, $\mathsf{AABAA}$ is $2$-stackable as well.
    
    \item $\mathsf{ABCAB}$ is $2$-stackable: Let $\pi$ be the projection that fixes $\A$ and $\B$ and takes $\C$ to $\B$. Note that $C_1(\mathsf{ABCAB})=\frac13=\frac12C_1(\mathsf{ABB})$. Now, (a) there exists an $\mathsf{ABCAB}$-extremal grid $\Gamma_0:=\Grid(\mathsf{ABC})$ in one dimension such that (b) $\pi(\Gamma_0)$ is $\mathsf{ABB}$-extremal. Moreover, (c) for any grid $\Gamma$ it holds that $2\ct(\mathsf{ABCAB}, \Gamma) \le \ct( \mathsf{ABB}, \pi(\Gamma))$. Since $\mathsf{ABB}$ is $2$-stackable, $\mathsf{ABCAB}$ is $2$-stackable as well.
\end{itemize}

Finally, two isomorphism classes can be seen to be $2$-stackable by reduction to a word isomorphic to $\mathsf{ABCC}$ using \cref{prop:proj-reductions}. Recall $\mathsf{ABCC}$ is $2$-stackable by \cref{lem:ABCC-2d}.

\begin{itemize}
    \item $\mathsf{AABCB}$ is $2$-stackable. The relevant data are the identity projection $\pi$, the inequality $\ct(\mathsf{AABCB}, \Gamma) \le \ct(\mathsf{AABC}, \Gamma)$ for any grid $\Gamma$, and the grid $\Gamma_0:=\Grid(\mathsf{AABCB})$. 
    
    \item $\mathsf{ABBAC}$ is $2$-stackable. The relevant data are the identity projection $\pi$, the inequality $\ct(\mathsf{ABBAC}, \Gamma)\le \ct(\mathsf{BBAC}, \Gamma)$ for any grid $\Gamma$, and the grid $\Gamma_0:=\Grid(\mathsf{ABBAC})$.
\end{itemize}

We have made no similar attempt to determine precisely which words of six or more letters we can address using our methods.

\section{Fourier analysis}\label{sec:fourier}

In this section, we will show \cref{thm:no-slant,thm:ABB-Zmod3} using Fourier analysis.

We begin by stating our conventions for the Fourier transform. Write $e(x)=e^{2\pi ix}$. Fix positive integers $n$ and $d$, and let $G=(\ZZ/n\ZZ)^d$ be an abelian group. Given a function $f\colon G\to\CC$, we define the \emph{Fourier transform} $\hat f\colon(\ZZ/n\ZZ)^d\to\CC$ of $f$ by
\[\hat f(\xi)=\EE_{x\in G}\left[f(x)e\left(-\frac{\xi\cdot x}n\right)\right]\]
where the expectation is taken over a uniformly random choice of $x\in G$, and the dot product $\xi\cdot x$ denotes $\sum_{i=1}^d \xi_ix_i$. 
We have the following standard identities:
\begin{align}
f(x)&=\sum_\xi \hat f(\xi)e\left(\frac{\xi\cdot x}n\right)\tag{inversion}\\
\EE_{x\in G} |f(x)|^2&=\sum_\xi|\hat f(\xi)|^2\tag{Parseval}.
\end{align}
(In these identities, both sums run over $\xi\in(\ZZ/n\ZZ)^d$.) In our applications, like in \cite{alonkravitz}, the function $f$ will always be closely related to the \emph{indicator function} $1_{A,\Gamma}$ of a letter $A$ in a grid $\Gamma$, i.e.~the function which evaluates to $1$ at the inputs $p$ for which $A=\Gamma(p)$ and $0$ at all other inputs.

\subsection{\texorpdfstring{$\mathsf{ABB}$}{ABB} in grids of shape \texorpdfstring{$(\Zmod{3})^d$}{(Z/3Z)d}} We now prove \cref{thm:ABB-Zmod3}. The key result is the following lemma on functions on $(\ZZ/3\ZZ)^d$, which we will prove using Fourier-analytic techniques.

\begin{lemma}\label{lem:n3ABB} Let $f\colon(\ZZ/3\ZZ)^d\to[0,1]$, and let $\alpha=\hat f(0)=\EE_xf(x)$. Then
\begin{equation}\label{eq:ABB3d-bound}\EE_{x,y}f(x)\big(1-f(x+y)\big)\big(1-f(x+2y)\big)\leq \frac 32\alpha(1-\alpha)^2.
\end{equation}
\end{lemma}

\begin{proof}[Proof of \cref{thm:ABB-Zmod3} assuming \cref{lem:n3ABB}] Let $\Gamma$ be any grid of shape $(\ZZ/3\ZZ)^d$ on alphabet $\{\A,\B\}$. We must show $c_d(\mathsf{ABB},\Gamma)\leq 2\cdot 3^{d-2}$. Let $f=1_{\B,\Gamma}$ be the indicator function of $\B$ in $\Gamma$. We have that
\[(x,y)\in\mathcal A(\mathsf{ABB},\Gamma)\Longleftrightarrow f(x)=0,\ f(x+y)=f(x+2y)=1.\]
Therefore,
\begin{align*}
c_d(\mathsf{ABB},\Gamma)
&=\EE_{x\in(\ZZ/3\ZZ)^d}\sum_{y\in\{-1,0,1\}^d\setminus\{\zero\}}f(x)\big(1-f(x+y)\big)\big(1-f(x+2y)\big)\\
&=3^d\EE_{x,y\in(\ZZ/3\ZZ)^d}f(x)\big(1-f(x+y)\big)\big(1-f(x+2y)\big),
\end{align*}
where we have used that $f(x)(1-f(x+y))(1-f(x+2y))=0$ when $y=0$, since $f(x)\in\{0,1\}$. So, by \cref{lem:n3ABB}, we conclude that
\[c_d(\mathsf{ABB},\Gamma)\leq 3^d\cdot \frac32\hat f(0)(1-\hat f(0))^2\leq 3^d\max_{0\leq\alpha\leq 1}\frac 32\alpha(1-\alpha)^2=\frac 293^d=2\cdot 3^{d-2}.\qedhere\]
\end{proof}

Through the remainder of this section we write $\omega:=e(1/3)=e^{2\pi i/3}$. To prove \cref{lem:n3ABB}, we need the following simple inequality.

\begin{lemma}\label{lem:real-part-cube} Let $\beta>0$, and $z$ be a complex number satisfying
\[\Re(z),\Re(\omega z),\Re(\omega^2z)\leq \beta.\]
Then $\Re(z^3)\leq \beta^3$.
\end{lemma}

We will apply this to $z=\hat f(\xi)$ for various $\xi$. 

\begin{proof} For $j\in\{0,1,2\}$, let $\gamma_j=\Re(\omega^jz)$. We have $\gamma_0+\gamma_1+\gamma_2=0$ and
\[\gamma_0\gamma_1\gamma_2=\frac 18\left(z+\overline z\right)\left(\omega z+\omega^2\overline z\right)\left(\omega^2z+\omega \overline z\right)=\frac 18\left(z^3+\overline z^3\right)=\frac 14\Re(z^3).\]
If all $\gamma_j$ are nonpositive, then $\Re(z^3)\leq 0$ and we are done. Otherwise, without loss of generality, let $\gamma_0>0$; then
\[\Re(z^3)=4\gamma_0\gamma_1\gamma_2=\gamma_0(4\gamma_1\gamma_2)\leq \gamma_0(\gamma_1+\gamma_2)^2=\gamma_0^3\leq \beta^3.\qedhere\]
\end{proof}

\begin{proof}[Proof of \cref{lem:n3ABB}] We begin by writing the left side of \eqref{eq:ABB3d-bound} in terms of Fourier coefficients. Let $g\colon(\ZZ/3\ZZ)^d\to[0,1]$ be such that $g(x)=1-f(x)$ for all $x$. We compute using Fourier inversion that
\begin{align*}
\EE_{x,y}f(x)g(x+y)g(x+2y)
&=\EE_{x,y}\sum_{\xi_0,\xi_1,\xi_2}\hat f(\xi_0)\hat g(\xi_1)\hat g(\xi_2)\omega^{\xi_0\cdot x+\xi_1\cdot (x+y)+\xi_2\cdot(x+2y)}\\
&=\sum_{\xi_0,\xi_1,\xi_2}\hat f(\xi_0)\hat g(\xi_1)\hat g(\xi_2)\left(\EE_x\omega^{(\xi_0+\xi_1+\xi_2)\cdot x}\right)\left(\EE_y\omega^{(\xi_1+2\xi_2)\cdot y}\right)\\
&=\sum_{\substack{\xi_0,\xi_1,\xi_2\\\xi_0+\xi_1+\xi_2=0\\\xi_1+2\xi_2=0}}\hat f(\xi_0)\hat g(\xi_1)\hat g(\xi_2)=\sum_\xi \hat f(\xi)\hat g(\xi)^2.
\end{align*}
Since $\hat g(0)=1-\alpha$, and $\hat g(\xi)=-\hat f(\xi)$ when $\xi\neq 0$, we conclude that
\begin{equation}\label{eq:fourier-exp}
\EE_{x,y}f(x)\big(1-f(x+y)\big)\big(1-f(x+2y)\big)=\alpha(1-\alpha)^2+\sum_{\xi\neq 0}\hat f(\xi)^3.
\end{equation}
We claim that, for each $\xi\neq 0$,
\begin{equation}\label{eq:fourier-cube}
\Re\left(\hat f(\xi)^3\right)\leq\min\left(|\hat f(\xi)|^3,\frac18(1-\alpha)^3\right).
\end{equation}
Indeed, we compute for each $j\in\ZZ/3\ZZ$ that
\begin{align*}
\EE_x\big[f(x)\mid\xi\cdot x=j\big]
&=3\EE_x\big[f(x)\cdot(1\text{ if }\xi\cdot x=j,\ 0\text{ otherwise})\big]\\
&=\EE_x\left[f(x)\left(1+\omega^{\xi\cdot x-j}+\omega^{-\xi\cdot x+j}\right)\right]\\
&=\alpha+\omega^j\hat f(\xi)+\omega^{-j}\overline{\hat f(\xi)}=\alpha+2\Re\left(\omega^jz\right).
\end{align*}
The fact that $f$ has image in $[0,1]$ implies that this expectation is at most $1$, and so $\Re(\omega^jz)\leq \frac{1-\alpha}2$. Therefore, \cref{lem:real-part-cube}, combined with the simple bound $\Re(\hat f(\xi)^3)\leq|\hat f(\xi)^3|=|\hat f(\xi)|^3$, gives \eqref{eq:fourier-cube}. We thus have, using \eqref{eq:fourier-exp} and Parseval's identity, that
\begin{align*}
\EE_{x,y}f(x)\big(1-f(x+y)\big)\big(1-f(x+2y)\big)
&=\alpha(1-\alpha)^2+\sum_{\xi\neq 0}\Re\left(\hat f(\xi)^3\right)\\
&\leq\alpha(1-\alpha)^2+\sum_{\xi\neq 0}\min\left(|\hat f(\xi)|^3,\left(\frac{1-\alpha}2\right)^3\right)\\
&\leq\alpha(1-\alpha)^2+\sum_{\xi\neq 0}\left(\frac{1-\alpha}2\right)|\hat f(\xi)|^2\\
&=\alpha(1-\alpha)^2+\frac{1-\alpha}2\left(\alpha-\alpha^2\right)=\frac 32\alpha(1-\alpha)^2.\qedhere
\end{align*}
\end{proof}

\subsection{Slantability}\label{subsec:fourier-slant}

We now show \cref{thm:no-slant}, i.e.~that a word of length $\ell$ cannot be $d$-slantable for any $d$ much larger than $\log\ell$. The structure of the proof is as follows. Assume that a word $w$ of length $\ell$ is $d$-slantable.
\begin{enumerate}
    \item Find a large grid $\Theta_0$ in which $w$ has near-optimal concentration.
    
    \item Fix some $n=\ell^{O(1)}$. Take a random ``subgrid'' $\Theta$ with shape $(\ZZ/n\ZZ)^d$ from $\Theta_0$. Show that, with positive probability, every search line $\Gamma$ of $\Theta$ has nearly the maximal concentration of $w$.

    \item Apply our one-dimensional stability result (\cref{prop:1d-stable}) to conclude that each of these search lines has nearly the same letter distribution.

    \item Using Fourier analysis, show that it is impossible for every search line in a grid of shape $(\ZZ/n\ZZ)^d$ to have nearly the same letter distribution.
\end{enumerate}

We begin by stating the two main lemmas, corresponding to steps (2) and (4) in the above outline.

\begin{lemma}\label{lem:slantable} Suppose that a word $w$ is $d$-slantable for some $d$, and let $\ell=\len(w)$. Then, for each positive integer $n$, there exists a grid $\Theta$ of shape $(\ZZ/n\ZZ)^d$ in which every search line contains at least $C_1(w)(n-4d\ell)$ copies of $w$.
\end{lemma}

\begin{lemma}\label{lem:density-along-lines} Let $A$ be a letter, and let $\Theta$ be a grid of shape $(\ZZ/n\ZZ)^d$ which contains $\alpha n^d$ copies of $A$, where $\alpha\in[0,1]$. Suppose $n<2^d$. Then there exist two search lines of $\Theta$ in which the number of copies of $A$ differ by at least
\[\frac{\sqrt{\min(\alpha,1-\alpha)}}{3^{d/2}}n.\]
\end{lemma}

The proof of \cref{lem:slantable} is a conceptually simple (although notationally complex) application of the probabilistic method, while the proof of \cref{lem:density-along-lines}, which is the crux of this argument, requires Fourier analysis. 

\begin{proof}[Proof of \cref{thm:no-slant} given \cref{lem:slantable} and \cref{lem:density-along-lines}] Suppose $w$ is $d$-stackable with $d>10$, and let $\ell=\len(w)$. Let $n=2^d-1$, and let $\Theta$ be a grid of shape $(\ZZ/n\ZZ)^d$ given by \cref{lem:slantable}, so that every search line has at least $(n-4d\ell)\cdot C_1(w)$ copies of $w$. Let $\delta=4d\ell/n$, and let $h_w$ be the letter distribution of a one-dimensional extremal configuration for $w$, unique by \cref{lem:same-dist}. By \cref{prop:1d-stable}, for each search line $\Gamma$ in $\Theta$, the letter distribution $h_\Gamma$ of $\Gamma$ satisfies
\begin{equation}\label{eq:distribution-difference}
\TV{h_\Gamma-h_w}\leq\delta.
\end{equation}
Let $A$ be the letter in $w$ which appears the fewest times. Then $1/\ell\leq h_w(A)\leq 1/2$. Let $\alpha$ be the proportion of letters in $\Theta$ equal to $A$. Since \eqref{eq:distribution-difference} holds for $\Theta$ as well as individual search lines (by averaging $h_w$), we have
\[\frac1\ell-\delta\leq \alpha\leq \frac12+\delta\leq1-\frac1\ell+\delta.\]
By \cref{lem:density-along-lines}, there exist two search lines of $\Theta$ in which the number of copies of $A$ differ by at least
\[\frac{\sqrt{\min(\alpha,1-\alpha)}}{3^{d/2}}n\geq \frac{\sqrt{\frac1\ell-\delta}}{3^{d/2}}n.\]
On the other hand, by applying the triangle inequality to \eqref{eq:distribution-difference}, the number of copies of $A$ in any two search lines cannot differ by more than $2\delta n$. We conclude that
\[2\delta\geq \frac1{3^{d/2}}\sqrt{\frac1\ell-\delta}.\]
Rearranging we obtain $4\delta^23^d+\delta\geq1/\ell$. If $\delta\geq\frac1{2\ell}$, then $8d\ell^2\geq n=2^d-1$, which implies $d\leq 3\log_2\ell+10$. If, on the other hand, $4\delta^23^d\geq\frac1{2\ell}$, then
\[128d^2\ell^33^d\geq n^2>2^{2d-1}\implies 256d^2\ell^3>\left(\frac43\right)^d.\]
This gives $d<8\log_2\ell+47$, as desired.
\end{proof}

We now prove the two lemmas. To prove \cref{lem:slantable}, we will need the following intermediate result about subgrids of near-extremal one-dimensional grids.

\begin{lemma}\label{lem:1d-subgrid} Let $n$ and $N$ be positive integers, let $\Gamma$ be a grid of shape $\ZZ/N\ZZ$, and let $w$ be a word of length $\ell$. Suppose that $c_1(w,\Gamma)=C_1(w)-\delta$ for some $\delta\in[0,1]$. Let $t$ be chosen uniformly at random from $\ZZ/N\ZZ$, and let
\[\Gamma'_t:=\Grid\big(\Gamma(t)\Gamma(t+1)\cdots\Gamma(t+n-1)\big).\]
Then, with probability at least $1-\ell n\delta$, $c_1(w,\Gamma'_t)>\left(1-\frac{4\ell}n\right)C_1(w)$.
\end{lemma}
\begin{proof} If $\delta\leq\frac1{\ell n}$, the result is trivial, so we henceforth assume $\delta<\frac1{\ell n}$. Let
\[T_0=\big\{t : \ct(w,\Gamma'_t)<C_1(w)(n-4\ell)\big\}.\]
There is a subset $T_1\subset T_0$ of size at least $|T_0|/n$ with ``gaps'' of length at least $n$, i.e.~for which $\{T_1+k : 0\leq k<n\}$ are all disjoint. Let $U_1=T_1\sqcup(T_1+1)\sqcup\cdots\sqcup (T_1+n-1)$, and let $U_2=(\ZZ/N\ZZ)\setminus U_1$. We now count the copies of $w$ in $\Gamma$, considered via their locations as subsets of $\ZZ/N\ZZ$ of size $\ell$:
\begin{itemize}
    \item Some copies are contained entirely within $U_1$. These copies are in $\Gamma'_t$ for some $t\in T_1$, and thus there are at most
    \begin{equation}\label{eq:U1-count}
    \sum_{t\in T_1}\ct(w,\Gamma_t')\leq |T_1|\big(C_1(w)(n-4\ell)\big)
    \end{equation}
    such copies.

    \item Some copies are contained entirely within $U_2$. As $U_2$ consists of only $N-n|T_1|$ letters, there are at most
    \begin{equation}\label{eq:U2-count}
    C_1(w)\big(N-n|T_1|\big)
    \end{equation}
    such copies of $w$, by the definition of $C_1(w)$. (We, rather crudely, upper bound the number of copies of $w$ contained entirely in $U_2$ by the count of $w$ in the grid of shape $\ZZ/|U_2|\ZZ$ filled by the letters of $U_2$ in order.)

    \item Lastly, some copies are neither contained fully in $U_1$ nor $U_2$. Each such copy must contain some portion of the ``boundary'' between $U_1$ and $U_2$; that is, there exists some $t\in T_1$ so that the copy either contains $\{t,t+1\}$ or $\{t+n-1,t+n\}$. So, each copy of $w$ neither in $U_1$ nor $U_2$ is contained within
    \[\bigcup_{t\in T_1}\{t-\ell+2,\ldots,t+\ell-1\}\cup\bigcup_{t\in T_1}\{t+n-\ell+1,\ldots,t+n+\ell-2\}.\]
    This consists of at most $4(\ell-1)|T_1|$ letters, and so there are at most
    \begin{equation}\label{eq:U12-count}
    C_1(w)\big(4(\ell-1)|T_1|\big)
    \end{equation}
    copies of $w$.
\end{itemize}

Summing \eqref{eq:U1-count}, \eqref{eq:U2-count}, and \eqref{eq:U12-count} gives that
\[C_1(w)N-N\delta=\ct(w,\Gamma)\leq C_1(w)N-4C_1(w)|T_1|.\]
We conclude that $|T_1|\leq \frac{N\delta}{4C_1(w)}\leq \ell N\delta$, so $|T_0|\leq \ell n N\delta$. This means that the probability that $c_1(w,\Gamma_t')>C_1(w)(n-4\ell)$ is at most $\ell n\delta$, as desired.    
\end{proof}

\begin{proof}[Proof of \cref{lem:slantable}] Let $C=C_1(w)$. Since $w$ is $d$-slantable, for every $\eps>0$ there exists some positive integer $N$ and a grid $\Theta_0$ of shape $(\ZZ/N\ZZ)^d$ for which
\[\ct(w,\Theta_0)\geq (C-\eps)\frac{3^d-1}2N^d.\]
We select $\eps=(\ell n^{d+1}3^d)^{-1}$.

Choose $(t_1,\ldots,t_d)\in(\ZZ/N\ZZ)^d$ uniformly at random, and let $\Theta$ be the grid of shape $(\ZZ/n\ZZ)^d$ given by
\[\Theta\big((u_1,\ldots,u_d)\big)=\Theta_0\big((u_1+t_1,\ldots,u_d+t_d)\big)\]
for $0\leq u_1,\ldots,u_d<n$. For $p\in(\ZZ/n\ZZ)^d$ and $\bv\in\{-1,0,1\}^d\setminus\{\zero\}$, let $\Gamma_{p,\bv}$ be the search line of $\Theta$ through $p$ in the direction of $\bv$. The letters in $\Gamma_{p,\bv}$ consist of at most $d$ strings of letters along search lines in the direction of $\bv$ in $\Theta_0$. The lengths $n_1,\ldots,n_k$ of these strings depend only on $(p,\bv)$ and sum to $n$, and the starting points of these strings are each uniformly distributed in $(\ZZ/N\ZZ)^d$ (each is a translate of $(t_1,\ldots,t_d)$ by a fixed vector). If one of our strings, say, one of length $n_j$, is conditioned to lie in a particular search line $\Gamma_0$ with $c_1(w,\Gamma_0)=C_1(w)-\delta$, then \cref{lem:1d-subgrid} gives that the concentration of $w$ within this string is at least $\left(1-\frac{4\ell}{n_j}\right)C_1(w)$ with probability at least $1-\ell n_j\delta$. Let $\eps_\bv$ be so that
\[\EE_{\Gamma_0}c_1(w,\Gamma_0)=C_1(w)-\eps_\bv,\]
where $\Gamma_0$ is chosen uniformly at random among the search lines of $\Theta_0$ in the direction of $\bv$. By taking the expectation over $\Gamma_0$, the count of $w$ within this string in $\Gamma_{p,\bv}$ is at least $(n_j-4\ell)C_1(w)$ with probability at least $1-\ell n_j\eps_\bv$. Taking a union bound over all of the strings in $\Gamma_{p,\bv}$, we have
\[\Pr\left[c_1(w,\Gamma_{p,\bv})<\left(1-\frac{4d\ell}n\right)C_1(w)\right]\leq \ell n\eps_\bv.\]
Union-bounding over all pairs $(p,\bv)$, the probability that some search line has concentration of $w$ at most $\left(1-\frac{4d\ell}n\right)C_1(w)$ is at most
\[\ell n^{d+1}\sum_{\bv}\eps_{\bv}=\ell n^{d+1}(3^d-1)\eps<1.\]
So, there is a positive probability that every search line in $\Theta$ has at least $C_1(w)(n-4d\ell)$ copies of $w$, as desired.
\end{proof}

Finally, we prove \cref{lem:density-along-lines}. 

\begin{proof}[Proof of \cref{lem:density-along-lines}] The proof is essentially a second moment argument, using Fourier analysis. Let $G=(\ZZ/n\ZZ)^d$. Let $\mathcal L$ be a set of search lines in $\Theta$, each considered as a set of elements of $G$. Note that $|\mathcal L|=n^{d-1}\frac{3^d-1}2$. 

Let $f=1_{A,\Theta}-\alpha$ evaluate to $1-\alpha$ at appearances of $A$ and $-\alpha$ at appearances of other letters. By the definition of $\alpha$, $\hat f(0)=0$. Let $\zeta=e(1/n)=e^{2\pi i/n}$. For the search line $\Gamma\in\mathcal L$ in the direction $\bv$ containing $p$, we have
\begin{align*}
\sum_{x\in\Gamma}f(x)=\sum_{x\in\Gamma}\sum_\xi\hat f(\xi)\zeta^{\xi\cdot x}
&=\sum_\xi\hat f(\xi)\sum_{k=0}^{n-1}\zeta^{\xi\cdot(p+k\bv)}\\
&=\sum_\xi\hat f(\xi)\zeta^{\xi\cdot p}[n\text{ if }\xi\cdot\bv=0,\ 0\text{ otherwise}]=n\sum_{\xi : \xi\cdot\bv=0}\hat f(\xi)\zeta^{\xi\cdot p}.
\end{align*}
Fix $\bv\in\{-1,0,1\}^d\setminus\{\zero\}$. Let $\Gamma$ be a random search line of $\Theta$ in the direction $\bv$, and let $p$ be a random point in $\Gamma$. Using the above, we compute
\begin{align*}
\EE_{\Gamma\text{ in dir.\ }\bv}\left|\frac1n\sum_{x\in\Gamma}f(x)\right|^2
&=\EE_p\left|\sum_{\xi : \xi\cdot\bv=0}\hat f(\xi)\zeta^{\xi\cdot p}\right|^2\\
&=\EE_p\sum_{\substack{\xi_1,\xi_2\\\xi_1\cdot\bv=\xi_2\cdot\bv=0}}\hat f(\xi_1)\overline{\hat f(\xi_2)}\zeta^{\xi_1\cdot p}\zeta^{-\xi_2\cdot p}\\
&=\sum_{\substack{\xi_1,\xi_2\\\xi_1\cdot\bv=\xi_2\cdot\bv=0}}\hat f(\xi_1)\overline{\hat f(\xi_2)}\EE_p\zeta^{(\xi_1-\xi_2)\cdot p}=\sum_{\xi : \xi\cdot\bv=0}|\hat f(\xi)|^2.
\end{align*}
Now, consider $\Gamma$ sampled from the set $\mathcal L$ of \emph{all} search lines of $\Theta$, uniformly at random. We have
\begin{align*}
\EE_{\Gamma\in\mathcal L}\left|\frac1n\sum_{x\in\Gamma}f(x)\right|^2
&=\EE_{\bv\in\{-1,0,1\}^d\setminus\{\zero\}}\EE_{\Gamma\text{ in dir. }\bv}\left|\frac1n\sum_{x\in\Gamma}f(x)\right|^2\\
&=\EE_\bv\sum_{\xi : \xi\cdot\bv}|\hat f(\xi)|^2\\
&=\sum_{\xi}|\hat f(\xi)|^2\Pr[\xi\cdot\bv=0].
\end{align*}
Take any $\xi\in(\ZZ/n\ZZ)^d$. The dot products $\xi\cdot\bu$ for $\bu\in\{0,1\}^d$ cannot all be distinct, since $n<2^d$. So, there are two distinct vectors $\bu,\bu'\in\{0,1\}^d$ with $\xi\cdot(\bu-\bu')=0$. The vectors $\bv=\pm(\bu-\bu')$ lie in $\{-1,0,1\}^d\setminus\{\zero\}$, and so $\Pr[\xi\cdot\bv=0]\geq \frac 2{3^d-1}$. We conclude that
\[\EE_{\Gamma\in\mathcal L}\left|\frac1n\sum_{x\in\Gamma}f(x)\right|^2\geq\frac 2{3^d-1}\sum_{\xi}|\hat f(\xi)|^2=\frac 2{3^d-1}\EE_x|f(x)|^2=\frac{2\alpha(1-\alpha)}{3^d-1}.\]
In particular, there exists a search line $\Gamma$ of $\Theta$ with
\[\left|\sum_{x\in\Gamma}f(x)\right|\geq\frac{\sqrt{2\alpha(1-\alpha)}}{(3^d-1)^{1/2}}n>\frac{\sqrt{\min(\alpha,1-\alpha)}}{3^{d/2}}n.\]
By the definition of $f$, we have
\[\sum_{x\in\Gamma}f(x)=(\#\text{ of copies of }A\text{ in }\Gamma)-\alpha n.\]
The result follows from the fact that there exist search lines with at most $\alpha n$ copies of $A$ and search lines with at least $\alpha n$ copies of $A$.
\end{proof}

\begin{remark} In many cases, a positive answer to \cref{qn:extremal-exist} would allow us to circumvent \cref{lem:slantable}. Indeed, if $w$ is $d$-slantable and a $w$-extremal grid $\Theta$ exists, then every search line of the grid is a one-dimensional $w$-extremal grid. If $w$ falls into case (a) or (b) of \cref{prop:1d-structure}, then every one-dimensional $w$-extremal grid is periodic with some fixed period $n<2\ell$, and so $\Theta$ is equivalent to a grid of shape $(\ZZ/n\ZZ)^d$. From here \cref{lem:density-along-lines} allows us to directly conclude that $2^d<n$, i.e.~$d<\log_2\ell+1$.

However, if $w$ falls into case (c), the above argument fails, since $\Theta$ need not be equivalent to a grid of shape $(\ZZ/n\ZZ)^d$ for some $n$ whose size is controlled by $\ell$. Since the argument of \cref{lem:density-along-lines} relies crucially on the bound $n<2^d$, it is still necessary to reduce $\Theta$ to a smaller grid, which may not be exactly $w$-extremal. We do not know how to substantially improve the bound in \cref{thm:no-slant} under an assumption of \cref{qn:extremal-exist} in this case.
\end{remark}
\section{\texorpdfstring{$d$}{d}-slantable words and connections with modular nonattacking queens}\label{sec:not-stack}

In this section, we explore words which achieve the upper bound of \cref{prop:simple-d-bound}. Some of these configurations are closely related with the well-studied problem of placing as many queens as possible on a chessboard of a certain shape such that no two of the queens attack each other. 

The $n$-queens problem is one of the most famous problems in combinatorics. It was first posed in 1848, in the context of placing 8 queens on a standard 8-by-8 chessboard such that no two attack each other, but has since spawned numerous generalizations, most immediately to the problem of placing $n$ queens on an $n$-by-$n$ chessboard such that no two attack each other. (For brevity, we simply say that queens in such a configuration are \textit{nonattacking}.) Bell and Stevens have surveyed the $n$-queens problem and its many variations in great detail \cite{bellstevens}. 

It is natural that there might be connections between our problem and $n$-queens, as queens attack in the directions of our search lines. Specifically, the problem of what Bell and Stevens call ``modular $n$-queens," in which the attack lines for the queens wrap around the edges of the chessboard, is the relevant one for us, as we take search lines to wrap around the edges of the chessboard. Classical results about modular $n$-queens quickly give us important results about slantability.

\begin{proposition}\label{prop:ABn 2 slantable}
    Let $w= \mathsf{AB}^{\ell-1}$ for some $n>1$. Then $w$ is 2-slantable, i.e., $C_2(w) = 4 C_1(w)$, if $\gcd(\ell, 6) = 1$.
\end{proposition}

\begin{proof}
    P\'olya famously proved in \cite{polya} that there exists a configuration of $n$ queens on a chessboard of shape $\Zmod{n}\times \Zmod{n}$ such that none of the queens attack each other for $n$ with $\gcd(n, 6)=1$. Consider a grid $\Gamma$ of shape $\Zmod{\ell}\times \Zmod{\ell}$, such that there are $\ell$ $\A$'s, at the positions of $\ell$ nonattacking queens, and the rest of the squares are labeled with $\B$. The condition that the positions at which $\A$ appears are the positions of $\ell$ nonattacking queens means that $\ell-1$ $\B$'s appear after each $\A$ in every direction in which words can read. Therefore $c_2(\mathsf{AB}^{\ell-1}, \Gamma) = 8/\ell = 4 C_1(\mathsf{AB}^{\ell-1})$, as desired. For an illustration, see \cref{fig:nonattacking queens,fig:nonattacking queens grid}.
\end{proof}

\begin{figure}
    \begin{tikzpicture}[scale=0.5]
  \foreach \x in {0,1,2,3,4,5,6} {
    \foreach \y in {0,1,2,3,4,5,6} {
      \ifodd\numexpr\x+\y\relax
        \fill[gray!30] (\x,\y) rectangle ++(1,1);
      \else
        \fill[white] (\x,\y) rectangle ++(1,1);
      \fi
    }
  }
  \draw[step=1cm, black] (0,0) grid (7,7);
  \foreach \x in {0,1,2,3,4,5,6} {
    \node at (\x+0.5,-0.5) {\scriptsize \x};
    \node at (-0.5,\x+0.5) {\scriptsize \x};
  }

  \node at (0.5,6.5) {Q};
  \node at (2.5,5.5) {Q};
  \node at (4.5,4.5) {Q};
  \node at (6.5,3.5) {Q};
  \node at (1.5, 2.5) {Q};
  \node at (3.5, 1.5) {Q};
  \node at (5.5, 0.5) {Q};
\end{tikzpicture} 
\caption{7 nonattacking queens on a modular 7x7 chessboard.}
\label{fig:nonattacking queens}\end{figure}
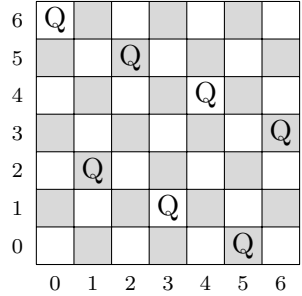

\begin{figure}
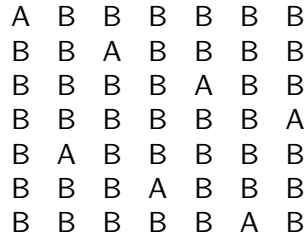
$\begin{array}{ccccccc}
\A & \B & \B & \B & \B & \B & \B \\
\B & \B & \A & \B & \B & \B & \B \\ 
\B & \B & \B & \B & \A & \B & \B \\ 
\B & \B & \B & \B & \B & \B & \A \\ 
\B & \A & \B & \B & \B & \B & \B \\ 
\B & \B & \B & \A & \B & \B & \B \\ 
\B & \B & \B & \B & \B & \A & \B \\ 

\end{array}$
\caption{A grid showing $\mathsf{ABBBBBB}$ is 2-slantable. The positions of the $\mathsf{A}$'s correspond to the nonattacking queens in \cref{fig:nonattacking queens}.}
\label{fig:nonattacking queens grid}
\end{figure}

Relating our problem to $n$-queens in higher dimensions requires choosing a convention for how which squares of a higher-dimensional chessboard a queen may attack. We allow queens to attack along search lines; that is, a queen in position $p$ can attack square $p'$ if $p'-p$ is a scalar multiple of some element of $\{-1,0,1\}^d\setminus\{\zero\}$. This is discussed in \cite[Section~6]{bellstevens}. (Other conventions are available, most notably that of \cite{nudelman}, in which queens attack in hyperplanes orthogonal to search lines rather than in search lines themselves.) 

P\'olya's construction generalizes naturally to $d$ dimensions with this convention. This generalization allows us to prove \cref{prop:queen-words}, which states that the word $w=\mathsf{AB}^{\ell-1}$ for some $\ell>1$ is $d$-slantable if $\gcd(\ell,(2^d)!) = 1$.

\begin{proof}[Proof of \cref{prop:queen-words}]
    Consider the grid $\Gamma$ of shape $(\Zmod{\ell})^d$ such that the letter $\A$ appears at the positions of the form
    \[\left(\sum_{i=1}^{d-1} 2^{i}\alpha_i, \alpha_1, \ldots, \alpha_{d-1}\right)\]
    for $0\le \alpha_i < \ell$, and the letter $\B$ appears at all other positions. Note in particular that all $\ell^{d-1}$ of these positions are distinct.
    
    Beginning with each $\A$, there is an appearance of $w$ reading in each of the $3^d-1$ possible directions in which words can appear. To see this, it suffices to show that, for any distinct positions $p_1, p_2$ at which the letter $\A$ appears, the difference $p_2-p_1$ is not proportional to $\bv$ for any $\bv \in \{-1, 0, 1\}^d \setminus\{\zero\}$.

    Write $p_1 = (\sum_{i=1}^{d-1} 2^{i}\alpha_i, \alpha_1, \ldots, \alpha_{d-1})$ and $p_2 = (\sum_{i=1}^{d-1} 2^{i}\beta_i, \beta_1, \ldots, \beta_{d-1})$, and suppose that $p_2-p_1=c\bv$ for some $\bv\in \{-1, 0, 1\}^d \setminus\{\zero\}$. We will show that $p_1=p_2$. Indeed, we have $\beta_i - \alpha_i = \eps_i c$ for all $i$, where $\eps_i\in\{1,0,-1\}$. Additionally,
    $\sum_{i=1}^{d-1} 2^i(\beta_i - \alpha_i) = \eps_0 c$, so
    \[c\left(-\eps_0 + \sum_{i=1}^{d-1} 2^i \eps_i\right)\equiv 0\pmod \ell.\]
    However, $|(-\eps_0 + \sum_{i=1}^{d-1} 2^i \eps_i)|\leq 2^d-1$ is a factor of $(2^d)!$. Since $\gcd(\ell,(2^d)!)=1$, this implies $c\equiv 0\pmod\ell$, and thus $p_1=p_2$.

    Therefore there are a total of $\ell^{d-1}(3^d-1)$ appearances of $w$ in $\Gamma$, which is of size $\ell^d$. We conclude $c_d(w,\Gamma)=\frac{3^d-1}{\ell}$. Since $C_1(w)=\frac2{\ell}$, we conclude that $w$ is $d$-slantable.
\end{proof}

Explicating the connection between $\A\B^{\ell-1}$ and the $\ell$-queens problem, we have the following immediate corollary from the proof of \cref{prop:queen-words}.

\begin{corollary}\label{cor:nd-1queens}
    For all $n$ and all $d$, if $\gcd(n,(2^d)!)=1$, then one can place $n^{d-1}$ nonattacking queens on the grid of shape $(\Zmod{n})^d$.
\end{corollary}

\begin{proof}
    Place the queens at the positions where the letter $\A$ appears in the above proof. Then, the condition that the difference between any two of these positions is not proportional to $\bv$ for any $\bv\in \{-1, 0, 1\}^d \setminus\{\zero\}$ is exactly the conditions that the queens do not attack each other. The result follows.
\end{proof}

In fact, \cref{cor:nd-1queens} gives one direction of \cite[Conjecture 25]{bellstevens}. As far as we know, this proof has not yet appeared in the literature, though it is likely already known: it is a straightforward generalization of the argument for $d=3$ which appears in \cite{klarner}. For more on the $d=3$ case of this argument, see \cite[Theorem~20]{bellstevens} or \cite[Theorem~5]{kunt}.

We again invoke the $n$-queens problem to prove \cref{prop:ABell-1}. 

\begin{proof}[Proof of \cref{prop:ABell-1}]
    Monsky showed in \cite{monsky} that, for any $\ell$, there exists a configuration of $\ell-2$ nonattacking queens on a modular $\ell$-by-$\ell$ board. Consider the grid $\Gamma$ of shape $(\Zmod{\ell})^2$ with $\A$'s at the locations of the nonattacking queens and $\B$'s on the other tiles. Then, the word $w$ appears $8(\ell-2)$ times in the grid $\Gamma$, so $c_2(w, \Gamma) = 8(\ell-2)/\ell^2$. We have $3C_1(w) = 6/\ell,$ so if $8(\ell-2)/\ell^2 > 6/\ell$, we have immediately that $w$ is not $2$-stackable. Therefore $w$ is not $2$-stackable for $\ell>8$.

    For $\ell=4$, the left grid in \cref{fig:ABBB} is a grid $\Gamma$ with $c_2(w, \Gamma) > 3C_1(w)$, showing that $w$ is not $2$-stackable. For $\ell=5$ and $\ell=7$, \cref{prop:ABn 2 slantable} shows $C_2(w) = 4C_1(w) >3C_1(w),$ so $w$ is not $2$-stackable. Finally, \cref{fig:ABn larger n} shows example grids that demonstrate that $w$ is not $2$-stackable for $\ell=6$ and $\ell=8$.
\begin{figure}
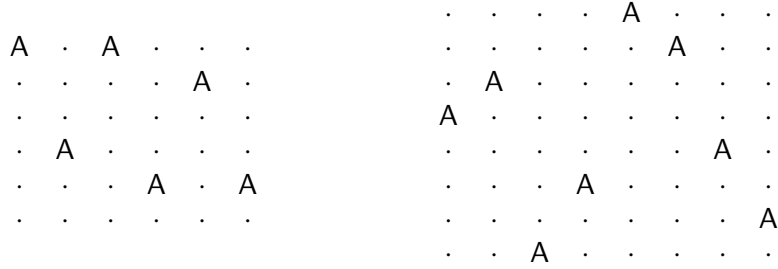

    $\begin{array}{cccccc}
        \A& \cdot & \A & \cdot & \cdot & \cdot \\ 
        \cdot & \cdot & \cdot & \cdot & \A & \cdot \\
        \cdot & \cdot & \cdot & \cdot & \cdot & \cdot \\ 
        \cdot & \A & \cdot & \cdot & \cdot & \cdot \\ 
        \cdot & \cdot & \cdot & \A & \cdot & \A \\ 
        \cdot & \cdot & \cdot & \cdot & \cdot & \cdot \\
    \end{array}$ \hspace{1.8cm} $\begin{array}{cccccccc} 
\cdot & \cdot & \cdot & \cdot & \A & \cdot & \cdot & \cdot \\ 
\cdot & \cdot & \cdot & \cdot & \cdot & \A & \cdot & \cdot \\ 
\cdot & \A & \cdot & \cdot & \cdot & \cdot & \cdot & \cdot \\ 
\A & \cdot & \cdot & \cdot & \cdot & \cdot  & \cdot & \cdot\\ 
\cdot &\cdot & \cdot & \cdot & \cdot & \cdot & \A & \cdot \\ 
\cdot & \cdot & \cdot & \A & \cdot & \cdot & \cdot & \cdot\\ 
\cdot & \cdot & \cdot & \cdot & \cdot & \cdot & \cdot & \A\\ 
\cdot & \cdot & \A & \cdot & \cdot & \cdot & \cdot & \cdot\\  \end{array}$

    \caption{On the left, a grid $\Gamma$ with $c_2(\mathsf{ABBBBB}, \Gamma) = 40/36 > 3 \cdot 2/6 = 3C_1(\mathsf{ABBBBB})$. On the right, a grid $\Gamma$ with $c_2(\mathsf{ABBBBBBB}, \Gamma) = 56/64 > 3 \cdot 2/8 = 3C_1(\mathsf{ABBBBBBB})$. (For legibility, we use $\cdot$ to denote the positions of the $\B$'s.)}
    \label{fig:ABn larger n}
\end{figure}
\end{proof}

We make no claims that the constructions we used to prove \cref{prop:ABell-1} are extremal.

\section{Conclusion and directions for future work}\label{sec:conclusion}

While we have introduced some methods to answer a number of cases of \cref{qn:main}, many interesting questions still remain. We highlight some of them here.

We begin by repeating \cref{qn:extremal-exist}: does there exist, for each word $w$ and each dimension $d$, a $d$-dimensional $w$-extremal grid? We have answered this question in the affirmative for $d=1$, and for $d=2$ for all words of length at most $4$ besides $\mathsf{ABBB}$. Two particularly interesting ``small'' cases for which \cref{qn:extremal-exist} remains open are $C_2(\mathsf{ABBB})$ and $C_3(\mathsf{ABB})$. (\cref{qn:main} also remains open for these two cases.) As remarked upon in \cref{subsec:intro-additive}, this question has some similarities to the periodic tiling conjecture; it is not entirely unreasonable to suggest that its answer might be negative in high dimension.

Relatedly, our investigation of short words leaves some natural questions, many of which we have gestured at in the previous sections.

\begin{question} Which $5$-letter words are $2$-stackable? 
\end{question}

Partial progress towards answering this question is given in \cref{subsec:short-5}. This question can naturally be posed for words of length 6 or longer, of course, as well.

In higher dimensions, there are even shorter words which remain open.

\begin{question}\label{qn:ABB} Is $\mathsf{ABB}$ $d$-stackable for every $d$?
\end{question}

By \cref{lem:ABB}, this question has an affirmative answer for $d=2$, and by \cref{thm:ABB-Zmod3}, the answer is affirmative for every $d$ if restricted to grids of shape $(\Zmod{3})^d$. However, we do not even know if $\mathsf{ABB}$ is 3-stackable.

Additionally, there are still short words whose properties in two dimensions we do not fully understand. For instance, the word $\mathsf{ABBB}$ is neither $2$-stackable (see \cref{fig:ABBB}) nor $2$-slantable (see \cref{sec:appendix}), and is the only four-letter word, up to equivalence, which is not $2$-stackable. As such, we are particularly interested in the computation of $C_2(\mathsf{ABBB})$.

\begin{question}\label{qn:ABBB} Does $C_2(\mathsf{ABBB})=8/5$?
\end{question}

In general, the study of words of the form $\mathsf{AB}^{\ell-1}$ for various $\ell$ in various dimensions seems a particularly interesting special case.

Our results focus on small classes of words for which we can prove concentration bounds, but we have no results about long words in general, or asymptotic estimates as to ``typical'' behavior for long words. We ask some questions along these lines.

\begin{question} What proportion of long words (say, on the letters $\A$ and $\B$) are $d$-stackable? $2$-stackable? What proportion are $d$-slantable?
\end{question}

Finally, we proved a structural result and a stability result for grids in one dimension, but we have made no attempt to prove results of this kind for grids in two or more dimensions, which admit more exotic classes of extremal constructions. One example of such behavior was given in \cite{alonkravitz} for the word $\mathsf{CAT}$, which admits two inequivalent extremal constructions. The word $\mathsf{ABB}$ does as well: see \cref{fig:two ABBs}.

\begin{figure}
    $\begin{array}{ccc}
        \A & \B & \B \\
        \B & \A & \B \\
        \B & \B & \A \\
    \end{array}$ \hspace{1.8cm} 
    $\begin{array}{ccc}
    \A & \B & \B \\
    \A & \B & \B \\
    \A & \B & \B \\
    \end{array}
    $\caption{Two inequivalent $\mathsf{ABB}$-extremal grids in two dimensions.}
    \label{fig:two ABBs}
\end{figure}

There are plenty of questions about structure and stability in $d$ dimensions that one can ask. We record two here, which generalize \cref{lem:same-dist} and \cref{prop:1d-stable}, respectively.

\begin{question}\label{qn:same-dist-ndim}
    Given $w$ and $d$, do all $w$-extremal grids in $d$ dimensions have the same letter distribution?
\end{question}

\begin{question}\label{qn:stability-ndim}
    Given $w$ and $d$, does there exist a letter distribution $h_{w,d}$ such that if $\Gamma$ is a $d$-dimensional grid satisfying $c_d(w, \Gamma) \ge C_d(w)(1-\delta)$, then $\TV{h_\Gamma-h_{w,d}} \le \delta$?
\end{question}

We note that $h_{w,d}$, if it exists, cannot generally be independent of the dimension. To see this, consider $w=\mathsf{BABBB}$. By \cref{lem:BABBB}, if $h_{w,d}$ exists for $d=1$ and $d=2$, it must be that $h_{w,1}(\A)=1/4$ and $h_{w,2}(\A)=1/5$. We have phrased \cref{qn:stability-ndim} in essentially the strongest quantitative form possible; the answer to \cref{qn:stability-ndim} may well be negative, and weaker versions of the same property would also be interesting.

\section*{Acknowledgments}

This research began at the 2023 University of Minnesota Duluth REU, with support from Jane
Street Capital, the National Security Agency, and the National Science Foundation (Grants 2052036
and 2140043). Schildkraut is supported by the National Science Foundation Graduate Research Fellowship Program under Grant No.~DGE-2146755.

We are extremely grateful to Noah Kravitz, who suggested the problem and provided an immense amount of guidance. We also thank Sam Spiro for some helpful comments. This article additionally benefited from a number of useful conversations with participants and visitors at the 2023 Duluth REU, including, but not limited to, Evan Chen, Mitchell Lee, Katherine Tung, and Daniel Zhu. Finally, we thank Colin Defant and Joe Gallian for organizing the REU and providing many helpful suggestions.

\bibliographystyle{abbrv}
\bibliography{bib.bib}

\appendix

\section{Local analysis: casework and additional results}\label{sec:appendix}

We have deferred two computations and one example arising from our ``local analysis'' strategy to this appendix.

The first and most central of these is a by-hand verification of the computation of $C_2(\mathsf{ABB})$, encapsulated in \eqref{eq:ABB-local}. This is covered by the following lemma. We strongly encourage the reader who wishes to read the following proof to draw some of the relevant diagrams.

\begin{lemma}\label{lem:ABB-by-hand} Let $F\colon \ZZ^2\times(\{-1,0,1\}^2\setminus\{\zero\})\to\mathbb R$ be the function described by the diagram in \cref{fig:ABB-2d-wt}, reproduced as \cref{fig:ABB-2d-wt-appendix}. Then, for any grid $\mathcal G$, we have
\begin{equation}\label{eq:ABB-local-app}
\sum_{(p,\bv)\in \mathcal A(\mathsf{ABB},\mathcal G)}F(p,\bv)\leq 12.
\end{equation}
\end{lemma}

\begin{figure}
\begin{tikzpicture}[scale=1]
    \localDiagram{0/0/1/1/2/below right,
                  0/0/-1/1/2/below left,
                  0/0/-1/-1/2/above left,
                  0/0/1/-1/2/above right,
                  1/0/-1/0/2/below,
                  1/0/-1/1/1/above right,
                  1/0/-1/-1/1/below right,
                  -1/0/1/0/2/below,
                  -1/0/1/1/1/above left,
                  -1/0/1/-1/1/below left,
                  0/1/0/-1/2/right,
                  0/1/-1/-1/1/above left,
                  0/1/1/-1/1/above right,
                  0/-1/0/1/2/right,
                  0/-1/-1/1/1/below left,
                  0/-1/1/1/1/below right,
                  1/1/-1/0/2/above,
                  1/1/0/-1/2/right,
                  1/1/-1/-1/2/above left,
                  -1/1/1/0/2/above,
                  -1/1/0/-1/2/left,
                  -1/1/1/-1/2/above right,
                  -1/-1/1/0/2/below,
                  -1/-1/0/1/2/left,
                  -1/-1/1/1/2/below right,
                  1/-1/-1/0/2/below,
                  1/-1/0/1/2/right,
                  1/-1/-1/1/2/below left}
\end{tikzpicture}
\caption{The weight function used to upper-bound $C_2(\mathsf{ABB})$.}
\label{fig:ABB-2d-wt-appendix}
\end{figure}
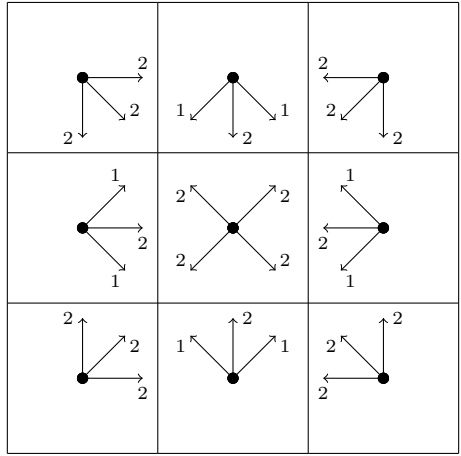

\begin{proof} Fix some grid $\mathcal G$. We must introduce some additional terminology and notation. The ``center square'' is the square in the center of \cref{fig:ABB-2d-wt-appendix}. The ``side squares'' are the squares which share a side of the center square, and the ``corner squares'' are the remaining four squares in the figure. Let $m_\A$ and $m_\B$ (resp.\ $n_\A$ and $n_\B$) be the number of $\A$'s and $\B$'s in the side (resp.\ corner) squares. Let $\ell$ be the letter in the center square. Write ($\star$) for the expression on the left side of \eqref{eq:ABB-local-app}. We observe that ($\star$) is even, since the contribution of the odd-labeled arrows is exactly the number of pairs of distinct letters obtained when reading the labels of the side squares around the center square. 

We can bound various terms in ($\star$) in the following way. (All terms will be accounted for in one of these three bullet points.)
\begin{enumerate}
    \item Consider the diagonal arrows labeled ``$2$'' in \cref{fig:ABB-2d-wt-appendix}. If $\ell=\A$ then the contribution of these arrows to the sum ($\star$) is exactly $2n_\B$. If $\ell=\B$ then the contribution of these arrows to ($\star$) is at most $2\min(n_\A,n_\B)$.

    \item Consider the arrows emanating from the squares adjacent to the center square of \cref{fig:ABB-2d-wt-appendix}. If $\ell=\A$ then the only arrows which may contribute to ($\star$) are those labeled ``$1$,'' and their total contribution is at most $2\min(m_\A,m_\B)$. If $\ell=\B$ then their contribution to ($\star$) is exactly $m_\A m_\B$ plus one for every occurrence of $\mathsf{ABB}$ passing horizontally or vertically through the center square. This count is at most $\min(m_\A,m_\B)$. The total contribution of these arrows is thus at most $m_\A m_\B+\min(m_\A,m_\B)$. 

    \item Consider the horizontal and vertical arrows emanating from the diagonally-adjacent-to-center squares. The total contribution of these arrows to ($\star$) is at most $2m_\B$, since each occurrence of $\mathsf{ABB}$ indicated by such an arrow uniquely identifies a $\B$ in a side-adjacent square. It is also at most $4\min(n_\A,n_\B)$, since each corner $\A$ or $\B$ identifies at most two such occurrences of $\mathsf{ABB}$.
\end{enumerate}

We can now verify \eqref{eq:ABB-local-app}. First, we treat the case where $\ell=\A$. Summing the contributions in ($\star$) as described in the bulleted list, we have
\begin{align}
\frac12(\star)
\notag&\leq n_\B+\min(m_\A,m_\B)+\min(2n_\A,2n_\B,m_\B)\\
\label{eq:A-center}&\leq n_\B+\frac{3m_\A+m_\B}4+\frac{2n_\A+m_\B}2\\
\notag&=\frac34(m_\A+m_\B)+n_\A+n_\B=7.
\end{align}
Since $\frac12(\star)$ is an integer, because, as we noted above, $(\star)$ is even, we either obtain \eqref{eq:ABB-local-app} or that equality holds everywhere. If equality holds in \eqref{eq:A-center}, then $2n_\A=m_\A=m_\B=2$. For equality to hold in (2) above, since $\ell=\A$, the two occurrences of $\B$ in side squares must be diametrically opposite the center square. However, for equality to hold in (3) above, the two occurrences of $\B$ must be adjacent to the unique $\A$ in a corner square. This means that equality cannot hold simultaneously in (2), (3), and \eqref{eq:A-center}. We conclude the result in the $\ell=\A$ case.

The case of $\ell=\B$ is slightly more unpleasant. Here we have
\[\frac12(\star)\leq\min(n_\A,n_\B)+\frac{m_\A m_\B+\min(m_\A,m_\B)}2+\min(2n_\A,2n_\B,m_\B).\]
If $m_\B=2$ then $m_\A m_\B+\min(m_\A,m_\B)=6$ and we have
\[\frac12(\star)\leq\min(n_\A,n_\B)+2\min(1,\min(n_\A,n_\B))+3\leq \min(n_\A,n_\B)+5\leq 7,\]
with equality if and only if equality holds in each of (1), (2), (3) and additionally $n_\A=n_\B=2$. If $m_\B\in\{0,4\}$ then $m_\A m_\B+\min(m_\A,m_\B)=0$ and we have $\frac12(\star)\leq 3\min(n_\A,n_\B)\leq 6$. If $m_\B\in\{1,3\}$ then $m_\A m_\B+\min(m_\A,m_\B)=4$ and we have
\[\frac12(\star)\leq\min(n_\A,n_\B)+2+\min(m_\B,2\min(n_\A,n_\B))\leq \min(n_\A,n_\B)+2+m_\B\leq 2+2+3=7,\]
with equality if and only if equality holds in each of (1), (2), (3) and additionally $n_\A=n_\B=2$ and $m_\B=3$. In either case, we obtain $n_\A=n_\B=2$. Since equality holds in (1), the two occurrences of $\B$ in corner squares must occur in the same row or column. This implies that at most two occurrences of $\mathsf{ABB}$ can be counted in (3). So (for equality to hold in (3)) we must have $m_\B=2$; moreover, the two occurrences of $\B$ in side squares must be diametrically opposite the center square. However, this means that equality cannot hold in (2), as the total contribution of those arrows to ($\star$) in this case is $4$ instead of $6$. This concludes the proof.    
\end{proof}

It is similarly possible to prove by hand that $\mathsf{ABCC}$ is $2$-stackable. This requires a very similar computation to that of \cref{lem:ABB-by-hand}. As before, the reader is encouraged to draw the relevant pictures.

\begin{lemma}\label{lem:ABCC-by-hand}
Let $F\colon \ZZ^2\times(\{-1,0,1\}^2\setminus\{\zero\})\to\mathbb R$ be the function described by the diagram in \cref{fig:ABCC-2d-wt}, reproduced as \cref{fig:ABCC-2d-wt-appendix}. Then, for any grid $\mathcal G$, we have
\begin{equation}\label{eq:ABCC-local-app}
\sum_{(p,\bv)\in \mathcal A(\mathsf{ABCC},\mathcal G)}F(p,\bv)\leq 12.\end{equation}
\end{lemma}

\begin{figure}
\begin{tikzpicture}[scale=0.6]
    \localDiagram{-2/-2/1/1/6/below right,
                  2/-2/-1/1/6/below left,
                  2/2/-1/-1/6/above left,
                  -2/2/1/-1/6/above right,
                  -1/-1/1/0/3/below,
                  -1/-1/0/1/3/left,
                  1/-1/-1/0/3/below,
                  1/-1/0/1/3/right,
                  1/1/-1/0/3/above,
                  1/1/0/-1/3/right,
                  -1/1/1/0/3/above,
                  -1/1/0/-1/3/left,
                  -1/-1/1/1/1/below right,
                  1/-1/-1/1/1/below left,
                  1/1/-1/-1/1/above left,
                  -1/1/1/-1/1/above right,
                  -1/0/1/0/4/above,
                  1/0/-1/0/4/above,
                  0/-1/0/1/4/left,
                  0/1/0/-1/4/left,
                  0/0/1/1/3/below right,
                  0/0/-1/1/3/below left,
                  0/0/-1/-1/3/above left,
                  0/0/1/-1/3/above right}
\end{tikzpicture}
\caption{The weight function used to upper-bound $C_2(\mathsf{ABCC})$.}
\label{fig:ABCC-2d-wt-appendix}
\end{figure}
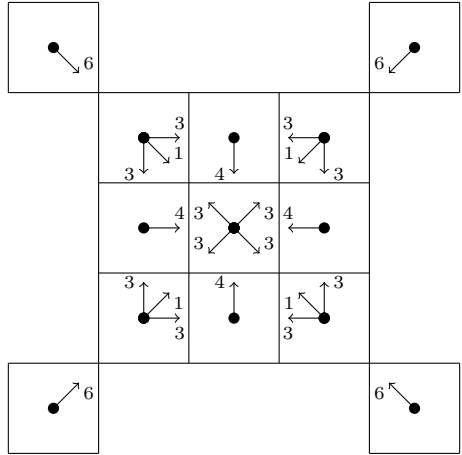

\begin{proof} Fix some grid $\mathcal{G}$. The ``outer corner squares" are the corners of the larger 5-by-5 square. The ``inner corner squares" are the corners of the central 3-by-3 square. The ``center square" is the square in the very center of the figure, and the ``side squares" are those squares which share an edge with the center square. Write $(\dagger)$ for the expression on the left side of \eqref{eq:ABCC-local-app}.

The most natural way to do the casework is by the highest-weight arrows contributing to $(\dagger)$.
\begin{itemize}
    \item If one of the arrows in the outer corner squares contributes to $(\dagger)$, the combination of squares
    whose labels are forced means none of the other arrows,
    except for the arrows in the two nearest other outer corner squares, one vertical arrow in one of the inner corner squares, and one horizontal arrow in another of the inner corner squares, may contribute to $(\dagger)$. If one of the arrows in one of the other corner squares contributes to $(\dagger)$, this is the only other arrow which may contribute to $(\dagger)$, and in this case $(\dagger)$ is exactly 12. Otherwise, the two remaining eligible arrows each have weight 3, so $(\dagger)$ is again at most 12.
    
    \item If none of the arrows with weight 6 contributes to $(\dagger)$, but at least one of the arrows with weight 4 contributes to $(\dagger)$:
    \begin{itemize}
        \item If another arrow with weight 4 contributes to $(\dagger)$, the two contributing arrows have to be in diagonally adjacent side squares. This rules out contributions from every other arrow except the four arrows with weight 1, at most two of which can contribute simultaneously, hence in this case $(\dagger)$ is at most 10.
        \item If there is only one contributing arrow with weight 4, this rules out contributions from every arrow except four of the arrows of weight 3 (the ones parallel to the arrow of weight 4) and the four arrows of weight 1. At most two of each can contribute simultaneously, so in this case $(\dagger)$ is at most 12.
    \end{itemize}
    \item If none of the arrows with weight 4 or 6 contribute to $(\dagger)$, but at least one of the arrows from the center square with weight 3 contributes to $(\dagger)$:
    \begin{itemize}
        \item If at least three arrows in the center square contribute to $(\dagger)$, this rules out all other contributions. In this case $(\dagger)$ is exactly 9 or 12.
        \item If exactly two arrows in the center square contribute to $(\dagger)$,
        \begin{itemize}
        \item if they form a right angle, that rules out all other contributions except for two other arrows of weight 3. These cannot both contribute simultaneously, hence in this case $(\dagger)$ is at most 9.
        \item if they form a 180-degree angle, that rules out all other contributions, hence in this case $(\dagger)$ is exactly 6.
        \end{itemize}
        \item If exactly one arrow in the center square contributes to $(\dagger)$, that rules out all other contributions, except for four of the arrows of weight 3, at most two of which may contribute to $(\dagger)$, hence in this case $(\dagger)$ is at most 9.
        \end{itemize}
    \item Finally, if no arrows in the center square contributes to $(\dagger)$, it is possible for at most four of the arrows of weight 3 in the inner corner squares to contribute to $(\dagger)$. If exactly four of the arrows of weight 3 in the inner corner squares contribute to $(\dagger)$, they must be the arrows at two diagonally opposing inner corner squares, and all other contributions are ruled out, hence in this case $(\dagger)$ is exactly 12.
    \end{itemize}
    
We have covered every possible case with at least four arrows of weight 3 or more. To finish, note that at most two of the vectors of weight 1 may simultaneously contribute to $(\dagger)$, hence in all remaining cases $(\dagger)$ is at most 11, as at most three arrows of weight 3 and two arrows of weight 1 can contribute to $(\dagger)$. In all cases $(\dagger)$ is at most 12, as desired.
\end{proof}

Lastly, we provide an upper bound on $C_2(\mathsf{ABBB})$, and thus a (computer-assisted) justification for the claim made in \cref{sec:conclusion} that $\mathsf{ABBB}$ is not $2$-slantable. In our attempts to answer \cref{qn:ABBB}, we have performed some computations with different heuristics to attempt to upper-bound $C_2(\mathsf{ABBB})$. The best bound we have been able to obtain is $C_2(\mathsf{ABBB})\leq \frac{59526}{35459}\approx 1.679$, which comes from the weight function described in \cref{fig:ABBB-upper}. This bound follows from \cref{prop:local}, with the relevant bound (ii) verified by computer.

\begin{figure}
\begin{tikzpicture}[scale=1.5]
    \localDiagram{1/1/0/1/12448/below left,
		  1/1/1/0/12448/below left,
		  1/5/0/-1/12448/above left,
		  1/5/1/0/12448/above left,
		  5/1/0/1/12448/below right,
		  5/1/-1/0/12448/below right,
		  5/5/0/-1/12448/above right,
		  5/5/-1/0/12448/above right,
		  1/1/1/1/5656/above right,
		  1/5/1/-1/5656/below right,
		  5/1/-1/1/5656/above left,
		  5/5/-1/-1/5656/below left,
		  1/2/0/1/11198/below left,
		  1/4/0/-1/11198/above left,
		  5/2/0/1/11198/below right,
		  5/4/0/-1/11198/above right,
		  2/1/1/0/11198/below left,
		  4/1/-1/0/11198/below right,
		  2/5/1/0/11198/above left,
		  4/5/-1/0/11198/above right,
		  1/2/1/0/13432/above,
		  1/4/1/0/13432/below,
		  5/2/-1/0/13432/above,
		  5/4/-1/0/13432/below,
		  2/1/0/1/13432/above,
		  4/1/0/1/13432/above,
		  2/5/0/-1/13432/below,
		  4/5/0/-1/13432/below,
		  1/2/1/1/18025/above,
		  5/2/-1/1/18025/above,
		  1/4/1/-1/18025/below,
		  5/4/-1/-1/18025/below,
		  2/1/1/1/18025/above,
		  2/5/1/-1/18025/below,
		  4/1/-1/1/18025/above,
		  4/5/-1/-1/18025/below,
		  1/2/1/-1/1079/below left,
		  1/4/1/1/1079/above left,
		  5/2/-1/-1/1079/below right,
		  5/4/-1/1/1079/above right,
		  2/1/-1/1/1079/below left,
		  4/1/1/1/1079/below right,
		  2/5/-1/-1/1079/above left,
		  4/5/1/-1/1079/above right,
		  1/3/0/1/2682/left,
		  1/3/0/-1/2682/left,
		  5/3/0/1/2682/right,
		  5/3/0/-1/2682/right,
		  3/1/1/0/2682/below,
		  3/1/-1/0/2682/below,
		  3/5/1/0/2682/above,
		  3/5/-1/0/2682/above,
		  1/3/1/0/3600/above,
		  5/3/-1/0/3600/above,
		  3/1/0/1/3600/above,
		  3/5/0/-1/3600/below,
		  1/3/1/-1/7086/below right,
		  1/3/1/1/7086/above right,
		  5/3/-1/-1/7086/below left,
		  5/3/-1/1/7086/above left,
		  3/1/-1/1/7086/above left,
		  3/1/1/1/7086/above right,
		  3/5/-1/-1/7086/below left,
		  3/5/1/-1/7086/below right,
		  2/2/1/0/17472/below left,
		  2/2/0/1/17472/left,
		  2/4/1/0/17472/above left,
		  2/4/0/-1/17472/left,
		  4/2/-1/0/17472/below right,
		  4/2/0/1/17472/right,
		  4/4/-1/0/17472/above right,
		  4/4/0/-1/17472/right,
		  2/2/1/1/20490/above,
		  4/4/-1/-1/20490/below,
		  2/4/1/-1/20490/below,
		  4/2/-1/1/20490/above,
		  2/2/-1/-1/13400/below,
		  4/4/1/1/13400/above,
		  2/4/-1/1/13400/above,
		  4/2/1/-1/13400/below,
		  2/3/0/-1/6358/below,
		  4/3/0/-1/6358/below,
		  2/3/0/1/6358/above,
		  4/3/0/1/6358/above,
		  3/2/-1/0/6358/below,
		  3/2/1/0/6358/below,
		  3/4/-1/0/6358/above,
		  3/4/1/0/6358/above,
		  2/3/1/-1/11758/below,
		  2/3/1/1/11758/above,
		  4/3/-1/-1/11758/below,
		  4/3/-1/1/11758/above,
		  3/2/-1/1/11758/above,
		  3/2/1/1/11758/above,
		  3/4/-1/-1/11758/below,
		  3/4/1/-1/11758/below,
		  2/3/1/0/5400/above,
		  4/3/-1/0/5400/above,
		  3/2/0/1/5400/right,
		  3/4/0/-1/5400/right,
		  2/3/-1/-1/8129/above left,
		  2/3/-1/1/8129/below left,
		  4/3/1/-1/8129/above right,
		  4/3/1/1/8129/below right,
		  3/2/-1/-1/8129/below left,
		  3/2/1/-1/8129/below right,
		  3/4/-1/1/8129/above left,
		  3/4/1/1/8129/above right,
		  3/3/0/-1/5656/right,
		  3/3/0/1/5656/right,
		  3/3/-1/0/5656/below,
		  3/3/1/0/5656/below,
		  3/3/1/1/10136/above,
		  3/3/1/-1/10136/below,
		  3/3/-1/1/10136/above,
		  3/3/-1/-1/10136/below}
\end{tikzpicture}
\caption{The weight function used to upper-bound $C_2(\mathsf{ABBB})$.}
\label{fig:ABBB-upper}
\end{figure}
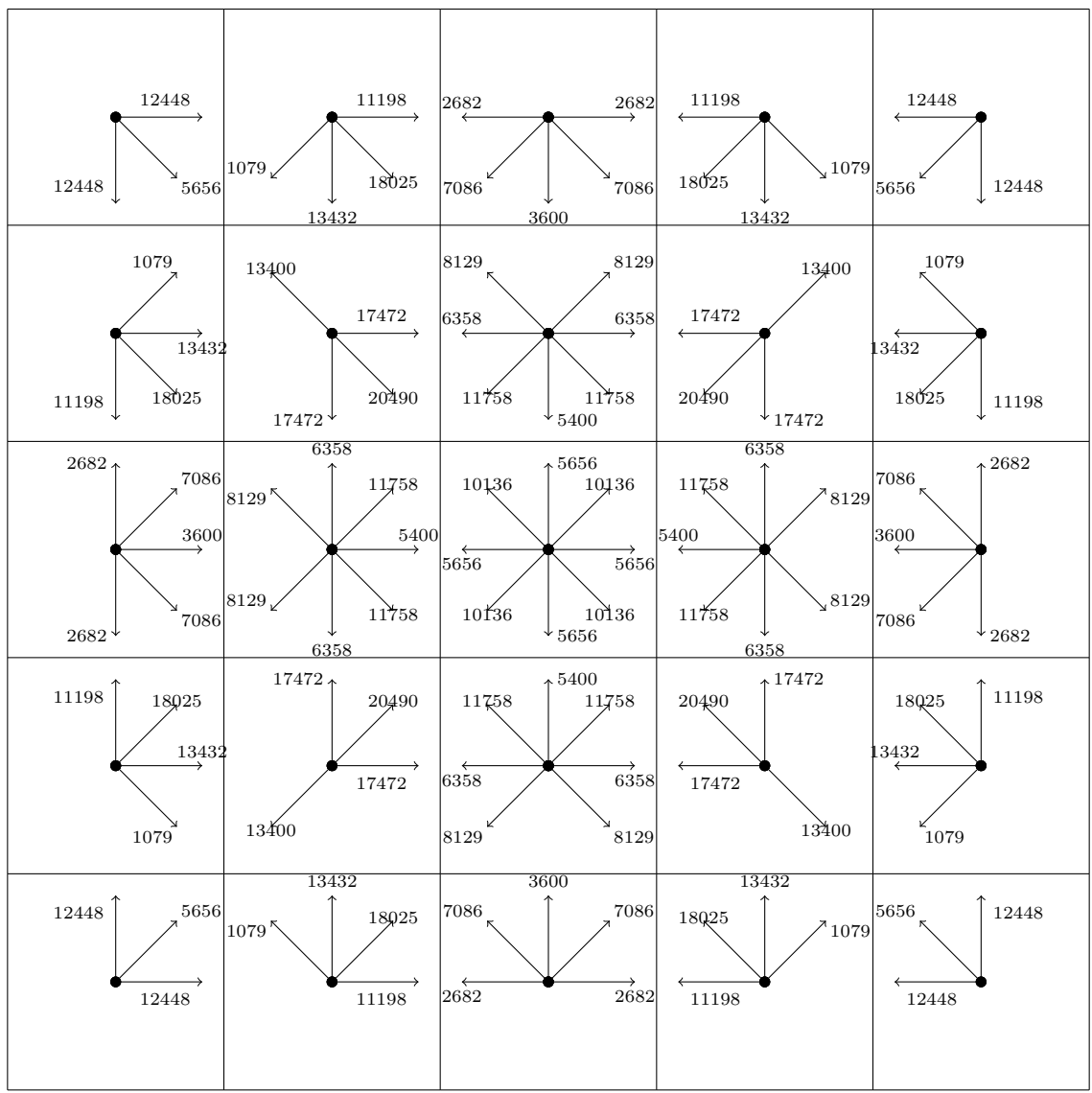

\end{document}